\documentclass[a4paper,11pt]{amsart}
\usepackage[dvips]{graphicx}
\usepackage{amssymb}
\usepackage{amsmath, mathtools}
\usepackage[utf8]{inputenc}
\usepackage[T1]{fontenc}
\usepackage{graphicx}
\usepackage{amsmath}
\usepackage{amsthm}
\usepackage{amsfonts}
\usepackage{amssymb}
\usepackage{enumerate}
\usepackage{vaucanson-g}
\usepackage{graphicx}
\usepackage{setspace}
\usepackage{mathrsfs}
\usepackage{mathabx}
\usepackage{enumitem}
\usepackage{dsfont}
\usepackage{mathrsfs}
\usepackage{amsmath,amstext, amsthm, amssymb, stmaryrd}
\usepackage{latexsym, amsthm, amsfonts, amsmath,amssymb}
\usepackage{bbm}
\usepackage{tikz-cd}
\usepackage{tikz}
\usetikzlibrary{arrows,automata,positioning}
\usepackage{xcolor}
\usepackage{ulem}
\usepackage{soul}
\usepackage{cancel}

\usepackage{dutchcal}

\makeatletter
\@namedef{subjclassname@2020}{%
  \textup{2020} Mathematics Subject Classification}
\makeatother

\definecolor{liens}{rgb}{1,0,0}
\usepackage[colorlinks=true, linkcolor=blue, 
hyperfootnotes=true,citecolor=blue,urlcolor=black]{hyperref}

\newcommand{\malop}[1]{\phi_{#1}}
\newcommand{\malopop}{\Phi_{\ellmahl}}
\newcommand{\ellmahl}{p}
\newcommand{\Hahn}{\mathscr{H}}
\newcommand{\val}{\operatorname{val}_{z}}

\newcommand{\cond}[1]{(\mathcal{O}\Omega_{#1})}
\newcommand{\condprime}[1]{(\mathcal{O}_{#1})}
\newcommand{\condpuisprime}[1]{(\puis-\mathcal{O}_{#1})}

\newcommand{\denommahl}[1]{\mathfrak d_{#1}}
\newcommand{\dfrak}{\mathfrak d}
\newcommand{\efrak}{\mathfrak e}
\newcommand{\cj}{\Xi}

\newcommand{\ala}{{\boldsymbol \Lambda}}
\newcommand{\alat}[1]{\boldsymbol \Lambda_{#1}}
\newcommand{\alaNp}{\ala_{\rm st}} 
\newcommand{\alaNpt}[1]{\ala_{{\rm st},#1}} 
\newcommand{\alaZ}{\Lambda_{\Z}}
\newcommand{\alaZt}[1]{\Lambda_{\Z,#1}}
\newcommand{\tuple}{\omeg}

\newcommand{\starpourexp}{\infty}
\newcommand{\puisinfzero}{\Qbar[z^{-\frac{1}{\starpourexp}}]}
\newcommand{\polyraminfzero}{\Qbar[z^{-\frac{1}{\starpourexp}}]}
\newcommand{\polyramstrictinfzero}{\Qbar[z^{-\frac{1}{\starpourexp}}] \setminus \Qbar}



\newcommand{\dpuis}{\mathcal{D}_{\puis}}

\newtheorem{theo}{Theorem}[]
\newtheorem{prop}[theo]{Proposition}
\newtheorem{lem}[theo]{Lemma}

\newtheorem{defi}[theo]{Definition}

\newtheorem{rem}[theo]{Remark}
\theoremstyle{definition}

\newtheorem{notation}[theo]{Notation}
\newtheorem*{nota*}{Notation}
\newtheorem*{defi*}{Definition}

\newcommand{\hgt}{\mathcal{h}}

\newcommand{\QQ}{\mathbb Q}
\newcommand{\Q}{{\overline{\mathbb Q}}}
\newcommand{\Qbar}{{\overline{\mathbb Q}}}
\newcommand{\V}{\mathcal V}
\renewcommand{\r}{\boldsymbol r}

\newcommand{\C}{\mathbb{C}}
\newcommand{\ZZ}{\mathbb{Z}}
\newcommand{\Z}{\mathbb{Z}}

\newcommand{\N}{\mathbb N}
\newcommand{\Rspe}{\mathcal R}

\newcommand{\K}{\mathcal K}
\newcommand{\Kinf}{\mathbb K_{\infty}}

\newcommand{\I}{\mathcal I}

\newcommand{\E}{\mathcal E}

\newcommand{\lambd}{{\boldsymbol{\lambda}}}

\newcommand{\balpha}{{\boldsymbol{\alpha}}}

\newcommand{\omeg}{{\boldsymbol{\omega}}}

\renewcommand{\v}{{\boldsymbol v }}

\renewcommand{\a}{{\boldsymbol{a}}}
\renewcommand{\b}{{\boldsymbol{b}}}

\newcommand{\Mat}{{\operatorname{M}}}
\newcommand{\RR}{{\mathbb{R}}}
\newcommand{\mS}{\mathcal S}

\newcommand{\logm}{\ell}

\newcommand{\indC}{r}

\newcommand\puis{\mathscr{P}}
\newcommand\bigzero{\makebox(0,0){\text{\huge0}}}
\newcommand{\bord}{}
\newcommand{\cld}{\operatorname{cld}_{z}}

\newcommand{\tuplespt}{\widecheck{\tuple}}
\newcommand{\balphaspt}{\widecheck{\balpha}}
\newcommand{\lambdspt}{\widecheck{\lambd}}
\newcommand{\aspt}{\widecheck{\a}}
\newcommand{\nbblocks}{\sigma}

\DeclareMathOperator{\GL}{GL}

\DeclareMathOperator{\supp}{supp}

\newcommand{\NGA}{\operatorname{NGA}}

\newcommand{\R}{\mathcal{R}}

\title{A purity theorem for Mahler equations}
\author{C.~Faverjon}
\address{Universit\'e Picardie Jules Vernes, UMR CNRS 7352, LAMFA, 80039 Amiens, France}
\email{colin.faverjon@math.cnrs.fr}
\author{J.~Roques }
\address{Universit\'e de Lyon, Universit\'e Claude Bernard Lyon 1, CNRS UMR 5208, Institut Camille Jordan, F-69622 Villeurbanne, France}
\email{Julien.Roques@univ-lyon1.fr}
\date{\today}
\subjclass[2020]{Primary 39A06.}

\begin{document}
	
	\begin{abstract}
		The principal aim of this paper is to establish a purity theorem for Mahler functions that is reminiscent of famous purity theorems for $G$-functions by D. and G.~Chudnovsky and for $E$-functions (and, more generally, for holonomic arithmetic Gevrey series) by Y.~André. Our approach is based on a preliminary study of independent interest of the nature of the solutions of Mahler equations. 
		Roughly speaking, we prove a reduction result for Mahler systems, implying that any Mahler equation admits a complete basis of solutions formed of what we call generalized Mahler series. These are sums involving Puiseux series, Hahn series of a very special type and solutions of inhomogeneous  equations of order $1$  with constant coefficients. 
		In the light of B.~Adamczewski, J.~P.~Bell and D.~Smertnig's recent height gap theorem, we introduce a natural filtration on the set of generalized Mahler series according to the arithmetic growth of the coefficients of the Puiseux series involved in their decomposition. This filtration has five pieces. Our purity theorem states that the membership of a generalized Mahler series to one of the three largest pieces of this filtration propagates to any other generalized Mahler series solution of its minimal Mahler equation. We also show that this statement does not extend to the smallest two pieces.
	\end{abstract}
\maketitle

\setcounter{tocdepth}{1}
\tableofcontents

\vskip 10 pt 

\section{Introduction and statement of the main results}

The primary aim of this article is to establish a purity theorem for Mahler functions that is reminiscent of famous purity theorems for $G$- and $E$-func\-tions and, more generally, for holonomic arithmetic Gevrey  series by D. and G.~Chudnovsky and Y.~André  respectively. Although there is no concrete link between the latter results and those of the present article, 
they played a fundamental role in the genesis of our work, so we start by briefly recalling them.

Following Y.~André in \cite{AndreSGA1}, we say that a power series 
$$
f=\sum_{n=0}^{\infty} a_n z^n \in \Qbar[[z]]
$$ 
with coefficients in the field of algebraic numbers $\Qbar$ is an arithmetic Gevrey series of order $s \in \QQ$ if   there exists $C>0$ such that
\begin{enumerate} 
\item[-] for all $n \in \Z_{\geq 0}$, the maximum of the moduli of the Galoisian conjuguates of $b_{n}:=\frac{a_n}{n!^{s}}$ is bounded by $C^{n+1}$;
\item[-] there exists a sequence of positive integers $(d_n)_{n\geq 0}$ such that, for all $n \in \Z_{\geq 0}$, $d_n \le C^{n+1}$ and $d_n b_0,d_n b_1,\ldots,d_n b_n$ are algebraic integers.
\end{enumerate}
An holonomic  arithmetic Gevrey series of order $0$ (resp.~$-1$) is nothing but a $G$-function (resp.~$E$-function) in the sense of C.~L.~Siegel \cite{Siegel29,Siegel29-trad}. By holonomic, we mean solution of a nonzero linear differential equation with coefficients in $\Qbar(z)$.

To state the purity theorem for these series, it is convenient to introduce the differential $\C[z]$-algebra $\NGA\{z\}_{s}$ of arithmetic Nilsson-Gevrey series of order $s$. An element of $\NGA\{z\}_{s}$ 
is by definition a $\C$-linear combination of terms of the form
$$
u(z) z^{\alpha} \log^{j}(z)
$$
where $\alpha \in \QQ$, $j \in \Z_{\geq 0}$, and $u(z)$ is an arithmetic Gevrey series of order $s$. 

The following fundamental purity theorem is due to D. and G.~Chudnovsky for $s=0$ in \cite{ChudnovskyAPADIVGF} (the proof contains a slight mistake corrected by Y.~André in \cite[Chapter VI]{AndreGFAG}) and to Y.~André for $s \neq 0$ in \cite{AndreSGA1} (see also \cite{AndreArithGevreySeriesSurvey}). 

\begin{theo}[Y.~André, D. and G.~Chudnovsky]\label{theo:purity theo andre chudnovsky}
Let $y$ be an element of $\NGA\{z\}_{s}$ solution of a nonzero linear differential equation $\Psi y = 0$ with coefficients in $\C(z)$. We assume $\Psi$ of minimal order. Then:
\begin{enumerate}
\item if $s \leq 0$, $\Psi$ admits a full basis of solutions in $\NGA\{z\}_{s}$;
\item if $s > 0$, $\Psi$ admits a full basis of solutions of the form $e^{\alpha_{i} z^{-\frac{1}{s}}} y_{i}$ with 
$y_{i} \in \NGA\{z\}_{s}$ and $\alpha_{i} \in \Qbar$. 
\end{enumerate}
\end{theo}

\begin{rem}
Any linear differential equation of order $d$ with coefficients in $\C(z)$ -- or, more generally, in $\C(\{z\})$ -- has a basis of solutions made of $\C$-linear combinations of terms of the form
$$
u(z) z^{\alpha} \log^{j}(z) e^{Q(z^{-1/d!})}
$$
where $\alpha \in \C$, $j \in \Z_{\geq 0}$, $u(z)$ is a Gevrey series and $Q(X) \in X\C[X]$. Theorem~\ref{theo:purity theo andre chudnovsky} shows in particular that the fact that the differential operator under consideration is a nonzero differential operator of minimal order annihilating an holonomic arithmetic Gevrey series imposes severe restrictions on the $\alpha$, $u(z)$ and $Q(X)$ involved in its basis of solutions.  
\end{rem}

We now come to the subject of our study, the Mahler equations. 
Let $\ellmahl \geq 2$ be an integer. By $\ellmahl$-Mahler equation, we mean a linear functional equation of the form   
\begin{equation}\label{eq: mahler1} 
 a_0(z) f(z) + a_1(z) f(z^{\ellmahl})+ \cdots + a_d(z) f(z^{\ellmahl^d}) = 0 
\end{equation}
with coefficients $a_{0},\ldots,a_d$ such that $a_0a_d\neq 0$ in the field 
$$
\Kinf= \Q(z^{\frac{1}{\starpourexp}}) = \bigcup_{k\in \Z_{\geq 1}} \Q(z^{\frac{1}{k}})
$$ 
of ramified rational functions. A solution of such an equation will be called a $\ellmahl$-Mahler function. If such a solution is a power series (resp.~a Puiseux series, a Hahn series), we will say that it is a $\ellmahl$-Mahler power series (resp.~a $\ellmahl$-Mahler Puiseux series, a $\ellmahl$-Mahler Hahn series).

Although the solutions of Mahler equations are very different in nature from those of differential equations, certain properties bring them close to arithmetic Gevrey series. Indeed, on the one hand, it is well-known that any $\ellmahl$-Mahler series $f \in \Qbar[[z]]$ is an arithmetic Gevrey series of order $0$ \cite[Chapitre 3, Corollaire 8 and Théorème 6]{DumasThese} (but, be careful, if it is not rational, then $f$ is not holonomic  \cite{RandeThese,BCR13}, and isn't even differentially algebraic over $\C(z)$ \cite{ADHHLDE}). On the other hand, the famous refinement of the Siegel-Shidlovskii theorem due to F.~Beukers in \cite{Beukers06} and reproved by Y.~André in \cite{Andre14} admits a Mahlerian analogue  proved by P.~Philippon \cite{PhilipponGaloisNA} and supplemented
 by B.~Adamczewski and the first author in \cite{BorisFaverjonMethodeMahler17}, which brings $\ellmahl$-Mahler functions closer to $E$-functions. Note that another proof of Philippon's result, in the spirit of \cite{Andre14}, was subsequently given by L.~Nagy and T.~Szamuely in \cite{NagySzamuelyAGTAndreSolAlg} and that a third proof was given by B.~Adamczewski and the first author in \cite{AF23}.   It is properties like these that have encouraged us to investigate a possible extension of the above purity theorem to Mahler equations.
 
A natural motivation for looking at the growth properties of the coefficients of $\ellmahl$-Mahler series or, more generally, of $\ellmahl$-Mahler Hahn series comes from the Bombieri-Dwork conjecture predicting that the minimal differential equation of a $G$-function comes from geometry. In the light of this conjecture, it is natural to ask whether a $\ellmahl$-Mahler Hahn series whose coefficients have a special growth has a special nature. For results in this direction, we refer to our forthcoming paper~\cite{FaverjonRoquesHGT}.

The theory of Mahler equations is a dynamic and fast-growing research area.  
Since the pioneering work of Mahler in \cite{MahlerArith1929, MahlerArith1930, MahlerUber1930}, numerous articles have been written on Mahler equations, which has known important recent developments; see for instance \cite{pellarinAIMM,NG,NGT,NishiokaNishiokaATDEMF,BCR13,BBC15,PhilipponGaloisNA,BCZ16,BorisAboutMahler,BorisFaverjonMethodeMahler17,BorisFaverjonMethodeMahler,CompSolMahlEq,DHRMahler,Ro15,BorisMahlerSelecta,BeckerConj,SchafkeSinger,ADHHLDE,RoquesLSMS,ArrecheZhangMahlerDiscreteResidues,medvedev2022skewinvariantcurvesalgebraicindependence,BorisBellSmertnigGap,AF23,PouletDensity,AF24_SevVar,AF24_EM,FaverjonRoquesHahnMahlAlgoAspects,FaverjonPoulet2} and the references therein, to name but a few recent articles. We believe the results presented in this paper will be useful for further work on Mahler equations.

The main results of the present paper are described in Sections~\ref{sec:solving mahler equations intro} and~\ref{section:purity theo intro} below. 
In Section~\ref{sec:solving mahler equations intro}, we outline our main results about the structure of the solutions of $\ellmahl$-Mahler equations at $0$. Theorem~\ref{thm:decomp_chi}, which is the outcome of this study, shows that any $\ellmahl$-Mahler equation has a full basis of solutions consisting of what we call generalized  $\ellmahl$-Mahler series. This result, of independent interest, is a necessary prerequisite for the statement and proof of our purity theorem, to which Section~\ref{section:purity theo intro} is devoted.

\subsection{Solving $\ellmahl$-Mahler equations}\label{sec:solving mahler equations intro}

\subsubsection{Hahn series and $\ellmahl$-Mahler equations}\label{sec:intro hahn and mahler}
Hahn series are a key ingredient for solving $\ellmahl$-Mahler equations of the form \eqref{eq: mahler1}. We let $\Hahn=\Qbar((z^{\QQ}))$ be the field of Hahn series with coefficients in $\Qbar$ and with value group $\QQ$. This field contains the field 
$$\puis=\Qbar((z^{\frac{1}{\starpourexp}}))=\bigcup_{k \in \Z_{\geq 1}} \Qbar((z^{\frac{1}{k}}))
$$ 
of Puiseux series as a subfield but it is much bigger. Roughly speaking,  Hahn series are a generalization of Puiseux series allowing arbitrary exponents of the indeterminate as long as the set that supports them is well-ordered; we refer to Section~\ref{sec hahn} for details. The interest of Hahn series in our context lies in the following result: 
the difference field $(\Hahn,\malop{\ellmahl})$, where $\malop{\ellmahl}$ is the field automorphism of $\Hahn$ sending $f (z)$ on $f(z^{\ellmahl})$, has a difference ring extension $(\mathcal{R},\malop{\ellmahl})$ with field of constants $\mathcal{R}^{\malop{\ellmahl}}=\{f \in \mathcal{R} \ \vert \  \malop{\ellmahl}(f)=f\}$ equal to $\Qbar$ such that
\begin{itemize}
 \item for any $c \in \Qbar^{\times}$, there exists an invertible element $e_{c} \in \mathcal{R}$ satisfying $\malop{\ellmahl}(e_{c})=ce_{c}$;
\item there exists $\logm \in \mathcal{R}$ satisfying $\malop{\ellmahl}(\logm)=\logm+1$; 
\item any $\ellmahl$-Mahler equation of the form \eqref{eq: mahler1} has a full basis\footnote{We say that a $\ellmahl$-Mahler equation of order $d$  has a  ``full basis'' of solutions of a given form, if it has $d$ solutions of the given form that are $\Qbar$-linearly independent.} of solutions $y_{1},\ldots,y_{d} \in \mathcal{R}$ of the form 
\begin{equation}\label{form yi intro}
 y_{i} = \sum_{(c,j) \in \cj}  f_{i,c,j} e_{c} \logm^{j}\end{equation}
 where the sum has finite support contained in 
 $$
 \cj=\Qbar^{\times}\times\Z_{\geq 0}
 $$
and where each $f_{i,c,j}$ belongs to $\Hahn$.
\end{itemize}
See for instance~\cite[Lemma~23]{RoquesFrobForMahler}. 

\subsubsection{Generalized $\ellmahl$-Mahler series and $\ellmahl$-Mahler equations}\label{sec:intro gpms} The first main result of this paper -- namely Theorem~\ref{thm:decomp_chi} below -- ensures that the Hahn series $f_{i,c,j}$ involved in \eqref{form yi intro} have a very special form: 
they are linear combinations with coefficients in the ring of $\ellmahl$-Mahler Puiseux series of a family of Hahn series $(\xi_{\tuple})_{\tuple \in \ala}$ defined as follows.  The set $\ala$ is given by 
$$
\ala= \bigcup_{t \in \Z_{\geq 0}} \alat{t}
\text{ \ \ \ where \ \ \ }
\alat{t}= \Z_{\geq 0}^{t} \times (\Qbar^\times)^{t} \times \QQ_{> 0}^{t}.
$$
\begin{notation}\label{notation:omega, alpha, etc}
Throughout this paper, unless stated otherwise, we let $\tuple$ be an element of $\ala$. We will set $\tuple=(\balpha,\lambd,\a)$ with $\balpha=(\alpha_1,\ldots,\alpha_t) \in \Z_{\geq 0}^{t}$, $\lambd=(\lambda_1,\ldots,\lambda_t) \in (\Qbar^\times)^{t}$ and $\a=(a_1,\ldots,a_t) \in \QQ_{> 0}^{t}$ where $t$ is the unique element of $\Z_{\geq 0}$ such that $\tuple \in \alat{t}$.
When several elements of $\ala$ are considered simultaneously, we will use variants such as  $\tuple'=(\balpha',\lambd',\a')$ with $\balpha'=(\alpha'_1,\ldots,\alpha'_{t'}) \in \Z_{\geq 0}^{t'}$, $\lambd'=(\lambda'_1,\ldots,\lambda'_{t'}) \in (\Qbar^\times)^{t'}$ and $\a'=(a'_1,\ldots,a'_{t'}) \in \QQ_{> 0}^{t'}$. 
\end{notation}
If $t \in \Z_{\geq 1}$, we consider the Hahn series defined by
\begin{multline}\label{eq:def_chi intro}
\xi_{\tuple}(z) 
=\\ 
\sum_{k_1,\ldots,k_t \geq 1} k_{1}^{\alpha_{1}} \cdots k_{t}^{\alpha_{t}} \lambda_1^{k_1}\lambda_2^{k_1+k_2}\cdots\lambda_t^{k_1+\cdots+k_t} z^{-\frac{a_1}{p^{k_1}} - \frac{a_2}{p^{k_1+k_2}}  - \cdots - \frac{a_t}{p^{k_1+k_2+\cdots + k_t}}} \in \Hahn.
\end{multline}
 When $t=0$, the sets $\Z_{\geq 0}^{t}$, $(\Qbar^\times)^{t}$ and $\QQ_{> 0}^{t}$ have just one element, namely the empty vector $()$, so  $\alat{0}=\{((),(),())\}$ and we set
$$
\xi_{((),(),())}(z)= 1.
$$ 
The fact that the $\xi_{\tuple}$ are indeed well-defined Hahn series is established in Section~\ref{sec: xi well def hahn series}.

\begin{defi}\label{defi:gen_mahl_series}
 A generalized $\ellmahl$-Mahler series is an element of $\mathcal R$ of the form 
\begin{equation}\label{eq:generalized_mahler_series_bis}
\sum_{ (c,j) \in  \cj}\sum_{\tuple\in \ala} 
f_{c,j,\tuple} \xi_{\tuple}e_{c} \logm^{j} 
\end{equation}
where the sums have finite support and where the $f_{c,j,\tuple} \in \puis$ are $\ellmahl$-Mahler Puiseux series. 
\end{defi}
\begin{theo}\label{thm:decomp_chi}
Any $\ellmahl$-Mahler equation of the form \eqref{eq: mahler1} has a full basis of generalized $\ellmahl$-Mahler series solutions, {\it i.e.},   it has $d$ generalized $\ellmahl$-Mahler series solutions $y_{1},\ldots,y_{d} \in \mathcal R$ that are $\Qbar$-linearly independent. 
\end{theo}

We will obtain this result as a by-product of the construction of fundamental matrices of solutions of $\ellmahl$-Mahler systems with a very specific form, which are reminiscent of the fundamental matrices of solutions of differential systems given by Turrittin's theorem; as this requires further notations, we say no more about this result of independent interest here and refer the reader to Section~\ref{sec: fund mat sol} and, especially, to Theorem~\ref{theo: fund syst sol princ}. We point out that Section~\ref{subsec:rks proof theo fund syst sol princ} provides a discussion of the similarities and differences between $\ellmahl$-Mahler systems and both differential and $q$-difference systems.

\begin{rem}
	${\rm (a)}$ We will prove in Section~\ref{sec: the xi are p mahler} that the Hahn series $\xi_{\tuple}$ are $p$-Mahler. Since the property of being a solution of a $p$-Mahler equation is stable under sums and products, we obtain that any generalized $p$-Mahler series is a solution of a $p$-Mahler equation.

${\rm (b)}$ The family $(\xi_{\tuple})_{\tuple \in \ala}$ is $\puis$-linearly dependent and, hence, the decomposition \eqref{eq:generalized_mahler_series_bis} is not unique. We address this issue in Section~\ref{sec:standard}. 
\end{rem}

\subsection{Purity Theorem} \label{section:purity theo intro}

Our purity theorem (Theorem~\ref{thm:pureté1} below)
involves growth conditions for generalized $p$-Mahler series inspired by the recent paper \cite{BorisBellSmertnigGap} by B.~Adamczewski, J.~P.~Bell and D.~Smertnig. Let us briefly recall their main result. 

\subsubsection{Growth of the coefficients of $\ellmahl$-Mahler power series}
In \cite{BorisBellSmertnigGap}, the authors study the asymptotic growth of the coefficients of $\ellmahl$-Mahler power series with coefficients in $\Qbar$, as measured by their logarithmic Weil height. Their main result is the following height gap theorem, which shows that there are five different growth behaviors.

\begin{theo}[{\cite[Proposition~5.2]{BorisBellSmertnigGap}}]\label{thm: hgt}
Any $\ellmahl$-Mahler Puiseux series
$f=\sum_{\gamma \in \QQ}f_\gamma z^{\gamma} \in \puis$ satisfies one of the following mutually exclusive properties\footnote{Strictly speaking, this result is only proved for power series and for $\ellmahl$-Mahler equations with coefficients in $\Qbar(z)$, but the extension to Puiseux series and to $\ellmahl$-Mahler equations with coefficients in $\Kinf$ is straightforward.
}:
	\begin{enumerate}
	\item[$\cond{1}$] $\hgt(f_\gamma) \in \mathcal O\cap\Omega(H(\gamma))$;
	
	\item[$\cond{2}$] $\hgt(f_\gamma) \in \mathcal O\cap\Omega(\log^2H(\gamma))$;
	
	\item[$\cond{3}$] $\hgt(f_\gamma) \in \mathcal O\cap\Omega(\log H(\gamma))$;
	
	\item[$\cond{4}$] $\hgt(f_\gamma)  \in \mathcal O\cap\Omega(\log\log H(\gamma))$;
	
	\item[$\cond{5}$] $\hgt(f_\gamma) \in \mathcal O(1)$.
\end{enumerate}
\end{theo}

In this result and throughout this paper, $H(\alpha)$ denotes the Weil height of $\alpha \in \Qbar$ and $\hgt(\alpha)=\log H(\alpha)$ its logarithmic Weil height 
(see \cite{Wald2000} for details and references). Roughly speaking, they measure the ``complexity'' of the algebraic number $\alpha$. For instance, when $\gamma = \frac{a}{b}$ is a rational number, with $a,b \in \Z$ relatively prime, $H(\gamma)=\max\{\vert a \vert, \vert b \vert\}$.
Moreover, for any $(a_\gamma)_{\gamma \in \QQ},(b_\gamma)_{\gamma \in \QQ} \in \RR^{\QQ}$, the notation 
$
a_\gamma \in \mathcal O(b_\gamma)
$
means that there exists $C>0$ such that, for all but finitely many $\gamma \in \QQ$, we have 
$
\vert a_\gamma \vert \leq C \vert b_\gamma \vert
$ and the notation 
$
a_\gamma\in\Omega(b_\gamma)
$
means that there exists $c>0$ such that, for infinitely many $\gamma \in \QQ$, we have $\vert a_\gamma \vert > c \vert b_\gamma \vert$. The notation $a_\gamma \in \mathcal O \cap \Omega(b_\gamma)$ means that $a_\gamma \in \mathcal O(b_\gamma)$ and $a_\gamma \in \Omega(b_\gamma)$.

\subsubsection{Purity theorem}

Theorem~\ref{thm: hgt} reveals five $\mathcal O$-growth conditions for Puiseux series: 
 we say that $f=\sum_{\gamma} f_\gamma z^{\gamma} \in \puis$ satisfies
	\begin{itemize}
	  \item $\condprime{1}$\label{tms:generic} if $\hgt(f_\gamma) \in \mathcal O(H(\gamma))$;
	
	\item $\condprime{2}$\label{tms:hn} if $\hgt(f_\gamma) \in \mathcal O(\log^2H(\gamma))$;
	
	\item $\condprime{3}$\label{tms:hlog2} if $\hgt(f_\gamma) \in \mathcal O(\log H(\gamma))$;	
	\item $\condprime{4}$\label{tms:hlog} if $\hgt(f_\gamma) \in \mathcal O(\log\log H(\gamma))$;
	
	\item $\condprime{5}$\label{tms:hloglog} if $\hgt(f_\gamma) \in \mathcal O(1)$.
	\end{itemize}

We extend these $\mathcal O$-growth conditions to generalized $\ellmahl$-Mahler series as follows. 

\begin{defi}	\label{defi:def Cr gms}
We say that a generalized $\ellmahl$-Mahler series $f$ satisfies $\condpuisprime{\indC}$ for some $\indC \in \{1,2,3,4,5\}$ if it admits a decomposition of the form \eqref{eq:generalized_mahler_series_bis} such that all the Puiseux series $f_{c,j,\tuple}$ satisfy $\condprime{\indC}$. 
\end{defi}

It is clear that, for any $r \in \{1,\ldots,4\}$, the condition $\condpuisprime{r+1}$ is stronger than $\condpuisprime{r}$ in the sense that any generalized $\ellmahl$-Mahler series satisfying $\condpuisprime{r+1}$ also satisfies $\condpuisprime{r}$. Moreover, it follows from Definition~\ref{defi:gen_mahl_series} and Theorem~\ref{thm: hgt} that any generalized $\ellmahl$-Mahler series satisfies $\condpuisprime{1}$. 
Therefore, the five growth conditions $\condpuisprime{1}$ to $\condpuisprime{5}$ induce the following filtration on the set of generalized $\ellmahl$-Mahler series: 
\begin{align*}
 &\{\text{generalized $\ellmahl$-Mahler series}\} \\
 &\hskip 20 pt =  \{\text{generalized $\ellmahl$-Mahler series satisfying $\condpuisprime{1}$}\} \\
 &\hskip 40 pt  \supsetneq   \{\text{generalized $\ellmahl$-Mahler series satisfying $\condpuisprime{2}$}\} \\
  &\hskip 60 pt \supsetneq   \{\text{generalized $\ellmahl$-Mahler series satisfying $\condpuisprime{3}$}\} \\
   &\hskip 80 pt \supsetneq   \{\text{generalized $\ellmahl$-Mahler series satisfying $\condpuisprime{4}$}\} \\
    &\hskip 100 pt \supsetneq   \{\text{generalized $\ellmahl$-Mahler series satisfying $\condpuisprime{5}$}\}. 
\end{align*}
This filtration has five pieces. We are now ready to state our purity theorem guaranteeing that the membership of a generalized $\ellmahl$-Mahler series to one of the three largest pieces of this filtration propagates to any other generalized Mahler series solution of its minimal Mahler equation.

\begin{theo}[Purity Theorem]\label{thm:pureté1}
	Let $f$ be a generalized $\ellmahl$-Mahler series satisfying $\condpuisprime{\indC}$ for some $\indC \in \{1,2,3\}$. Then, the minimal $\ellmahl$-Mahler equation of $f$ over $\Kinf$ has a full basis of generalized $\ellmahl$-Mahler series solutions satisfying $\condpuisprime{\indC}$.
\end{theo}

\begin{rem}
		{\rm (i)} Theorem~\ref{thm:pureté1} does not extend to $r \in \{4,5\}$, cf. Section~\ref{sec:final_rem}.
		
	{\rm (ii)} Considering the \textit{minimal} $\ellmahl$-Mahler equation is of course essential for the conclusions of Theorem~\ref{thm:pureté1} to hold. The constant function $1$ satisfies $\condpuisprime{3}$ and is solution of the equation
	\begin{equation}\label{eq:non_mini}
	(1-2z)y(z) + (-1+2z-z^2+3z^3-3z^4)y(z^2) +(z^2-3z^3+3z^4)y(z^4)=0
	\end{equation}
%
which is not minimal with respect to $1$. This equation also has a solution
	\begin{equation}\label{eq:ex_Laurent}
	-z^{-1}+3z+6z^2+6z^3+21z^4+21z^5+60z^6 +99z^7+234z^8+\mathcal O(z^{9})
	\end{equation}
	which does not satisfy $\condpuisprime{3}$, as one could prove. 
	
	{\rm (iii)} Instead of considering the generalized $\ellmahl$-Mahler series as $\puis$-linear combinations of the $\xi_{\tuple} e_{c} \logm^{j}$, we could view them as $\Hahn$-linear combinations of the $e_{c} \logm^{j}$, as in \eqref{form yi intro}. From this perspective, it is natural to study the arithmetic properties of the coefficients of the Hahn series involved in such linear combinations. This is the subject of the forthcoming paper \cite{FaverjonRoquesHGT}.
	\end{rem}

\subsection{Organization of the paper}
In Section~\ref{sec hahn}, we first recall the definition of the field of Hahn series. We then prove that the $\xi_{\tuple}(z)$ introduced in Section~\ref{sec:intro gpms} are well-defined Hahn series, that they can be constructed inductively and that they are $\ellmahl$-Mahler. We also introduce a filtered algebra associated with them, which will turn out to play an important role in the remainder of the paper. 
In Section~\ref{trad syst equ mod}, we recall the dictionary between Mahler equations, systems and modules, that will be useful for the next sections. Section~\ref{sec: fund mat sol} is mainly devoted to the construction of fundamental matrices of solutions at $0$ for $\ellmahl$-Mahler systems. The main result of this section, Theorem~\ref{theo: fund syst sol princ}, forms the cornerstone of the present paper. Theorem~\ref{thm:decomp_chi} is deduced from Theorem~\ref{theo: fund syst sol princ} at the end of Section~\ref{sec: fund mat sol}. In Section~\ref{sec:standard}, we introduce the notion of standard decomposition for generalized $\ellmahl$-Mahler series, which aims to address the non-uniqueness of the decomposition \eqref{eq:generalized_mahler_series_bis}. The remainder of the paper is mainly dedicated to establishing Theorem~\ref{thm:pureté1}. We begin in Section~\ref{sec:sketch proof purity theo} with an outline of the proof, which also serves to motivate the material developed in the subsequent sections. Section~\ref{sec:mahler denom} reviews key aspects of the notion of $\ellmahl$-Mahler denominator for $\ellmahl$-Mahler power series, as introduced in \cite{BorisBellSmertnigGap}, and explores its extension to generalized $\ellmahl$-Mahler series. 
The result presented there will be instrumental to the proof of Theorem~\ref{thm:pureté1}, which is given in Section~\ref{sec:proof purity theo}. 
Finally, in Section~\ref{sec:final_rem}, we give an example showing that Theorem~\ref{thm:pureté1} cannot be extended to conditions $\condpuisprime{4}$ or $\condpuisprime{5}$.
\vskip 5pt 
\noindent {\bf Acknowledgements.} This work grew out of a question posed by Boris Adamczewski, and benefited from discussions we had with him throughout the genesis of this article. Our thanks go to him.  We are very grateful to the referees for their insightful and constructive reports, which helped us improve the quality and clarity of the manuscript. The work of the second author was supported by the ANR De rerum natura project, grant ANR-19-CE40-0018 of the French Agence Nationale de la Recherche.

\section{Hahn series} \label{sec hahn}

\subsection{Definitions and generalities} 
We recall that an element of the field of Hahn series $\Hahn$ is an $f=(f_{\gamma})_{\gamma \in \QQ} \in \Qbar^{\QQ}$ whose support 
$$
\supp(f)=\{\gamma\in \QQ \ \vert \ f_{\gamma}\neq 0\}
$$ 
is well-ordered ({\it i.e.}, any nonempty subset of $\supp(f)$ has a least element) with respect to the restriction to $\supp(f)$ of the usual order on $\QQ$. Such an element  of $
\Hahn
$
is usually (and will be) denoted by 
$$
f=\sum _{{\gamma \in \QQ }}f_{\gamma}z^{\gamma}. 
$$
The field of Hahn series contains the field $\puis$ of Puiseux series as a subfield but it is much bigger. A typical example of Hahn series which is not a Puiseux series is given by 
$$
\xi_{((0),(1),(1))}=\sum_{k \geq 1} z^{-\frac{1}{\ellmahl^{k}}}.
$$ 

We recall that the sum and product of two elements 
$f=\sum _{{\gamma\in \QQ }}f_{\gamma}z^{\gamma}$ and 
$g=\sum _{{\gamma\in \QQ }}g_{\gamma}z^{\gamma}$ of $\Hahn$ 
are given by
$$
f+g=\sum _{{\gamma\in \QQ }}(f_{\gamma}+g_{\gamma})z^{\gamma}
\text{\ \ and \ \ }
fg=\sum _{{\gamma\in \QQ }}\left(\sum _{{\gamma'+\gamma''=\gamma}}f_{{\gamma'}}g_{{\gamma''}}\right)z^{\gamma}.
$$
The fact that the supports of $f$ and $g$ are well-ordered implies that there are only finitely many $(\gamma',\gamma'') \in \QQ \times \QQ$ such that $\gamma'+\gamma''=\gamma$ and $f_{{\gamma'}}g_{{\gamma''}}\neq 0$, so the sums $\sum _{{\gamma'+\gamma''=\gamma}}f_{{\gamma'}}g_{{\gamma''}}$ are meaningful.

We say that a family $ (f_{i})_{i\in I}$ of elements of $\Hahn$ is summable if the following properties are satisfied~: 
\begin{itemize}
 \item the set $\bigcup_{i \in I } \supp(f_{i})$ is well-ordered;
 \item for any $\gamma \in \QQ$, the set 
$\{i \in I \ \vert \  \gamma \in \supp(f_{i})\}$ is finite.
\end{itemize}
 In this case, we can  legitimately set  
$$
\sum_{i\in I} f_{i} = \sum_{\gamma \in \QQ} \left(\sum_{i \in I} f_{i,\gamma}\right) z^{\gamma} \in \Hahn
$$
where $f_{i}=\sum_{\gamma \in \QQ} f_{i,\gamma} z^{\gamma}$. 
We let $\Hahn^{<0}$ be the set made of the $f \in \Hahn$ such that $\supp(f) \subset \QQ_{<0}$.
We have the following elementary result.

\begin{lem}[{\cite[Lemma~23]{RoquesFrobForMahler}}]\label{lem: phik summable}
For any $f \in \Hahn^{<0}$, the family $(\malop{\ellmahl}^{k}(f))_{k \leq -1}$ of elements of $\Hahn$ is summable.
\end{lem}

\subsection{The Hahn series $\xi_{\tuple}$ and a related filtered algebra}\label{sec:les chi}

This section is devoted to the Hahn series $(\xi_{\tuple})_{\tuple \in \ala}$ introduced in Section~\ref{sec:intro gpms}.   
We successively show that they are indeed well-defined Hahn series, that they can be constructed inductively—a property that will be used several times throughout the rest of the paper—and that they are $\ellmahl$-Mahler.
We conclude this section by introducing a filtered algebra associated with these Hahn series, which will play a central role in the present paper.

\subsubsection{The $\xi_{\tuple}$ are well-defined Hahn series} \label{sec: xi well def hahn series}
To prove that the right-hand side of \eqref{eq:def_chi intro} is legitimate and defines a Hahn series, it is sufficient to prove that, 
for any $t \in \Z_{\geq 1}$ and $\a=(a_1,\ldots,a_t) \in \QQ_{> 0}^{t}$, we have~: 
\begin{enumerate}
 \item \label{xi well def 1} the following set is well-ordered:
$$\left\{-\frac{a_1}{\ellmahl^{k_1}} - \frac{a_2}{\ellmahl^{k_1+k_2}}  - \cdots - \frac{a_t}{\ellmahl^{k_1+k_2+\cdots + k_t}} \ \vert \  (k_{1},\ldots,k_{t}) \in \Z_{\geq 1}^{t}\right\};$$  
\item \label{xi well def 2} for any $\gamma \in \QQ$, there are at most finitely many $(k_{1},\ldots,k_{t}) \in \Z_{\geq 1}^{t}$ such that 
$$
\gamma=-\frac{a_1}{p^{k_1}} - \frac{a_2}{p^{k_1+k_2}}  - \cdots - \frac{a_t}{p^{k_1+k_2+\cdots + k_t}}.
$$ 
\end{enumerate}
Let us prove these properties by induction on $t \in \Z_{\geq 1}$.
 \vskip 5pt 
\noindent {\bf Base case $t=1$.} 
Properties \eqref{xi well def 1} and \eqref{xi well def 2} are obviously satisfied when $t=1$. 
 \vskip 5pt 
\noindent {\bf Inductive step $t \rightarrow t+1$.} We assume that properties \eqref{xi well def 1} and \eqref{xi well def 2} 
hold true for some $t \in \Z_{\geq 1}$. Consider $\a=(a_1,\ldots,a_{t+1}) \in \QQ_{> 0}^{t+1}$. 
By induction hypothesis (applied to $(a_{2},\ldots,a_{t+1}) \in \QQ_{>0}^{t}$),  
 $$
f=\sum_{k_2,\ldots,k_{t+1} \geq 1}  z^{-a_1 - \frac{a_2}{\ellmahl^{k_2}}  - \cdots - \frac{a_{t+1}}{\ellmahl^{k_2+\cdots + k_{t+1}}}}
$$ 
is a well-defined element of $\Hahn^{<0}$. 
Lemma~\ref{lem: phik summable} guarantees that  $(\malop{\ellmahl}^{k_{1}}(f))_{k_{1} \leq -1}$ is summable; this means exactly that:
\begin{enumerate}[label=(\roman*)]
\item \label{supp phi k1 f well ordered} the following set is well-ordered:
 \begin{multline*}
 \bigcup_{k_{1} \in \Z_{\geq 1}} \supp(\malop{\ellmahl}^{k_{1}}(f))=\\
\left\{-\frac{a_1}{\ellmahl^{k_1}} - \frac{a_2}{\ellmahl^{k_1+k_2}}  - \cdots - \frac{a_{t+1}}{\ellmahl^{k_1+k_2+\cdots + k_{t+1}}} \ \vert \  (k_{1},\ldots,k_{t+1}) \in \Z_{\geq 1}^{t+1}\right\};
\end{multline*}
\item \label{finitely many k1} for any $\gamma \in \QQ$, there are at most finitely many $k_{1} \in \Z_{\geq 1}$ such that 
\begin{multline*}
 \gamma \in \supp(\malop{\ellmahl}^{k_{1}}(f)) = \\
 \left\{-\frac{a_1}{p^{k_1}} - \frac{a_2}{p^{k_1+k_2}}  - \cdots - \frac{a_{t+1}}{p^{k_1+k_2+\cdots + k_{t+1}}} \ \vert \  (k_{2},\ldots,k_{t+1}) \in \Z_{\geq 1}^{t}\right\}.
\end{multline*}
\end{enumerate}
Of course, \ref{supp phi k1 f well ordered} shows that $\a$ satisfies  \eqref{xi well def 1}. Moreover, combining \ref{finitely many k1} with the fact that, by induction, for any $k_{1} \in \Z_{\geq 1}$, there are at most finitely many $(k_{2},\ldots,k_{t+1}) \in \Z_{\geq 1}^{t}$ such that $\gamma=-\frac{a_1}{p^{k_1}} - \frac{a_2}{p^{k_1+k_2}}  - \cdots - \frac{a_{t+1}}{p^{k_1+k_2+\cdots + k_{t+1}}}$, we get that there are at most finitely many $(k_{1},\ldots,k_{t}) \in \Z_{\geq 1}^{t}$ such that 
$
\gamma=-\frac{a_1}{p^{k_1}} - \frac{a_2}{p^{k_1+k_2}}  - \cdots - \frac{a_t}{p^{k_1+k_2+\cdots + k_t}}
$ and, hence, $\a$ satisfies \eqref{xi well def 2}. This 
concludes the induction. 

\subsubsection{Inductive construction of the Hahn series $\xi_{\tuple}$} \label{sec: xi can be def induct}
An elementary yet key observation is that, for any $t \in \Z_{\geq 1}$ and for any $\tuple=(\balpha,\lambd,\a) \in \alat{t}$, we have 
 \begin{equation}\label{eq:def_xi_summable}
\xi_{\tuple}
=\sum_{k\geq 1} k^{\alpha_1}(\lambda_1\cdots\lambda_t)^{k}\malop{\ellmahl}^{-k}(z^{-a_1}\xi_{\tuplespt})
\end{equation}
where $\tuplespt=(\balphaspt,\lambdspt,\aspt)\in\alat{t-1}$ with $\balphaspt=(\alpha_2,\ldots,\alpha_t)$, $\lambdspt=(\lambda_2,\ldots,\lambda_t)$ and $\aspt=(a_2,\ldots,a_t)$.
Formula \eqref{eq:def_xi_summable} is inductive in the sense that it expresses $\xi_{\tuple}(z)$ with $\tuple \in \alat{t}$ in terms of $\xi_{\tuplespt}(z)$ with $\tuplespt \in \alat{t-1}$. This formula will be used several times in the rest of the paper.

\subsubsection{The Hahn series $\xi_{\tuple}$ are $\ellmahl$-Mahler}\label{sec: the xi are p mahler}
To prove that, for all  $\tuple=(\balpha,\lambd,\a) \in \ala$, the Hahn series $\xi_{\tuple}$ are $\ellmahl$-Mahler, {\it i.e.}, that they are solutions of equations of the form \eqref{eq: mahler1}, we argue by well-founded induction on $\balpha$. 
To do so, we consider the partial well-order $\prec$ on $\bigcup_{t \geq 0} \Z_{\geq 0}^{t}$ defined as follows. For $\balpha=(\alpha_1,\ldots,\alpha_t) \in \Z_{\geq 0}^{t}$ and $\balpha'=(\alpha'_1,\ldots,\alpha'_{t'}) \in \Z_{\geq 0}^{t'}$, we write $\balpha' \prec \balpha$ if one of the following holds:
\begin{itemize}
	\item $t'<t$,
	\item $t'=t\geq 1$ and $\alpha'_1 < \alpha_1$,
\end{itemize}

We are now ready to proceed with the well-founded induction.
Consider an arbitrary $\tuple=(\balpha,\lambd,\a) \in \ala$ and suppose that, for any $\tuple'=(\balpha',\lambd',\a') \in \ala$ such that $\balpha' \prec \balpha$, the Hahn series $\xi_{\tuple'}$ is $\ellmahl$-Mahler. If $\tuple=((),(),())$, then $\xi_{\tuple}(z)=\xi_{((),(),())}(z)=1$ is $\ellmahl$-Mahler. So, we can assume that $\tuple \neq ((),(),())$. The case $j=1$ of Lemma~\ref{lem:phichi-lamddchi} below ensures that there exists $\lambda \in \Q^\times$ such that  
$
\delta:=\malop{\ellmahl}(\xi_{\tuple})-\lambda \xi_{\tuple} 
$
is a $\Kinf$-linear combination of some $\xi_{\tuple'}$ with $\tuple'=(\balpha',\lambd',\a') \in \ala$ such that $\balpha' \prec \balpha$.
Since these $\xi_{\tuple'}$ are $\ellmahl$-Mahler by the induction hypothesis, $\delta$ 
is $\ellmahl$-Mahler  as well. But, it is a basic fact that, if $g$ is a Hahn series such that $g - c\malop{\ellmahl}(g)$ is $\ellmahl$-Mahler for some $c\in \Kinf$, then $g$ is $\ellmahl$-Mahler as well. 
Applying this to $g=\xi_{\tuple}$ and $c=\lambda$, we get that $\xi_{\tuple}$ is  $\ellmahl$-Mahler. This concludes the induction.

To complete the proof, it remains to establish the following lemma, which will also be used later in the paper. From now on, we will use the notation  
$$
\Q[z^{-\frac{1}{\starpourexp}}] = \bigcup_{k\in \Z_{\geq 1}} \Q[z^{-\frac{1}{k}}].
$$

\begin{lem}\label{lem:phichi-lamddchi}
For all $t \in \Z_{\geq 1}$, for all $\tuple=(\balpha,\lambd,\a) \in \alat{t}$, for all $j \in \Z$, 
we have 
\begin{equation}\label{eq:chi zp moins chi}
\xi_{\tuple}(z^{\ellmahl^{j}})-(\lambda_{1} \cdots \lambda_{t})^{j} \xi_{\tuple}(z) = 
 \sum_{\tuple'  \in E } \eta_{\tuple'} \xi_{\tuple'}(z) + \sum_{\tuple'  \in F } p_{\tuple'}(z)\xi_{\tuple'}(z) 
\end{equation}
where 
\begin{itemize}
 \item $E \subset 
(\{0,\ldots,\alpha_{1}-1\} \times \Z_{\geq 0}^{t-1}) \times   (\Qbar^{\times})^{t} \times \QQ_{>0}^t
$ is finite and the $\eta_{\tuple'}$ belong to $\Qbar^{\times}$;
 \item $F \subset \bigcup_{s \in \{0,\ldots,t-1\}} \alat{s}$ is finite and the $p_{\tuple'}(z)$ belong to $\polyramstrictinfzero$. 
\end{itemize}
Furthermore, it is possible to arrange that, for any $\tuple'=(\balpha',\lambd',\a')$ in the support of the sums in \eqref{eq:chi zp moins chi}, $\a'$ is a sub-tuple of $\a$ and, if $j\geq 1$ and $\a\in \Z_{>0}^{t}$, that  $p_{\tuple'}(z)$ belong to $\Qbar[z^{-1}] \setminus \Qbar$.
\end{lem}

\begin{proof}
In this proof, we set $\lambda=\lambda_1\cdots\lambda_t$. 

Let us first prove the result for $j=1$. Applying $\malop{\ellmahl}$ to formula \eqref{eq:def_xi_summable} from Section~\ref{sec: xi can be def induct}, and adopting the notation $\tuplespt$ introduced therein, we obtain 
\begin{align*}\label{eq:Mahler_xi}
\nonumber\xi_{\tuple}(z^\ellmahl) &= \sum_{k\geq 1} k^{\alpha_1}\lambda^{k}\malop{\ellmahl}^{1-k}(z^{-a_1}\xi_{\tuplespt}(z))\\
 &=  \sum_{l\geq 0} (l+1)^{\alpha_1}\lambda^{l+1}\malop{\ellmahl}^{-l}(z^{-a_1}\xi_{\tuplespt}(z))
\\ &=  \lambda z^{-a_1}\xi_{\tuplespt}(z) + \lambda \sum_{l\geq 1} \sum_{j=0}^{\alpha_1}\binom{\alpha_1}{j} l^j\lambda^{l}\malop{\ellmahl}^{-l}(z^{-a_1}\xi_{\tuplespt}(z))
\\ \nonumber &= \lambda z^{-a_1}\xi_{\tuplespt}(z) + \lambda \sum_{j=0}^{\alpha_1}\binom{\alpha_1}{j}\xi_{((j,\alpha_2,\ldots,\alpha_t),\lambd,\a)}(z),
\end{align*}
where 
the last equality follows from \eqref{eq:def_xi_summable} again. This entails that 
$$
\xi_{\tuple}(z^\ellmahl)-\lambda \xi_{\tuple}(z)
=
\lambda z^{-a_1}\xi_{\tuplespt}(z) + \lambda \sum_{j=0}^{\alpha_1-1}\binom{\alpha_1}{j}\xi_{((j,\alpha_2,\ldots,\alpha_t),\lambd,\a)}(z),
$$
which is an equality of the desired form. 

The case of an arbitrary $j \in \Z_{\geq 1}$ follows from an easy induction using the particular case $j=1$.

We now consider the case $j=-1$. 
Applying $\malop{\ellmahl}^{-1}$ to formula \eqref{eq:def_xi_summable} from Section~\ref{sec: xi can be def induct}, and adopting the notation $\tuplespt$ introduced therein, we obtain 
\begin{align*}
\xi_{\tuple}(z^{\ellmahl^{-1}}) &= \sum_{k\geq 1} k^{\alpha_1}\lambda^{k}\malop{\ellmahl}^{-1-k}(z^{-a_1}\xi_{\tuplespt}(z))\\
 &=  \sum_{l\geq 2} (l-1)^{\alpha_1}\lambda^{l-1}\malop{\ellmahl}^{-l}(z^{-a_1}\xi_{\tuplespt}(z))\\
 &=   \lambda^{-1} \sum_{l\geq 2} \sum_{j=0}^{\alpha_1}c_{j} l^j\lambda^{l}\malop{\ellmahl}^{-l}(z^{-a_1}\xi_{\tuplespt}(z))  \text{ \ \ \ \ \ \  with } c_{j}=(-1)^{\alpha_{1}-j}\binom{\alpha_1}{j}\\
&= -\lambda^{-1}  \sum_{j=0}^{\alpha_1} c_{j} \lambda \malop{\ellmahl}^{-1}(z^{-a_1}\xi_{\tuplespt}(z))+\lambda^{-1} \sum_{l\geq 1} \sum_{j=0}^{\alpha_1}c_{j} l^j\lambda^{l}\malop{\ellmahl}^{-l}(z^{-a_1}\xi_{\tuplespt}(z)) \\
&= -0^{\alpha_{1}} z^{-\frac{a_1}{\ellmahl}}\malop{\ellmahl}^{-1}(\xi_{\tuplespt}(z)) +\lambda^{-1}  \sum_{j=0}^{\alpha_1}c_{j} \xi_{((j,\alpha_2,\ldots,\alpha_t),\lambd,\a)}(z)
\end{align*}
where the last equality follows from \eqref{eq:def_xi_summable} again. This entails that 
\begin{multline*}
 \xi_{\tuple}(z^{\ellmahl^{-1}}) - \lambda^{-1} \xi_{\tuple}(z) 
= 
 -0^{\alpha_{1}} z^{-\frac{a_1}{\ellmahl}}\left(\xi_{\tuplespt}(z^{\ellmahl^{-1}})-\widecheck{\lambda}^{-1} \xi_{\tuplespt}(z)\right) 
  \\ -0^{\alpha_{1}} z^{-\frac{a_1}{\ellmahl}}\widecheck{\lambda}^{-1} \xi_{\tuplespt}(z) +\lambda^{-1}  \sum_{j=0}^{\alpha_1-1}c_{j} \xi_{((j,\alpha_2,\ldots,\alpha_t),\lambd,\a)}(z)
\end{multline*}
where $\widecheck{\lambda}=\lambda_{2} \cdots \lambda_{t}$ if $t \geq 2$ or $1$ if $t=1$. 
Now, since $\tuplespt \in \alat{t-1}$, the case $j=-1$ of the lemma follows from the latter formula using an obvious induction on $t$.

The case of an arbitrary $j \in \Z_{\leq -1}$ follows from an easy induction using the particular case $j=-1$.
\end{proof}

\subsubsection{The filtered $\polyraminfzero$-algebra $(\V,(\V_{s})_{ s \in \Z_{\geq 0}})$}\label{sec:space_V}
We will make repeated use of the $\Qbar$-vector spaces defined, for any $s\in \Z_{\geq 0}$, by 
\begin{equation*}
\V_{s}= \operatorname{Span}_{\Qbar}\left(\left\{z^{-\gamma}\xi_{\tuple} \ \vert \ \gamma \in \QQ_{\geq 0},\,
\tuple \in   \bigcup_{t=0}^{s}\alat{t}
\right\}\setminus\{1\}\right)\!
\end{equation*}
and by 
\begin{equation*}
\mathcal V 
= \bigcup_{s \in \Z_{\geq 0}} \V_s  = \operatorname{Span}_{\Qbar}\left(\left\{z^{-\gamma}\xi_{\tuple} \ \vert \ \gamma \in \QQ_{\geq 0},\,
	\tuple \in   \ala
	\right\}\setminus\{1\}\right).
\end{equation*}

The next result summarizes the main properties of the $\V_{s}$ and of $\V$ that will be useful to us. 

\begin{prop}\label{lem:VsVs'}
The pair $(\V,(\V_{s})_{ s \in \Z_{\geq 0}})$ is a filtered (non-unital) sub-$\polyraminfzero$-algebra of $\Hahn$, {\it i.e.}, 
 \begin{itemize}
 \item \label{V sub alg} $\V$ is a (non-unital) sub-$\polyraminfzero$-algebra of $\Hahn$;
 \item \label{Vs non dec etc} $(\V_{s})_{ s \in \Z_{\geq 0}}$ is a nondecreasing sequence of sub-$\polyraminfzero$-modules of $\Hahn$ such that 
\begin{itemize}
 \item \label{Vs non dec etc union} $\mathcal V = \bigcup_{s \in \Z_{\geq 0}}\mathcal V_s$;
 \item \label{Vs non dec etc prod} for all $s,s' \in \Z_{\geq 0}$, $\V_{s} \cdot \V_{s'} \subset \V_{s+s'}$.
\end{itemize}
\end{itemize} 
Moreover, we have $\malop{\ellmahl}(\V)=\V$ and, for all $s \in \Z_{\geq 0}$, $\malop{\ellmahl}(\V_{s})=\V_{s}$.
 \end{prop}

To prove this result, it will be convenient to use the following alternate description of the $\V_{s}$. 

\begin{lem}\label{lem:Vs in terms xi tilde}
For any $s\in \Z_{\geq 0}$, we have 
\begin{equation*}
\V_{s}= \operatorname{Span}_{\Qbar}\left(\left\{z^{-\gamma}\widetilde{\xi}_{\tuple} \ \vert \ \gamma \in \QQ_{\geq 0},\,
\tuple \in   \bigcup_{t=0}^{s}\alat{t}
\right\} \setminus \{1\}\right)
\end{equation*} 
where, for all $t \in \{1,\ldots,s\}$ and for all $\tuple=(\balpha,\lambd,\a) \in \ala_t$,
	$$
	\widetilde{\xi}_{\tuple}(z)=\sum_{1 \leq k_1< \cdots < k_t} k_{1}^{\alpha_{1}} \cdots k_{t}^{\alpha_{t}} \lambda_1^{k_1}\lambda_2^{k_2}\cdots\lambda_t^{k_t} z^{-\frac{a_1}{p^{k_1}} - \frac{a_2}{p^{k_2}}  - \cdots - \frac{a_t}{p^{ k_t}}} \in \Hahn
	$$
and  $
\widetilde{\xi}_{((),(),())}(z)= 1.
$ 	
\end{lem}

\begin{proof}
The result follows immediately from the following remarks. 
Using the formula
\begin{multline*}
 \xi_{\tuple}(z) 
=\\ 
\sum_{k_1,\ldots,k_t \geq 1} k_{1}^{\alpha_{1}} \cdots k_{t}^{\alpha_{t}} \lambda_1^{k_1}\lambda_2^{k_1+k_2}\cdots\lambda_t^{k_1+\cdots+k_t} z^{-\frac{a_1}{p^{k_1}} - \frac{a_2}{p^{k_1+k_2}}  - \cdots - \frac{a_t}{p^{k_1+k_2+\cdots + k_t}}} 
\\
=\sum_{1\leq l_1 < \cdots < l_t} l_{1}^{\alpha_{1}} (l_{2}-l_{1})^{\alpha_{2}} \cdots (l_{t}-l_{t-1})^{\alpha_{t}} \lambda_1^{l_1}\lambda_2^{l_{2}}\cdots\lambda_t^{l_{t}} z^{-\frac{a_1}{p^{l_1}} - \frac{a_2}{p^{l_2}}  - \cdots - \frac{a_t}{p^{l_t}}} 
\end{multline*}
and expanding $l_{1}^{\alpha_{1}} (l_{2}-l_{1})^{\alpha_{2}} \cdots (l_{t}-l_{t-1})^{\alpha_{t}} $, we see that $\xi_{\tuple}$ is a $\Qbar$-linear combination of certain $\widetilde{\xi}_{\tuple'}$ with $\tuple' \in \ala_t$. 
Conversely, using the formula 
\begin{multline*}
\widetilde{\xi}_{\tuple}(z) 
=
\sum_{1 \leq k_1< \cdots < k_t} k_{1}^{\alpha_{1}} \cdots k_{t}^{\alpha_{t}} \lambda_1^{k_1}\lambda_2^{k_2}\cdots\lambda_t^{k_t} z^{-\frac{a_1}{p^{k_1}} - \frac{a_2}{p^{k_2}}  - \cdots - \frac{a_t}{p^{ k_t}}}
\\
=\sum_{l_1,\ldots,l_t \geq 1} l_{1}^{\alpha_{1}} (l_{1}+l_{2})^{\alpha_{2}} \cdots (l_{1}+\cdots+l_{t})^{\alpha_{t}} \lambda_1^{l_1}\lambda_2^{l_{1}+l_2}\cdots\lambda_t^{l_{1}+\cdots+l_t}\\
\times z^{-\frac{a_1}{p^{l_1}} - \frac{a_2}{p^{l_{1}+l_2}}  - \cdots - \frac{a_t}{p^{l_{1}+\cdots+ l_t}}}
\end{multline*}
and expanding $ l_{1}^{\alpha_{1}} (l_{1}+l_{2})^{\alpha_{2}} \cdots (l_{1}+\cdots+l_{t})^{\alpha_{t}}$, we see that $\widetilde{\xi}_{\tuple}$ is a $\Qbar$-linear combination of certain $\xi_{\tuple'}$ with $\tuple' \in \ala_t$. 
\end{proof}

\begin{proof}[Proof of Proposition~\ref{lem:VsVs'}]
The fact that $(\V_{s})_{ s \in \Z_{\geq 0}}$ is a nondecreasing sequence of sub-$\polyraminfzero$-modules of $\Hahn$ satisfying $\mathcal V = \bigcup_{s \in \Z_{\geq 0}}\mathcal V_s$ 
is clear. 
Moreover, using Lemma~\ref{lem:Vs in terms xi tilde}, we see that to prove the inclusion $\V_{s} \cdot \V_{s'} \subset \V_{s+s'}$, 
it is sufficient to prove that, for any $s,s' \in \Z_{\geq 1}$, for any 
$\tuple \in \alat{s}$ and any 
$\tuple' \in \alat{s'}$,  
we have  $\widetilde{\xi}_{\tuple}(z) \widetilde{\xi}_{\tuple'}(z)\in \V_{s+s'}$. 
To show this, we first note that the product $\widetilde{\xi}_{\tuple}(z)\widetilde{\xi}_{\tuple'}(z)$ is the sum indexed by the set
$$
\mathcal S=\{(k_1,\ldots,k_s,k_1',\ldots,k_{s'}') \in \Z_{\geq 1}^{s+s'} \ \vert \ k_1<\cdots < k_s \text{ and } k_1' < \cdots < k_{s'}'\}
$$
of the following terms
\begin{multline}\label{eq:terms_products}
k_1^{\alpha_1}\cdots k_s^{\alpha_s}(k_1')^{\alpha'_1}\cdots (k_{s'})^{\alpha'_{s'}}\lambda_1^{k_1}\cdots\lambda_s^{k_s}(\lambda'_1)^{k_1'}\cdots(\lambda'_{s'})^{k_{s'}'} \\
\times z^{-\frac{a_1}{\ellmahl^{k_1}}-\cdots -\frac{a_s}{\ellmahl^{k_s}}-\frac{a'_1}{\ellmahl^{k_1'}}-\cdots -\frac{a'_{s'}}{\ellmahl^{k'_{s'}}}}, 
\end{multline}
where we have used Notation \ref{notation:omega, alpha, etc}.
The set $\mathcal S$ can be written as the disjoint union of finitely many sets of the form
\begin{multline*}
 \mathcal S_{E,\sigma}= \{(n_1,\ldots,n_{s+s'}) \in \mathcal S \ \vert \  n_{\sigma(1)}\leq \cdots \leq n_{\sigma(s+s')} \\ 
 \text{and, for any } i<j, \ (n_{i}=n_{j} \Leftrightarrow (i,j) \in E)\}
\end{multline*}
where $E$ is a subset of $\{1,\ldots,s\}\times\{s+1,\ldots,s+s'\}$ and $\sigma$ is a permutation of $\{1,\ldots,s+s'\}$. 
Moreover, it is easily seen that, for $\mathcal S_{E,\sigma}$ to be nonempty, it is necessary and sufficient that the following conditions be satisfied: 
\begin{enumerate}
	\item if $(i_{1},j_{1})$ and $(i_{2},j_{2})$ are distinct elements of $E$, then either ($i_{1}<i_{2}$ and $j_{1}<j_{2}$) or ($i_{2}<i_{1}$ and $j_{2}<j_{1}$); 
	\item $
	\sigma^{-1}(1)<\cdots < \sigma^{-1}(s)$ and $\sigma^{-1}(s+1)<\cdots <\sigma^{-1}(s+s');  
	$
	\item for all $(i,j) \in E$, $\vert \sigma^{-1}(j)-\sigma^{-1}(i)\vert =1$. 
\end{enumerate}
Now, it is clear that, for all $E$ and $\sigma$ satisfying these conditions, the sum of \eqref{eq:terms_products} over all $(k_1,\ldots,k_s,k_1',\ldots,k_{s'}') \in \mathcal S_{E,\sigma}$ equals $\widetilde{\xi}_{\tuple''}$ for some $\tuple'' \in \ala_{s+s'-{\rm Card}(E)}$. This implies that $\widetilde{\xi}_{\tuple}(z) \widetilde{\xi}_{\tuple'}(z)\in \V_{s+s'}$, as wanted.

What we have proven so far clearly implies that $\V= \bigcup_{s \in \Z_{\geq 0}}\mathcal V_s$ is a sub-$\polyraminfzero$-module of $\Hahn$ closed by product, so it is a sub-$\polyraminfzero$-algebra of $\Hahn$. 
This concludes the proof of the fact that $(\V,(\V_{s})_{ s \in \Z_{\geq 0}})$ is a filtered sub-$\polyraminfzero$-algebra of $\Hahn$.

The last assertion of the proposition follows from the equality 
$$
\malop{\ellmahl}(\xi_{(\balpha,\lambd,\a)}(z)) = 
\xi_{(\balpha,\lambd,\ellmahl \a)}(z). 
$$
\end{proof} 

\section{Equations, systems and modules}\label{trad syst equ mod}

In this section, we introduce the notions of $\ellmahl$-Mahler equations, systems and modules, and we explore the relationships between them. Note that, although our focus is on the Mahler case, the following discussion can be extended to arbitrary difference fields, such as the $q$-difference or shift cases. See for instance \cite{VdPS97}.

Recall that $\malop{\ellmahl}$ denotes the field automorphism of $\Kinf=\Q(z^{\frac{1}{\starpourexp}})$ that maps $f(z)$ to $f(z^p)$. The pair $(\Kinf,\malop{\ellmahl})$ forms a difference field, that is,  a field equipped with an automorphism. 
Throughout this section, we let $(K,\psi)$ be a difference field extension of $(\Kinf,\malop{\ellmahl})$, {\it i.e.}, $K$ is a field extension of $\Kinf$ and $\psi$ is a field automorphism of $K$ extending $\malop{\ellmahl}$.  Here are some examples: $K=\Kinf$, $K=\puis$ the field of Puiseux series or $K=\Hahn$ the field of Hahn series, 
each endowed with the automorphism that maps $\sum _{{\gamma\in \QQ }}f_{\gamma}z^{\gamma}$ to $\sum _{{\gamma\in \QQ }}f_{\gamma}z^{\ellmahl\gamma}$.  In what follows, we will slightly abuse notation by continuing to denote the automorphism $\psi$ as $\malop{\ellmahl}$.

\subsection{Equations}

A $\ellmahl$-Mahler equation over $K$, is an equation of the form
\begin{equation}\label{eq: mahler sec syst mod} 
a_0 f + a_1 \malop{\ellmahl}(f)+ \cdots + a_d \malop{\ellmahl}^{d}(f) = 0 
\end{equation}
with $a_{0},\ldots,a_d \in K$ such that $a_0a_d \ne 0$. 

\subsection{Systems} By $\ellmahl$-Mahler system over $K$, we mean a system of the form 
\begin{equation*}
\malop{\ellmahl}(F)=AF 
\end{equation*}  
with $A \in \GL_{d}(K)$. 

For any $A,R \in \GL_d(K)$, we set
	\begin{equation}\label{eq:chgt_jauge}
	R[A]:=\malop{\ellmahl}(R)AR^{-1}.
	\end{equation}
We say that two systems $\malop{\ellmahl}(F)=AF $ and $\malop{\ellmahl}(F)=BF $ with $A,B \in \GL_{d}(K)$ are $K$-equivalent if there exists $R \in \GL_{d}(K)$ such that 
$$
B=R[A].
$$ 
Such an $R$ is usually called a gauge transformation.

\subsection{From equations to systems}\label{sec:eq to syst} 
We recall that any $\ellmahl$-Mahler equation  can be converted into a $\ellmahl$-Mahler system as follows: the equation 
\eqref{eq: mahler sec syst mod}  is equivalent to the system 
\begin{equation*}\label{matrix system equa app}
\malop{\ellmahl}(F)=AF 
\end{equation*}  
where 
\begin{equation}\label{eq:syst assoc eq}
F=
\begin{pmatrix}
f \\
\malop{\ellmahl}(f)\\
\vdots \\
\malop{\ellmahl}^{d-1}(f)
\end{pmatrix}
\text{ and }  
A=\begin{pmatrix}
0&1&0&\cdots&0\\
0&0&1&\ddots&\vdots\\
\vdots&\vdots&\ddots&\ddots&0\\
0&0&\cdots&0&1\\
-\frac{a_{0}}{a_{d}}& -\frac{a_{1}}{a_{d}}&\cdots & \cdots & -\frac{a_{d-1}}{a_{d}}
\end{pmatrix}.
\end{equation}

\subsection{Modules}
We denote by 
$$
\mathcal{D}_{K}=K\langle \malopop,\malopop^{-1} \rangle
$$ 
the Ore algebra of noncommutative Laurent polynomials in the indeterminate $\malopop$ and with coefficients in $K$ such that, for all $f\in K$,
$$
\malopop f = \malop{\ellmahl}(f) \malopop.
$$ 
 A left $\mathcal{D}_{K}$-module which is finite dimensional as a $K$-vector space will be called a $\ellmahl$-Mahler module (over $K$). By definition, the rank of a $\ellmahl$-Mahler module is its dimension as a $K$-vector space.

Given two $\ellmahl$-Mahler modules $M$ and $N$, the notation $M \cong N$ means that $M$ and $N$ are isomorphic as left $\mathcal{D}_{K}$-modules. 

\subsection{From systems to modules and {\it vice versa}} \label{sec: from syst to mod and vv} 
It is sometimes useful to work with $\ellmahl$-Mahler modules instead of $\ellmahl$-Mahler systems and {\it vice versa}.
Let us briefly recall how to move from one point of view to the other.

One can associate to any $\ellmahl$-Mahler system 
\begin{equation*}\label{syst pour syst eq mod}
\malop{\ellmahl} (Y)= AY
\end{equation*}
with $A \in \GL_{d}(K)$ a $\ellmahl$-Mahler module $M_{A}$ whose underlying abelian group is $K^d$ by letting $\mathcal D_K$ act on any $\v \in K^d$ as follows:
$$
\malopop(\v)=A^{-1}  \malop{\ellmahl}(\v),
$$ 
where $\malop{\ellmahl}$ acts component-wise on $\v$ which is seen as a column vector.

Conversely, we can attach a $\ellmahl$-Mahler system  to any $\ellmahl$-Mahler module $M$ {\it via} the choice of a $K$-basis $\mathcal{B}=(e_{1},\ldots,e_{d})$ of $M$, namely $\malop{\ellmahl} (Y)= AY$ where $A^{-1} \in \GL_{d}(K)$ represents the action of $\malopop$ on $\mathcal{B}$ ({\it i.e.}, the $j$th column of $A^{-1}$ represents $\malopop(e_{j})$ in the basis $\mathcal{B}$). We have $M\cong M_{A}$. 

It is easily seen that two $\ellmahl$-Mahler systems $\malop{\ellmahl} (Y)= AY$ and  $\malop{\ellmahl} (Y)= BY$ with $A,B \in \GL_{d}(K)$ are $K$-equivalent
if and only if the corresponding $\ellmahl$-Mahler modules $M_{A}$ and $M_{B}$ are isomorphic.

\subsection{From modules to equations} 

Last, we recall the following classical result, known as the cyclic vector lemma, ensuring that any $\ellmahl$-Mahler module ``comes from'' an equation (for a proof, see for instance \cite[Theorem~B.2]{SolDiffEqFinTerms}).  

\begin{prop}\label{cyclic vect lem}
	For any $\ellmahl$-Mahler module $M$, there exists a nonzero $L \in \mathcal{D}_{K}$ such that $M \cong \mathcal{D}_{K}/\mathcal{D}_{K}L$.
\end{prop}

Finally, note that if \( M \cong \mathcal{D}_{K}/\mathcal{D}_{K}L \) where 
\( L = \sum_{i=0}^{d} a_{i} \malopop^{i} \in \mathcal{D}_{K} \) with $d \geq 1$ and $a_{0}a_{d} \neq 0$, 
then \( M \cong M_{B} \) where \( B = (A^{T})^{-1} \) is the 
contragredient ({\it i.e.}, the inverse of the transpose) of the matrix \( A \) 
given in \eqref{eq:syst assoc eq}; see for instance \cite[p.~192]{S04} 
for explanations about this in the context of \( q \)-difference equations, 
the present case being entirely similar.

\section{Fundamental matrices of solutions of Mahler systems and proof of Theorem~\ref{thm:decomp_chi}}\label{sec: fund mat sol}

Consider a $\ellmahl$-Mahler system 
\begin{equation}\label{the mahler syst}
\malop{\ellmahl}(Y)=AY \text{ with }  A \in \GL_{d}(\puis).
\end{equation}
According to \cite[Theorem~2]{RoquesLSMS}, it is $\Hahn$-equivalent to a system with constant coefficients, say 
\begin{equation}\label{the mahler syst const}
\malop{\ellmahl}(Y)=CY \text{ for some } C \in \GL_{d}(\Qbar),
\end{equation} 
which means that there exists a gauge transformation $F \in \GL_{d}(\Hahn)$ with
\begin{equation}\label{eq: FAz}
A=F[C].
\end{equation} 
There are situations where the entries of $F$ belong to $\puis$, as the following result, that will be used later, illustrates.

\begin{prop}[{\cite[Proposition~34]{Ro15}}]\label{prop: gauge eq cst reg sing case}
	Assume that the entries of $A$ have non-negative $z$-adic valuations and that $A(0) \in \GL_{d}(\Qbar)$. Then, the $\ellmahl$-Mahler system \eqref{the mahler syst} is $\puis$-equivalent to the system with constant coefficients $\malop{\ellmahl}(Y)=A(0)Y$. \end{prop}

However, in general, the entries of 
$F$ do not lie in 
$\puis$, and the first goal of this section is to gain a deeper understanding of their nature. This is achieved in Section~\ref{sec:forme F}, where we state and prove Theorem~\ref{theo: fund syst sol princ} showing that the gauge transformation $F$ can be chosen in the form
\[
F = F_1 F_2,
\]
where \( F_1 \in \GL_d(\puis) \) and $F_2 \in \GL_d(\V)$ is of a very specific shape (see Section~\ref{sec:space_V} for the definition of $\V$). \black In Section~\ref{sec: fund mat sol rappels}, we construct a fundamental matrix of solutions $e_{C}$ to~\eqref{the mahler syst const} with entries in a suitable difference ring extension \( \R \) of \( \Hahn \), from which we deduce a fundamental matrix of solutions to~\eqref{the mahler syst}, namely   
\[
Y_0 = F e_C = F_1 F_2 e_C \in \GL_{d}(\R).
\]
This will enable us to establish Theorem~\ref{thm:decomp_chi} in Section~\ref{sec:proof_thm_decomp_chi}, and will play a central role in the proof of Theorem~\ref{thm:pureté1}.

\subsection{The gauge transformation $F$}\label{sec:forme F} 
Before stating our main result, we need to introduce certain subgroups of $\GL_{d}(\Hahn)$.
Namely, for any $\nbblocks \in \Z_{\geq 1}$ and $\r=(r_{1},\ldots,r_{\nbblocks}) \in \Z_{\geq 1}^{\nbblocks}$ such that 
$r_{1}+\cdots+r_{\nbblocks}=d$,  
we let $\mathfrak{V}_{\r}$ be the set of unipotent block upper triangular matrices of the form 
\begin{equation}\label{eq: matrix F def mathfrak V}
F=\left(
\begin{array}{cccc}
	I_{r_{1}} & & \hspace*{-.4cm}\overset{j\text{th column}}{\downarrow} &
	\\ & \ddots &\hspace*{-.4cm} F_{i,j} & \hspace*{-.6cm}\overset{{\scriptsize{i\text{th row}}}}{\leftarrow}
	\\ &\bigzero & \hspace*{-.4cm}\ddots &
	\\  &  && \hspace*{-.6cm} I_{r_{\nbblocks}}  \\ 
\end{array}\right) \in \GL_{d}(\Hahn)
\end{equation}
such that, for all $i,j \in \{1,\ldots,\nbblocks\}$ with $j>i$, the strictly upper block
$
F_{i,j}
$ 
belongs to 
$
\Mat_{r_{i},r_{j}}(\V_{j-i}), 
$ 
the set $\V_{j-i}$ being defined in Section~\ref{sec:space_V}. 
Here and in the rest of the paper, for any $m,n \in \Z_{\geq 1}$ and any set $E$, we let $\Mat_{m,n}(E)$ be the set of $m\times n$ matrices with entries in $E$ and we write $\Mat_{n}(E)=\Mat_{n,n}(E)$ for the set of $n\times n$ square matrices with entries in $E$. 

\begin{prop}
$\mathfrak{V}_{\r}$ is a subgroup of $\GL_{d}(\Hahn)$. 
\end{prop}

\begin{proof}
Of course, $\mathfrak{V}_{\r}$ contains $I_{d}$. 
The fact that $\mathfrak{V}_{\r}$ is invariant by product follows immediately from the fact that $(\V,(\V_{s})_{ s \in \Z_{\geq 0}})$ is a filtered $\polyraminfzero$-algebra in virtue of  Proposition~\ref{lem:VsVs'}. It remains to prove that $\mathfrak{V}_{\r}$ is invariant by inversion. Consider $F \in \mathfrak{V}_{\r}$. We have $F^{-1}=(I_{d}+(F-I_{d}))^{-1}=\sum_{k=0}^{d} (-1)^{k} (F-I_{d})^{k}$. Using Proposition~\ref{lem:VsVs'} again, it is easily seen that the latter sum belongs to $\mathfrak{V}_{\r}$.  
\end{proof}

\begin{theo}\label{theo: fund syst sol princ}
Consider the $\ellmahl$-Mahler system \eqref{the mahler syst}. There exist 
\begin{itemize}
\item $\nbblocks \in \Z_{\geq 1}$ and $\r=(r_{1},\ldots,r_{\nbblocks}) \in \Z_{\geq 1}^{\nbblocks}$ such that $r_{1}+\cdots+r_{\nbblocks}=d$;
\item  an invertible block upper triangular matrix of the form
\begin{equation*}
\Theta=  \left(
\begin{array}{ccc}
A_{1}    &       &  *  \\ 
\bord & \ddots       &     \\ 
0 & \bord    & A_{\nbblocks}     \\ 
\end{array}\right) \in \GL_{d}(\polyraminfzero)
\end{equation*}
with $A_{1} \in \GL_{r_{1}}(\Qbar), \ldots, A_{\nbblocks} \in \GL_{r_{\nbblocks}}(\Qbar)$, whose constant term is denoted by $C \in \GL_{d}(\Qbar)$; 
\item  $F_{1} \in \GL_{d}(\puis)$;
\item $F_{2} \in \mathfrak{V}_{\r}$;
\end{itemize}
 such that 
 \begin{equation}\label{eq:F1Theta=AF1}
 A = F_{1}[\Theta] =\malop{\ellmahl}(F_{1})\Theta F_{1}^{-1} \text{ and \ } \Theta = F_{2}[C] = \malop{\ellmahl}(F_{2})CF_{2}^{-1}.
\end{equation} 
In particular, the matrix $F=F_{1}F_{2} \in \GL_{d}(\Hahn)$ satisfies 
\begin{equation}\label{eq:F1F2eC}
A=F[C]=\malop{\ellmahl}(F)CF^{-1}.
\end{equation}  
\end{theo}
\black

\begin{rem}\label{rem pour theo: fund syst sol princ}
One can always choose $\nbblocks=d$ and $r_{1}=\cdots=r_{d}=1$ in Theorem~\ref{theo: fund syst sol princ} by triangularizing the matrices $A_{i}$. However, we have stated the result in this form for the following reason:  if $L$ is a Mahler operator associated to the system \eqref{the mahler syst} by the cyclic vector lemma, then the proof of Theorem~\ref{theo: fund syst sol princ} shows that we can take $r_{1},\ldots,r_{\nbblocks}$ to be the multiplicities of the slopes of $L$ in the sense of \cite[Section~4]{RoquesFrobForMahler}. This more precise information will not be exploited in this paper but could be of interest for other purposes. 
\end{rem}

\begin{proof}[Proof of Theorem~\ref{theo: fund syst sol princ}]\label{sec: proof of theo: fund syst sol princ}

The proof relies on the following three steps, justified in the next three sections. 
\vskip 5 pt

\noindent {\bf Step 1}. 
There exist an integer $\nbblocks \in \Z_{\geq 1}$, a tuple $\r=(r_{1},\ldots,r_{\nbblocks}) \in \Z_{\geq 1}^{\nbblocks}$ such that $r_{1}+\cdots+r_{\nbblocks}=d$ and a matrix $G \in \GL_{d}(\puis)$ such that  
\begin{equation}\label{syst block diag}
A':=G[A]
=
\malop{\ellmahl}(G)AG^{-1}
=  \left(\begin{array}{ccc}
A_{1}    &       &  *   \\ 
\bord & \ddots       &     \\ 
0 & \bord    & A_{\nbblocks}     \\ 
\end{array}\right) 
\end{equation}
for some  $A_{1} \in \GL_{r_{1}}(\Qbar), \ldots, A_{\nbblocks} \in \GL_{r_{\nbblocks}}(\Qbar)$. \vskip 5 pt
\noindent {\bf Step 2}. 
There exists a block upper triangular matrix of the form 
$$
H = \left(
\begin{array}{ccc}
I_{r_{1}}    &       &  *    \\ 
\bord & \ddots       &      \\ 
0 & \bord    & I_{r_{\nbblocks}}     
\end{array}\right)  \in \GL_{d}(\puis)
$$ 
such that the entries of the above-diagonal blocks of  
\begin{equation}\label{eq: equiv puis leq 0}
 A'':=H[A']
 =
\malop{\ellmahl}(H)A'H^{-1}
 =\left(
\begin{array}{ccc}
A_{1}    &       &  *   \\ 
\bord & \ddots       &      \\ 
0 & \bord    & A_{\nbblocks}     
\end{array}\right)
\end{equation}
belong to $\puisinfzero$. 
\vskip 5 pt
\noindent {\bf Step 3}. 
There exists $K \in \mathfrak{V}_{\r}$ such that 
\begin{equation}\label{eq:KA''C}
 K[A'']=C, 
\end{equation}
where $C \in \GL_{d}(\Qbar)$ is the constant term of the matrix $A''$. \vskip 5pt

We are now in a position to conclude the proof. Indeed, we have:
$$
K[H[G[A]]]=\malop{\ellmahl}(KHG)A(KHG)^{-1}=C.
$$ 
Thus, the matrices $F_{1}=(HG)^{-1}$, $F_{2}=K^{-1}$, $\Theta=H[G[A]]$ and $C$ have the properties required by Theorem~\ref{theo: fund syst sol princ}. 
\end{proof}

\subsubsection{Proof of Theorem~\ref{theo: fund syst sol princ}: justification of Step~1}

Our approach is based on a factorization property for  $\ellmahl$-Mahler operators, with which we begin. Consider a nonzero $\ellmahl$-Mahler operator  
\begin{equation*}
L= a_{d} \malopop^{d} + a_{d-1}\malopop^{d-1} + \cdots + a_{0} \in \mathcal{D}_{\puis}
\end{equation*}
with coefficients $a_{0},\ldots,a_{d} \in \puis$ such that $a_{0}a_{d}\neq 0$.  
Using the terminology from \cite[Section~4]{RoquesFrobForMahler}, we let  
\begin{itemize}
\item $
\mu_{1} < \cdots < \mu_{\nbblocks}
$  be 
the slopes of $L$ with  respective multiplicities $r_{1},\ldots,r_{\nbblocks}$; 
\item $c_{i,1},\ldots,c_{i,r_{i}} \in \Q^\times$ be the exponents (repeated with multiplicities) of $L$ attached to the slope $\mu_{i}$. 
\end{itemize}
For any $f = \sum_{\gamma \in \mathbb Q} f_\gamma z^\gamma \in \Hahn$, we let $\val f = \min \supp(f) \in \QQ \cup \{+\infty\}$ denote the valuation of $f$ and, if $f \neq 0$, $\cld f = f_{\val f} \in \Q \setminus\{0\}$ denote the coefficient of least degree.

\begin{prop}\label{prop fact op}
The operator $L$ has a factorization of the form 
$$
L=aL_{\nbblocks}\cdots L_{1} 
$$
where 
\begin{itemize}
	\item $a \in \puis^{\times}$ is such that $\val a=\val a_{0}$;  
	\item $\cld a = \prod_{i=1}^{k} \prod_{j=1}^{r_{i}} (-c_{i,j})^{-1} \cld a_{0}$;
	\item the $L_{i}$ are given by 
	$$
	L_{i}= (z^{\nu_{i}}\malopop- c_{i,r_{i}})h_{i,r_{i}}^{-1}\cdots (z^{\nu_{i}}\malopop-c_{i,1})h_{i,1}^{-1}
	$$
	for some $h_{i,j}\in \puis^{\times}$ with $\val h_{i,j}=0$, $\cld h_{i,j} = 1$ and 
	\begin{equation*}\label{def nui}
	\nu_{i} =(\ellmahl-1) (\ellmahl^{r_{1}+\cdots+r_{i-1}}(\mu_{i}-\mu_{i-1})+\cdots+\ellmahl^{r_{1}}(\mu_{2}-\mu_{1})+\mu_{1}). 
	\end{equation*} 
\end{itemize}
\end{prop}

\begin{proof}
This result is proved in \cite[Proposition~15]{RoquesFrobForMahler} over $\Hahn$ instead of $\puis$; the proof in the present case is entirely similar. 
\end{proof}

We are now in a position to justify Step~1. 
With the notations of Section~\ref{sec: from syst to mod and vv}, it is equivalent to prove that the $\ellmahl$-Mahler module $M:=M_{A}$ over $\puis$ associated with $\malop{\ellmahl} (Y)= AY$ has a filtration 
$
\{0\}=M_{0}\subset M_{1} \subset \cdots \subset M_{\nbblocks}=M
$ 
such that, for all $i\in \{0,\ldots,\nbblocks-1\}$,  
$$
M_{i+1}/M_{i} \cong M_{A_{i}}
$$ 
for some $A_{i} \in \GL_{r_{i}}(\Qbar)$ where $r_{i} = \dim_{\puis} M_{i+1}/M_{i}$. 
Let us prove this. According to the cyclic vector lemma (see Proposition~\ref{cyclic vect lem}), there exists a nonzero $L \in \dpuis$ such that 
$
M \cong \dpuis/\dpuis L
$.
The factorization 
$$
L=aL_{\nbblocks} \cdots L_{1}
$$
given by Proposition~\ref{prop fact op} induces a filtration 
$
\{0\}=M_{0}\subset M_{1} \subset \cdots \subset M_{\nbblocks}=M
$ 
such that, for all $i \in \{0,\ldots,\nbblocks-1\}$, 
$$M_{i+1}/M_{i} \cong \dpuis/\dpuis L_{i}.$$ 
Note that 
\begin{equation*}
\dpuis/\dpuis L_{i} \cong \dpuis/\dpuis \widetilde{L}_{i} 
\end{equation*}
with 
\begin{equation*}
\widetilde{L}_{i} = z^{\frac{\nu_{i}}{\ellmahl-1}} L_{i} z^{-\frac{\nu_{i}}{\ellmahl-1}}= (\malopop- c_{i,r_{i}})h_{i,r_{i}}^{-1}\cdots (\malopop-c_{i,1})h_{i,1}^{-1}
=\sum_{j=0}^{r_{i}}\widetilde{a}_{j} \malopop^{j}
\end{equation*}
for some $\widetilde{a}_{j} \in \puis$ with non-negative $z$-adic valuation and $\widetilde{a}_{0}(0) \widetilde{a}_{r_{i}}(0) \neq 0$. 
Moreover, as recalled at the end of Section~\ref{sec: from syst to mod and vv}, we have 
\begin{equation*}
 \dpuis/\dpuis \widetilde{L}_{i}  \cong M_{\widetilde{A}_{i}}
\end{equation*}
where $\widetilde{A}_{i} \in \GL_{r_{i}}(\puis)$ is the contragredient of the matrix  of the $\ellmahl$-Mahler system associated with $\widetilde{L}_{i}$ by \eqref{eq:syst assoc eq}. The above mentioned properties of the $\widetilde{a}_{j}$ imply that the entries of $\widetilde{A}_{i}$ have non-negative $z$-adic valuation and that $\widetilde{A}_{i}(0) \in \GL_{r_{i}}(\Qbar)$. 
So, it follows from 
Proposition~\ref{prop: gauge eq cst reg sing case} that the systems $\malop{\ellmahl} (Y)= \widetilde{A}_{i}Y$ and $\malop{\ellmahl} (Y)= \widetilde{A}_{i}(0)Y$ are $\puis$-equivalent and, hence, that $M_{\widetilde{A}_{i}} \cong M_{\widetilde{A}_{i}(0)}$. Setting $A_i=\widetilde{A}_i(0)$, this concludes the justification of Step~1.

\subsubsection{Proof of Theorem~\ref{theo: fund syst sol princ}: justification of Step~2}
We let $\mathfrak{H}_{\r}$ be the group of block upper triangular matrices of the form 
$$
\left(
\begin{array}{ccc}
	I_{r_{1}}    &       &  *    \\ 
	\bord & \ddots       &      \\ 
	0 & \bord    & I_{r_{\nbblocks}}     
\end{array}\right)  \in \GL_{d}(\puis)
$$ 
and we let $\E$ be the set of block upper triangular matrices of the form
$$
\left(
\begin{array}{ccc}
	A_1    &       &  *    \\ 
	\bord & \ddots       &      \\ 
	0 & \bord    & A_{\nbblocks}     
\end{array}\right)  \in \GL_{d}(\puis). 
$$
Note that the matrix $A'$ given by Step~1 belongs to $\E$. In what follows, for any $B \in \E$ and any $(i,j) \in \{1,\ldots,\nbblocks\}^{2}$, we let $B_{i,j}$ be the $(i,j)$-block of $B$.

Our strategy for justifying Step~2 is as follows. We equip the set $\mathcal S:=\{(k,l) \in \{1,\ldots,\nbblocks\}^2 \ \vert \ k<l\}$ with the strict total order $<_{\mathcal S}$ defined by $(k,l)<_{\mathcal S} (k',l')$ if either ($l=l'$ and $k'<k$) or ($l<l'$). We have 
	$$(1,2)<_{\mathcal S} (2,3)<_{\mathcal S} (1,3)<_{\mathcal S} (3,4) <_{\mathcal S} (2,4)<_{\mathcal S} (1,4)  <_{\mathcal S} \cdots <_{\mathcal S} (\nbblocks-1,\nbblocks).$$
We shall construct, for any $(k,l) \in \mathcal S$, a matrix $B^{[k,l]} \in \E$ satisfying the following properties:  
\begin{enumerate}
 \item\label{eq:cond syst equiv} the systems $\malop{\ellmahl} Y=A'Y$ and $\malop{\ellmahl} Y=B^{[k,l]}Y$ are $\puis$-equivalent {\it via} a gauge transformation in $\mathfrak{H}_{\r}$; 
 \item\label{eq:syst coeff dans puis leq 0} for all $(i,j) \in\mathcal S$ with $(i,j)<_{\mathcal S} (k,l)$ or $(i,j)=(k,l)$, the $(i,j)$-block $(B^{[k,l]})_{i,j}$ of $B^{[k,l]}$ has coefficients in $\puisinfzero$.
\end{enumerate}
In particular, the systems $\malop{\ellmahl} Y=A'Y$ and $\malop{\ellmahl} Y=B^{[\nbblocks-1,\nbblocks]}Y$ are $\puis$-equivalent {\it via} a gauge transformation in $\mathfrak{H}_{\r}$ and $A''=B^{[\nbblocks-1,\nbblocks]}$ is of the form \eqref{eq: equiv puis leq 0}, thereby justifying Step~2.

Our construction will use the matrices defined, for any $(i,j) \in \mathcal S$ and $M \in \Mat_{r_{i},r_{j}}(\puis)$, by 
$$
T_{i,j}(M)=\left(
\begin{array}{cccc}
	I_{r_{1}} & & \hspace*{-.4cm}\overset{j\text{th  column}}{\downarrow} &
	\\ & \ddots &\hspace*{-.4cm} M & \hspace*{-.6cm}\overset{{\scriptsize{i\text{th  row}}}}{\leftarrow}
	\\ && \hspace*{-.4cm}\ddots &
	\\  &  && \hspace*{-.6cm} I_{r_{\nbblocks}}  \\ 
\end{array}\right) \in \mathfrak{H}_{\r},
$$
where all the unspecified blocks are zero. Straightforward calculations show that 
$$
T_{i,j}(M)^{-1}=T_{i,j}(-M),
$$ 
Moreover, for any $B \in \E$, it is easily seen that the matrix $T_{i,j}(M)[B]$ belongs to $\E$ and that its $(k,l)$-block is equal to 
\begin{equation}\label{eq:TijMB}
\left\{\begin{array}{ll}
-A_{i}M(z)+B_{i,j} + M(z^{\ellmahl})A_{j} & \text{ if } (k,l)=(i,j),
\\   B_{i,l}+ M(z^{\ellmahl})B_{j,l} & \text{ if } k= i \text{ and } l >j,
\\  -B_{k,i}M(z)+B_{k,j} & \text{ if } k <i \text{ and } l=j,
\\ B_{k,l} & \text{ otherwise.}\end{array}\right.
\end{equation}

We are now in a position to explain the construction of the matrices $B^{[k,l]}$, which is recursive with respect to $<_{\mathcal S}$. 
In what follows, for any $F=\sum_{\gamma \in \QQ} F_{\gamma} z^{\gamma} \in \Mat_{d}(\Hahn)$, we set 
\begin{equation}\label{eq:nota_troncation}
F^{0}=F_{0},\qquad F^{<0}=\sum_{\gamma \in \QQ_{<0}} F_{\gamma} z^{\gamma} \quad \text{ and }\quad F^{>0}=\sum_{\gamma \in \QQ_{>0}} F_{\gamma} z^{\gamma},
\end{equation} 
so that $F=F^{<0}+F^{0}+F^{>0}.$
\vskip 5pt
\noindent {\bf Base case: construction of $B^{[1,2]}$.} According to Lemma~\ref{lem phiFA1-A2F bis} below, there exists $M  \in \Mat_{r_{1},r_{2}}(\puis)$ such that $\malop{\ellmahl}(M)A_{2}-A_{1}M=-({A}_{1,2}')^{>0}$. Then, using~\eqref{eq:TijMB}, we see that the matrix
$$
B^{[1,2]}=T_{1,2}(M)[A'] 
$$
belongs to $\E$ and that 
$$
(B^{[1,2]})_{1,2} = A_{1,2}'+\malop{\ellmahl}(M)A_2-A_1M =A_{1,2}-({A}_{1,2}')^{>0}=({A}_{1,2}')^{\leq0}
$$ 
has coefficients in $\puisinfzero$. 
\vskip 5pt
\noindent {\bf Inductive step: construction of $B^{[k,l]}$ from its predecessor.}
Let $(k,l) \in \mS \setminus \{ (1,2)\}$. Let $(k',l')$ be the predecessor of $(k,l)$ in $\mS$ with respect to $<_{\mathcal S}$ and suppose that $B^{[k',l']}$ has already been built. 
According to Lemma~\ref{lem phiFA1-A2F bis} below, there exists $M \in \Mat_{r_{k},r_{l}}(\puis)$ such that $\malop{\ellmahl}(M)A_{l}-A_{k}M=-(B^{[k',l']})_{k,l}^{>0}$. Then, using \eqref{eq:TijMB}, we see that the matrix 
$$
B^{[k,l]}=T_{k,l}(M)[B^{[1,2]}] 
$$
belongs to $\E$ and that the matrices 
\begin{eqnarray*}
 (B^{[k,l]})_{i,j}&=&(B^{[k',l']})_{i,j}, \ \ \text{ for any } (i,j) \in {\mathcal S} \text{ such that } (i,j)<_{\mathcal S}(k,l),\\ 
 (B^{[k,l]})_{k,l}&=&(B^{[k',l']})_{k,l}+\malop{\ellmahl}(M)A_l-A_kM =(B^{[k',l']})_{k,l}^{\leq0}
\end{eqnarray*}
have coefficients in $\puisinfzero$. This concludes the inductive step.
\vskip 5pt
To complete the justification of Step~2, it only remains to prove the following result. 
\begin{lem}\label{lem phiFA1-A2F bis}
Let $d_1,d_2 \geq 1$ be integers. Consider $C_{1} \in \GL_{d_1}(\Qbar)$, $C_{2} \in \GL_{d_2}(\Qbar)$ and $B \in \Mat_{d_1,d_2}(\puis)$ whose entries have positive valuation. There exists $M \in \Mat_{d_1,d_2}(\puis)$ such that 
\begin{equation}\label{eq:phiFA1-A2F bis}
C_{1}M-\malop{\ellmahl}(M)C_{2}=B.  
\end{equation}
\end{lem}

\begin{proof}
An easy calculation shows that the following matrix satisfies \eqref{eq:phiFA1-A2F bis}:
$$
M =\sum_{k \geq 0} C_{1}^{-k-1} \malop{\ellmahl}^k(B) C_{2}^{k} \in \Mat_{d_1,d_2}(\puis).
$$
\end{proof}

\subsubsection{Proof of Theorem~\ref{theo: fund syst sol princ}: justification of Step~3}
We use the following notation for the matrix $A''$ given by Step~2:
$$
A''=\left(
\begin{array}{cccc}
	A_{1} & & \hspace*{-.4cm}\overset{j\text{th column}}{\downarrow} &
	\\ & \ddots &\hspace*{-.4cm} A''_{i,j} & \hspace*{-.6cm}\overset{{\scriptsize{i\text{th row}}}}{\leftarrow}
	\\ && \hspace*{-.4cm}\ddots &
	\\  0 &  && \hspace*{-.6cm} A_{\nbblocks}  \\ 
\end{array}\right)
$$
with $A''_{i,j} \in \Mat_{r_{i},r_{j}}(\puisinfzero)$.
For any 
$$
K=\left(
\begin{array}{cccc}
	I_{r_{1}} & & \hspace*{-.4cm}\overset{j\text{th  column}}{\downarrow} &
	\\ & \ddots &\hspace*{-.4cm}   K_{i,j} & \hspace*{-.6cm}\overset{{\scriptsize{i\text{th  row}}}}{\leftarrow}
	\\ & & \hspace*{-.4cm}\ddots &
	\\   0  &  && \hspace*{-.6cm} I_{r_{\nbblocks}}  \\ 
\end{array}\right) \in \mathfrak{V}_{\r}, 
$$
the equation \eqref{eq:KA''C} 
is equivalent, with the notations \eqref{eq:nota_troncation}, to:   for all $(i,j)\in\mathcal S:=\{(k,l) \in \{1,\ldots,\nbblocks\}^2 \ \vert \ k<l\}$, 
$$
\sum_{k=i}^{j} \malop{\ellmahl}(K_{i,k}) A''_{k,j} =   \sum_{k=i}^{j} (A''_{i,k})^{0} K_{k,j}. 
$$
This can be rewritten as follows: for all $(i,j) \in \mathcal S$, 
\begin{multline}\label{eq Fij}
 \malop{\ellmahl}(K_{i,j}) A_{j}-A_{i}K_{i,j}=   \sum_{k=i+1}^{j} (A''_{i,k})^{0} K_{k,j}-\sum_{k=i}^{j-1}  \malop{\ellmahl}(K_{i,k}) A''_{k,j}  \\
=-(A''_{i,j})^{<0} +   \sum_{k =i+1}^{j-1} (A''_{i,k}) K_{k,j} -\sum_{k =i+1}^{j-1}  \malop{\ellmahl}(K_{i,k}) A''_{k,j}
\end{multline}
where the sums $ \sum_{k =i+1}^{j-1}$ are $0$ if $j-1<i+1$. 
Let us prove the existence of such matrices $K_{i,j} \in \Mat_{r_{i},r_{j}}(\V_{j-i})$ by induction on $\delta=j-i \in \{1,\ldots,\nbblocks-1\}$. 
\vskip 5 pt
\noindent {\bf Base case $\delta = 1$.} The case $\delta=1$ corresponds to $j=i+1$. In this case, \eqref{eq Fij} reduces to 
$$
 \malop{\ellmahl}(K_{i,j}) A_{j}-A_{i}K_{i,j}= -(A''_{i,j})^{<0}. 
$$
The right hand side $-(A''_{i,j})^{<0}$ has entries in $\V_{0}$ 
and Lemma~\ref{lem phiFA1-A2F} below ensures the existence of $K_{i,j} \in \Mat_{r_{i},r_{j}}(\V_{1})$ satisfying the latter equation.
\vskip 5 pt
\noindent {\bf Inductive step.} Consider $\delta \in \{1,\ldots,\nbblocks-2\}$ and assume the existence of $K_{i,j} \in \Mat_{r_{i},r_{j}}(\V_{j-i})$ satisfying \eqref{eq Fij} for all $(i,j)\in\mathcal S$ such that $j-i \leq \delta$. Then, for any $(i,j)\in\mathcal S$ with $j-i=\delta+1$, the right hand side of \eqref{eq Fij} is known and has entries in $\V_{\delta}$ and  Lemma~\ref{lem phiFA1-A2F} below ensures that we can find $K_{i,j} \in \Mat_{r_{i},r_{j}}(\V_{\delta+1})$ satisfying \eqref{eq Fij}. This proves the inductive step.  \vskip 5 pt

To complete the proof of Step~3, it remains to prove the following lemma. At first glance, it may appear similar to Lemma~\ref{lem phiFA1-A2F bis}, but they actually address fundamentally distinct situations: in Lemma~\ref{lem phiFA1-A2F bis}, $B$ has Puiseux series coefficients with positive valuation, whereas in Lemma~\ref{lem phiFA1-A2F} its coefficients are in $\V_{s}$ and, hence,  are either $0$ or Hahn series with support in $\mathbb Q_{<0}$.

\begin{lem}\label{lem phiFA1-A2F}
Let $s\geq 0$ and $d_1,d_2 \geq 1$ be integers. Consider $C_{1} \in \GL_{d_1}(\Qbar)$, $C_{2} \in \GL_{d_2}(\Qbar)$ and $B \in \Mat_{d_1,d_2}(\V_s)$. There exists $F \in \Mat_{d_1,d_2}(\V_{s+1})$ such that 
\begin{equation}\label{eq:phiFA1-A2F}
\malop{\ellmahl}(F)C_{2}-C_{1}F=B.  
\end{equation}
\end{lem}

\begin{proof}
By $\Qbar$-linearity, it suffices to treat the case 
$$
B=hR
$$
with $R \in \Mat_{d_{1},d_{2}}(\Qbar)$ and $h \in \V_s$ of one of the following forms:
\begin{enumerate}[label=(\roman*)]
\item \label{h case 1} $h=z^{-\gamma}\xi_{\tuple}$ with $\tuple \in \ala_t$ for some $t \in \{0,\ldots,s\}$ and $\gamma \in \QQ_{>0}$; 
\item \label{h case 2} $h=\xi_{\tuple}$ with $\tuple \in \ala_t$ for some $t \in \{1,\ldots,s\}$.
\end{enumerate}
Since $h \in \V_s \subset \Hahn^{<0}$, Lemma~\ref{lem: phik summable} ensures that 
$(\malop{\ellmahl}^{-k}(h)C_{1}^{k-1} R C_{2}^{-k})_{k \geq 1}$ is summable. So, we can legitimately consider the matrix
$$
F = \sum_{k \geq 1} \malop{\ellmahl}^{-k}(h) C_{1}^{k-1}  R C_{2}^{-k},
$$
which has Hahn series entries and is easily seen to satisfy~\eqref{eq:phiFA1-A2F}. It only remains to prove that it belongs to $ \Mat_{d_1,d_2}(\V_{s+1})$.  
We let $C_{1}=D_{1}+N_{1}$ and $C_{2}^{-1}=D_{2}+N_{2}$ be the Dunford-Jordan decomposition of $C_{1}$ and $C_{2}^{-1}$ respectively ({\it i.e.}, for each $i \in \{1,2\}$, $D_{i}$ is diagonalizable, $N_{i}$ is nilpotent and $D_{i}N_{i}=N_{i}D_{i}$). Using the Newton binomial formula, we obtain 
$$
C_{1}^{k}
=\sum_{l=0}^{d_{1}-1} \binom{k}{l}D_{1}^{k-l}N_{1}^{l}
\text{\ \ and \ \ }
C_{2}^{-k}
=\sum_{l=0}^{d_{2}-1} \binom{k}{l}N_{2}^{l}D_{2}^{k-l}
$$ 
and we see that $F$ is a $\Qbar$-linear combination of terms of the form  
\begin{equation}\label{sol Di}
\sum_{k \geq 1} k^{\alpha} \malop{\ellmahl}^{-k}(h)  D_{1}^{k-1} S D_{2}^{k}
\end{equation}
with $S \in \Mat_{d_{1},d_{2}}(\Qbar)$ and $\alpha \in \Z_{\geq 0}$. In order to conclude the proof, it remains to prove that the entries of \eqref{sol Di} are in $\V_{s+1}$. Since the latter property is invariant by left and right multiplication by elements of $\GL_{d_1}(\Qbar)$ and $\GL_{d_2}(\Qbar)$ respectively, we can assume that the $D_{i}$ are diagonal, say $D_{i}=\operatorname{diag}(c_{i,1},\ldots,c_{i,d_i})$. In that case, setting $S=(s_{i,j})_{1\leq i\leq d_1,\,1 \leq j \leq d_2}$, we have 
$$
\eqref{sol Di} = \left(s_{i,j} \sum_{k \geq 1} k^{\alpha} c_{1,i}^{k-1}c_{2,j}^{k} h(z^{\frac{1}{\ellmahl^{k}}})  \right)_{1\leq i\leq d_1,\,1 \leq j \leq d_2}. 
$$  
Thus, it only remains to prove the following statement: for any $c \in \Qbar^\times$, for any $\alpha \in \Z_{\geq 0}$ and for any $h \in \V_{s}$ 
of the form \ref{h case 1} or \ref{h case 2}, the sum 
 \begin{equation}\label{eq:sum_negative_MahlerOp}
\sum_{k \geq 1} k^{\alpha}c^k\malop{\ellmahl}^{-k}(h) 
\end{equation}
belongs to $\V_{s+1}$.
Let us prove this. 
The first case~\ref{h case 1} follows directly from~\eqref{eq:def_xi_summable}. In the second case \ref{h case 2}, writing $\tuple=(\balpha,\lambd,\a)$ as in Notation~\ref{notation:omega, alpha, etc} and setting $\lambda_0=c/(\lambda_1\cdots\lambda_t)$, we have 
$$
\eqref{eq:sum_negative_MahlerOp}=\sum_{k,k_1,\ldots,k_t\geq 1} k^{\alpha}k_1^{\alpha_1}\cdots k_t^{\alpha_t}\lambda_0^k\lambda_{1}^{k+k_1}\cdots\lambda_t^{k+k_1+\cdots+k_t}z^{-\frac{a_1}{\ellmahl^{k+k_1}}-\cdots-\frac{a_t}{\ellmahl^{k+k_1+\cdots+k_t}}}.
$$
Making the change of variable $k+k_1 \to k_1$ in the summation, we get 
$$
\eqref{eq:sum_negative_MahlerOp}=\sum_{j=0}^{\alpha_1} \sum_{k_1\geq 2}\sum_{k_2,\ldots,k_t \geq 1} \theta_j(k_1)k_1^{j}k_2^{\alpha_2}\cdots k_t^{\alpha_t}\lambda_{1}^{k_1}\cdots\lambda_t^{k_1+\cdots+k_t}z^{-\frac{a_1}{\ellmahl^{k_1}}-\cdots-\frac{a_t}{\ellmahl^{k_1+\cdots+k_t}}}
$$
where $\theta_j(k_1)=\sum_{k=1}^{k_1-1}\binom{\alpha_1}{j}(-1)^{\alpha_1-j}k^{\alpha+\alpha_1-j}\lambda_0^k=:P_j(k_1)\lambda_0^{k_1}$ for some polynomials $P_j \in \Q[X]$. It follows that \eqref{eq:sum_negative_MahlerOp} is a $\Qbar$-linear combination of 
\begin{equation}\label{eq: sum start 2}
 \sum_{k_1\geq 2}\sum_{k_2,\ldots,k_t \geq 1} k_1^{j}k_2^{\alpha_2}\cdots k_t^{\alpha_t}(\lambda_{0}\lambda_1)^{k_1}\cdots\lambda_t^{k_1+\cdots+k_t}z^{-\frac{a_1}{\ellmahl^{k_1}}-\cdots-\frac{a_t}{\ellmahl^{k_1+\cdots+k_t}}}
\end{equation}
with $j \in \Z_{\geq 0}$. 
Since the sum in \eqref{eq: sum start 2} differs from the same sum starting at $k_1=1$ only by an element of $\V_{t-1}$, and since this extended sum clearly lies in 
 $\V_{t}$,  it follows that the original sum \eqref{eq: sum start 2} belongs to $\V_t$ as well. Thus \eqref{eq:sum_negative_MahlerOp} belongs to $\V_t$ and, therefore, to $\V_{s+1}$ since $t<s+1$. This ends the proof of the lemma.
\end{proof}

\subsection{Remarks on the proof of Theorem~\ref{theo: fund syst sol princ}}\label{subsec:rks proof theo fund syst sol princ}

The proof of Step~1 relies on the possibility of factoring any $\ellmahl$-Mahler operator into factors that each have a single slope.
Similar factorization properties are encountered in other types of operators, such as $q$-difference or differential operators. This implies that any $q$-difference or differential system can be brought—after a suitable gauge transformation—into block upper triangular form, with diagonal blocks having only one slope (see \cite{dVS11,VdPS}). What is particular to the Mahler case is that the diagonal blocks can be made constant. As seen in the proof of Step~1, this is due to the fact that systems associated with $\ellmahl$-Mahler operators having a single slope are regular singular in the sense of \cite[Section~5]{Ro15}. This property holds neither in the $q$-difference nor in the differential case, where regular singularity is a very specific situation corresponding to slope zero. Moreover, in the $q$-difference or differential case, a system that is block upper triangular with regular singular diagonal blocks is regular singular and, hence, can be put, after a suitable gauge transformation, in constant form. This is not true for $\ellmahl$-Mahler systems, as reflected by the fact that, in general, reducing such a system to constant form requires a gauge transformation with Hahn series coefficients—rather than merely Puiseux series coefficients.

\subsection{Solving the $\ellmahl$-Mahler system \eqref{the mahler syst}}\label{sec: fund mat sol rappels}

\subsubsection{Solving systems with constant coefficients} \label{sec: fund mat sol rappels the ring R}

In what follows, we let $\Rspe$ be a difference ring extension of $\Hahn$ 
with field of constants $\Qbar$ such that:
\begin{itemize}
	\item there exists $\logm \in \Rspe$ satisfying $\malop{\ellmahl}(\logm)=\logm+1$;
	\item for any $c \in \Q^\times$, there exists an invertible element $e_c \in \Rspe$ satisfying $\malop{\ellmahl}(e_c)=ce_c$.
\end{itemize}  
Such a ring $\Rspe$ exists. 
Indeed, let $(X_{c})_{c \in \Qbar^{\times}}$ and $Y$ be indeterminates over~$\Hahn$, and consider the quotient ring 
$$
\Rspe := \Hahn [(X_{c})_{c \in \Qbar^{\times}},Y]/I
$$
of the polynomial ring $\Hahn [(X_{c})_{c \in \Qbar^{\times}},Y]$ by its ideal $I$ generated by $\{X_{c}X_{d}-X_{cd} \ \vert \ c,d \in \Qbar^{\times}\} \cup\{X_{1}-1\}$.
Let $e_{c}$ (resp.~$\logm$) be the image of $X_{c}$ (resp.~$Y$) in $\Rspe$ and let us endow $\Rspe$ with the unique ring automorphism $\malop{p}$ extending  $\malop{p}: \Hahn \rightarrow \Hahn$ and such that
$$
\forall c \in \Qbar^{\times}, \ \malop{p}(e_{c}) = ce_{c} \text{ and } \malop{p}(\ell)=\ell+1.
$$
The pair $(\Rspe,\malop{p})$ is a difference ring extension of $(\Hahn,\malop{p})$ with field of constants $\Qbar$. We omit the proof of this assertion as it is entirely similar to the proof of the second assertion of \cite[Theorem~35]{Ro15}.  

We are now in a position to construct a fundamental matrix of solutions to any $\ellmahl$-Mahler system with constant coefficients. 
\begin{lem}\label{lem:sys_const}
	For any $C \in \GL_{d}(\Qbar)$, there exists  $e_{C} \in \GL_{d}(\Rspe)$ such that 
\begin{equation}\label{eq: eAz}
\malop{\ellmahl}(e_{C})=Ce_{C},
\end{equation}
whose entries are $\Qbar$-linear combinations of elements of the form $e_c\logm^j$ where 
$c$ is an eigenvalue of $C$ and $j \in \{0,\ldots,d-1\}$.
\end{lem}

\begin{proof}
Let 
$
C=UD
$ 
be the multiplicative Dunford-Jordan decomposition of the matrix $C$, {\it i.e.}, $D \in \GL_{d}(\Qbar)$ is diagonalizable, $U \in \GL_{d}(\Qbar)$ is unipotent and $UD=DU$. 
We recall that $D$ and $U$ belong to $\Qbar[C]$. 
We set 
$$
\logm^{[k]}= \left\{
\begin{array}{cc}
\binom{\logm}{k}=\frac{\logm(\logm-1)\cdots (\logm-k+1)}{k!} & \mbox{if } k \in \Z_{\geq 0}, \\
0 & \mbox{if } k \in \Z_{\leq -1}.
\end{array}
\right.
$$
It follows from the equality $\malop{\ellmahl}(\logm)=\logm+1$ that 
$$
\malop{\ellmahl} (\logm^{[k]})=\logm^{[k]}+\logm^{[k-1]}
$$
and that (the finite sum) 
$$e_{U} = \sum_{k \geq 0} \logm^{[k]} (U-I_{d})^{k}\in \GL_{d}(\Rspe)
$$ 
satisfies 
$$
\malop{\ellmahl}(e_{U} )=U e_{U}.
$$ 
Note that $e_{U}$ belongs to $\Rspe[U] \subset \Rspe[C]$ and that $\det e_{U}=1$. 

Moreover, we consider $P \in \GL_{d}(\Qbar)$ and $c_{1},\ldots,c_{d} \in \Qbar^{\times}$ such that $$D = P \operatorname{diag}(c_{1},\ldots,c_{d})P^{-1}$$ and we set 
$$
e_{D} = P\operatorname{diag}(e_{c_{1}},\ldots,e_{c_{d}})P^{-1} \in \GL_{d}(\Rspe).
$$ 
This $e_{D}$ is independent of $P$ and satisfies
$$
\malop{}(e_{D} )=D e_{D}.
$$ 

Since $DU=UD$, the matrices $e_U$ and $D$ commute and it follows from what precedes that $e_{C}=e_{U}e_{D}$ has the required properties.
 \end{proof}

\subsubsection{Fundamental matrix of solutions of \eqref{the mahler syst}}\label{sec:fms}

Combining \eqref{eq:F1F2eC} with \eqref{eq: eAz}, we get that 
$$
Y_{0}=Fe_C=F_{1}F_{2}e_C \in \GL_{d}(\R)
$$
is a fundamental matrix of solutions of \eqref{the mahler syst}, where  $F_{1},F_{2},F,C$ are given by Theorem~\ref{theo: fund syst sol princ} and $e_{C}$ is given by Lemma~\ref{lem:sys_const}. 

\subsection{Proof of Theorem~\ref{thm:decomp_chi}}\label{sec:proof_thm_decomp_chi}
Consider the $\ellmahl$-Mahler system  
\begin{equation}\label{syst associated eq}
\malop{\ellmahl}(Y)=AY 
\end{equation}
associated with \eqref{eq: mahler1}, where $A$ is given by \eqref{eq:syst assoc eq}. 
We recall that the coefficients $a_i$ of \eqref{eq: mahler1} and, hence, the entries of $A$, belong to $\Kinf=\Q(z^{\frac{1}{\starpourexp}})$.

Section~\ref{sec:fms}, applied to \eqref{syst associated eq}, and whose notations we will use freely, provides us with a fundamental matrix of solutions $F_{1}F_{2}e_{C}$ of \eqref{syst associated eq}.  
Note that it follows from \eqref{eq:F1Theta=AF1} that $\malop{\ellmahl}(F_1)=AF_1\Theta^{-1}$; as $A$ and $\Theta$ have coefficients in $\Kinf$, this implies that the finite dimensional $\Kinf$-vector space spanned by the entries of $F_1$ is invariant under $\malop{\ellmahl}$ and, hence, that the entries of $F_1$ are $\ellmahl$-Mahler Puiseux series\footnote{We emphasize that we have used the fact that the coefficients of equation~\eqref{eq: mahler1} lie in $\Kinf$. If we only assume that they belong to $\puis$, then, in general, the entries of $F_1$ would no longer be $\ellmahl$-Mahler. }. 
Now, the entries $y_1,\ldots,y_d$ of the first row of $F_{1}F_{2}e_{C}$ are $\Qbar$-linearly independent solutions of \eqref{eq: mahler1}. Since the entries of $F_{1}$ (resp.~$F_{2}$) are $\ellmahl$-Mahler Puiseux series (resp.~elements of $\V$) and since the entries of $e_{C}$ are $\Qbar$-linear combinations of the $e_c\logm^j$ with $c \in \operatorname{Sp}(C)$ and $j \in \{0,\ldots,d-1\}$ (see Lemma~\ref{lem:sys_const}), we see that  
$$
y_i = \sum_{c \in \operatorname{Sp}(C)}\sum_{j \in \{0,\ldots,d-1\}}f_{i,c,j}e_c\logm^j 
$$
where the $f_{i,c,j}$ are finite sums of products of a $\ellmahl$-Mahler Puiseux series by an element of $\V$. It follows clearly from this and from the fact that any element of $\V$ is a $\puisinfzero$-linear combination of some $\xi_{\tuple}$ with $\tuple \in \ala$ that the $y_{i}$ are generalized $\ellmahl$-Mahler series. This concludes the proof.

\subsection{On the unicity of the fundamental matrix of solutions}\label{sec:unicity fms}
Section~\ref{sec:fms} provides us with a fundamental matrix of solutions $Fe_{C}=F_{1}F_{2}e_{C}$ of \eqref{syst associated eq}. Consider another fundamental matrix of solutions of \eqref{syst associated eq} of the form $Ge_D$ with $G \in \GL_d(\Hahn)$ and $D\in \GL_{d}(\Qbar)$. We claim that there exists $R \in \GL_{d}(\Qbar)$ such that
	\begin{itemize}
		\item[\rm (i)]$D=R^{-1}CR$,
		\item[\rm (ii)] $G=F_{1}G_2$ with $G_2=F_2R \in \GL_d(\V_{d-1})$.	
	\end{itemize} 
Indeed, since  $Fe_C$ and $Ge_D$ are fundamental matrices of solutions of the same system~\eqref{the mahler syst}, we have that $R=F^{-1}G \in \GL_d(\Hahn)$ satisfies $\malop{\ellmahl}(R)D=CR$.  Setting $R=\sum_{\gamma \in \QQ} R_{\gamma} z^{\gamma}$, we obtain, for all $\gamma \in \QQ$, $R_{\frac{\gamma}{\ellmahl}}D=CR_{\gamma}$. This implies that the support $\supp(R)$ of $R$ is invariant by multiplication by $\ellmahl$ and $\ellmahl^{-1}$. Since it is well-ordered, this implies that $\supp(R) \subset \{0\}$, {\it i.e.}, $R \in \GL_{d}(\Qbar)$ and our claim follows immediately.

\section{Standard decomposition}\label{sec:standard}
Recall that a generalized $\ellmahl$-Mahler series is an element of $\mathcal R$ of the form~\eqref{eq:generalized_mahler_series_bis}, {\it i.e.}, of the form 
	$$
\sum_{ (c,j) \in  \cj}\sum_{\tuple\in \ala} 
f_{c,j,\tuple} \xi_{\tuple}e_{c} \logm^{j}	
	$$
where $\cj=\Qbar^{\times}\times\Z_{\geq 0}$, the sums have finite support and each $f_{c,j,\tuple} \in \puis$ is a $\ellmahl$-Mahler Puiseux series.
The decomposition of a generalized $\ellmahl$-Mahler series into a $\Hahn$-linear combination of the $e_{c} \logm^{j}$ is unique, but this is not the case of the decomposition~\eqref{eq:generalized_mahler_series_bis} into a $\puis$-linear combination of the $\xi_{\tuple} e_{c} \logm^{j}$. For exemple, one has
$$
\xi_{(0),(1),(\ellmahl))}=z^{-1}+\xi_{((0),(1),(1))}.
$$
The following definition and proposition remedy this problem.

\begin{defi}\label{defi:gen p-mahl}
We will say that a  decomposition of the form \eqref{eq:generalized_mahler_series_bis}
 is standard if the tuples $\tuple$ involved in the support of the sum in \eqref{eq:generalized_mahler_series_bis} belong to the set 
$\alaNp$ defined as follows: 
$$
	\alaNp=\bigcup_{t \in \Z_{\geq 0}} \alaNpt{t}	
\text{ \ \ \  with \ \ \ }
\alaNpt{t} = \Z_{\geq 0}^{t} \times (\Qbar^{\times})^{t} \times \N_{(\ellmahl)}^{t}, 
$$
where $\N_{(\ellmahl)}$ is the set of positive rational numbers whose denominators are relatively prime with $\ellmahl$ and whose numerators are not divisible by $\ellmahl$. 
\end{defi}

\begin{prop}\label{prop:gen p-mahl}
Any generalized $\ellmahl$-Mahler series $f$ has a unique standard decomposition. Furthermore, the Puiseux series involved in the standard decomposition of $f$ are $\puisinfzero$-linear combinations of the Puiseux series $f_{c,j,\tuple}$ involved in any decomposition of $f$ of the form \eqref{eq:generalized_mahler_series_bis}. 
\end{prop}

Moreover, the fact that a generalized $\ellmahl$-Mahler series satisfies the growth condition $\condpuisprime{\indC}$ introduced in Definition~\ref{defi:def Cr gms} can be inferred from its standard decomposition, as shown by the following result. 

\begin{prop}\label{prop:def Cr gms}
A generalized $\ellmahl$-Mahler series 
satisfies $\condpuisprime{\indC}$ if and only if the Puiseux series 
involved in its standard decomposition satisfy $\condprime{\indC}$.
\end{prop}

The previous two results are proved in the next two sections.

\subsection{Proof of Proposition~\ref{prop:gen p-mahl}}\label{sec:proof prop:gen p-mahl}
Proposition~\ref{prop:gen p-mahl} states that any generalized $\ellmahl$-Mahler series has a unique standard decomposition.
The uniqueness is proved in Section~\ref{sec:proof uniqueness decomp sol}, the existence is proved in Section~\ref{sec:exist std decomp}.

\subsubsection{Uniqueness}\label{sec:proof uniqueness decomp sol}
To prove the uniqueness of the standard decomposition in Proposition~\ref{prop:gen p-mahl}, it is sufficient to prove that the family 
 $$
 (\xi_{\tuple}e_c\logm^j)_{\tuple \in \alaNp,(c,j)  \in \cj}
 $$
 is $\puis$-linearly independent.   
Since, in virtue of  \cite[Lemma~30]{RoquesFrobForMahler}, the family 
 $
 (e_c\logm^j)_{(c,j) \in \cj}
 $
 is $\Hahn$-linearly independent, it is sufficient to prove that the family 
 $
 (\xi_{\tuple})_{\tuple \in \alaNp}
 $
 is $\puis$-linearly independent.  To prove this, we argue by contradiction: we assume that 
 $
 (\xi_{\tuple})_{\tuple \in \alaNp}
 $ 
 is $\puis$-linearly dependent, {\it i.e.}, that there exist a positive integer $r \in \Z_{\geq 1}$, pairwise distinct parameters $\tuple_1,\ldots,\tuple_r\in \alaNp$ and Puiseux series $f_1,\ldots,f_{r} \in \puis^{\times}$ such that 
 \begin{equation}\label{lem:combi lin}
	\sum_{i=1}^r f_i\xi_{\tuple_i}=0. 
\end{equation}		
	 For any $i \in \{1,\ldots,r\}$, we write $\tuple_i=(\balpha_i,\lambd_i,\a_i)$ with
	$\balpha_i=(\alpha_{i,1},\ldots,\alpha_{i,t_{i}})$, $\lambd_i=(\lambda_{i,1},\ldots,\lambda_{i,t_{i}})$ and $\a_i=(a_{i,1},\ldots,a_{i,t_i})$. 
	Up to renumbering, we may suppose that $t_1\leq t_2 \leq \cdots \leq t_r$. 
	
	If $t_r=0$, then $r=1$ and $\tuple_1=((),(),())$. So, \eqref{lem:combi lin} reduces to $f_{1}\xi_{\tuple_1}=f_{1}=0$, whence a contradiction.

	We now assume that $t_r\geq 1$. 
	Up to multiplying \eqref{lem:combi lin} by some power of $z$, we may suppose that the $z$-adic valuation of $f_r$ is $0$. Let $d \in \Z_{\geq 1}$ be such that the supports of $f_{1},\ldots,f_{r}$ are included in $\frac{1}{d}\Z$. It follows from Lemma~\ref{lem:unicite_exposants} below that there exists $C>0$ (for instance, $C=\max\{C_{\a_{r},\a_{1},d},\ldots,C_{\a_{r},\a_{r},d}\}$ works, where the constants $C_{\a_{r},\a_{1},d},\ldots,C_{\a_{r},\a_{r},d}$ are given by  Lemma~\ref{lem:unicite_exposants}) such that, for any $\gamma \in \frac{1}{d}\Z$, for any $i \in \{1,\ldots,r-1\}$, for any $k_{1},\ldots,k_{t_{r}},\ell_{1},\ldots,\ell_{t_{i}} \in \Z_{\geq 1}$ such that $k_{1},\ldots,k_{t_{r}} \geq C$, the equality
	$$
	\frac{a_{r,1}}{\ellmahl^{k_1}}+\cdots+ \frac{a_{r,t_r}}{\ellmahl^{k_1+\cdots+k_{t_r}}} =\gamma+\frac{a_{i,1}}{\ellmahl^{\ell_1}}+\cdots+ \frac{a_{i,t_i}}{\ellmahl^{\ell_1+\cdots+\ell_{t_i}}} 
	$$
holds if and only if $\gamma=0$, $t_{r}=t_{i}$, $\a_{r}=\a_{i}$ and $(k_{1},\ldots,k_{t_{r}})=(\ell_{1},\ldots,\ell_{t_{r}})$. 
This implies that, for any $k_1,\ldots,k_{t_r} \in \Z_{\geq C}$, the coefficient of
\begin{equation*}\label{eq:puiss z}
 z^{-\frac{a_{r,1}}{\ellmahl^{k_1}}-\cdots- \frac{a_{r,t_r}}{\ellmahl^{k_1+\cdots+k_{t_r}}}}
\end{equation*}
in $f_i\xi_{\tuple_i}$ is equal to 
\begin{itemize}
 \item $c_{i} k_1^{\alpha_{i,1}}\cdots k_{t_{r}}^{\alpha_{i,t_{r}}}\lambda_{i,1}^{k_1}\cdots\lambda_{i,t_{r}}^{k_1+\cdots+k_{t_{r}}}$ if $\a_i = \a_r$ where $c_i$ is the constant coefficient of $f_i$; 
 \item $0$ if $\a_i \neq \a_r$. 
\end{itemize}
In virtue of~\eqref{lem:combi lin}, the sum of these coefficients is equal to $0$, so, 
for any $k_1,\ldots,k_t \in \Z_{\geq C}$, we have
	\begin{equation}\label{eq:exp_poly_nul}
\sum_{i\in \I} c_i k_1^{\alpha_{i,1}}\cdots k_t^{\alpha_{i,t}}\lambda_{i,1}^{k_1}\cdots\lambda_{i,t}^{k_1+\cdots+k_t}=0
	\end{equation}
where $\I$ is the set of $i\in \{1,\ldots,r\}$ such that $\a_i=\a_r$ and where $t=t_{r}$ is the common value of the $t_{i}$ for $i \in \I$.
Since the $2t$-tuples $(\balpha_{i},\lambd_{i})$ are pairwise distinct when $i$ ranges over $\I$, it follows from \cite[Lemma~2.2]{Schmidt03} that the family  $$\left((k_1^{\alpha_{i,1}}\cdots k_t^{\alpha_{i,t}}\lambda_{i,1}^{k_1}\cdots\lambda_{i,t}^{k_1+\cdots+k_t})_{(k_{1},\ldots,k_{t})\in \Z^t}\right)_{i \in \I}$$ is $\Qbar$-linearly independent. Actually, a straightforward adaptation of the proof of \cite[Lemma~2.2]{Schmidt03} implies that the family $$\left((k_1^{\alpha_{i,1}}\cdots k_t^{\alpha_{i,t}}\lambda_{i,1}^{k_1}\cdots\lambda_{i,t}^{k_1+\cdots+k_t})_{(k_{1},\ldots,k_{t})\in \Z_{\geq C}^t}\right)_{i \in \I}$$ is $\Qbar$-linearly independent. This contradicts \eqref{eq:exp_poly_nul} because $c_r\neq 0$.

To complete the proof, it only remains to prove the following result.
\begin{lem}\label{lem:unicite_exposants}
  For all $s,t \in \Z_{\geq 0}$ such that $s \geq t \geq 0$, for all $\a=(a_1,\ldots,a_s) \in \N_{(\ellmahl)}^{s}$, for all $\b=(b_1,\ldots,b_t) \in \N_{(\ellmahl)}^{t}$, for all $d \in \Z_{\geq 1}$, there exists $C_{\a,\b,d}>0$ such that, for all $\gamma \in \frac{1}{d} \ZZ$, for all $k_1,\ldots,k_s,\ell_1,\ldots,\ell_t\in \Z_{\geq 1} $ such that $k_1,\ldots,k_s\geq C_{\a,\b,d}$, the following properties are equivalent:
\begin{enumerate}
 \item we have 
 \begin{equation}\label{eq:egalite_puissances_Chi}
	\frac{a_1}{p^{k_1}} + \frac{a_2}{p^{k_1+k_2}}  + \cdots + \frac{a_s}{p^{k_1+k_2+\cdots+k_s}} = \gamma+ \frac{b_1}{p^{\ell_1}} + \frac{b_2}{p^{\ell_1+\ell_2}}  + \cdots + \frac{b_t}{p^{\ell_1+\ell_2+\cdots+\ell_t}}; 
\end{equation}
 \item we have $\gamma=0$, $s=t$ and, for all  $i \in \{1, \ldots, t\}$,  $a_i=b_i$ and $k_i=\ell_i$. 
\end{enumerate}
\end{lem}

\begin{proof} 
We first consider the case $t=0$.	
Let $C_{\a,d} \in \Z_{\geq 1}$ be such that 
$$
\frac{a_{1}+\cdots+a_{s}}{p^{C_{\a,d}}} <\frac{1}{d}.
$$
Consider $k_1,\ldots,k_s\in \Z_{\geq 1} $ such that $k_1,\ldots,k_s\geq C_{\a,d}$ and satisfying \eqref{eq:egalite_puissances_Chi}. The left-hand side of \eqref{eq:egalite_puissances_Chi} belongs to $[0,\frac{1}{d})$ and its right-hand side belongs to $\frac{1}{d}\Z$. Therefore, both sides of \eqref{eq:egalite_puissances_Chi} belong to $[0,\frac{1}{d}) \cap \frac{1}{d}\Z=\{0\}$, whence $s=0$ and $\gamma=0$. This concludes the proof in the case $t=0$. 
\vskip 5pt
Let us now turn to the proof of the result in the general case. We argue by induction on $n=s+t$. 

\vskip 5 pt
\noindent\textbf{Base case: $n=0$ or $n=1$.} The result holds when $n=0$ or $n=1$ for, in that case, we have $t=0$. 

\vskip 5 pt
\noindent\textbf{Inductive step.}
Consider an integer $n \geq 2$ and assume that the result is true for all $s\geq t \geq 0$ such that $s+t \leq n-1$. Consider $s\geq t \geq 0$ such that $s+t = n$ and let us prove that the result is true for these $s,t$. We already know that the result is true if $t=0$, so we can and will assume that $t \geq 1$. We distinguish two cases: either $a_{s} = b_{t}$ or $a_{s} \neq b_{t}$. 

\vskip 5pt
{\it Case $a_{s} = b_{t}$.} Assume that $a_s=b_t$. Let $D \in \Z_{\geq 1}$ be such that $\ellmahl$ is coprime with the denominator of $p^{D}\frac{1}{d}$.  Set $\widecheck{\a}=(a_{1},\ldots,a_{s-1})$ and $\widecheck{\b}=(b_{1},\ldots,b_{t-1})$. Consider the constant $C_{\widecheck{\a},\widecheck{\b},d}$ given by the induction hypothesis. Set $C_{\a,\b,d}=\max \{D+1,C_{\widecheck{\a},\widecheck{\b},d}\}$. 

Consider $\gamma \in \frac{1}{d} \ZZ$ and $k_1,\ldots,k_s,\ell_1,\ldots,\ell_t\in \Z_{\geq 1}$ such that $k_1,\ldots,k_s\geq C_{\a,\b,d}$ and satisfying \eqref{eq:egalite_puissances_Chi}. 
We rewrite the equality \eqref{eq:egalite_puissances_Chi} as follows
\begin{equation}\label{eq:egalite_puissances_Chi bis}
	\frac{a_1}{p^{k_1}} + \frac{a_2}{p^{k_1+k_2}}  + \cdots + \frac{a_s}{p^{k_1+k_2+\cdots+k_s}} -\gamma=  \frac{b_1}{p^{\ell_1}} + \frac{b_2}{p^{\ell_1+\ell_2}}  + \cdots + \frac{b_t}{p^{\ell_1+\ell_2+\cdots+\ell_t}}.
\end{equation}
The left-hand side of \eqref{eq:egalite_puissances_Chi bis} belongs to $\frac{1}{p^{k_1+k_2+\cdots+k_s}} \N_{(\ellmahl)}$ whereas its right-hand side belongs to $\frac{1}{p^{\ell_1+\ell_2+\cdots+\ell_t}} \N_{(\ellmahl)}$. 
So, we have 
\begin{equation}\label{eq:k_i ell_i}
k_1+k_2+\cdots+k_s = \ell_1+\ell_2+\cdots+\ell_t
\end{equation}
and 
\eqref{eq:egalite_puissances_Chi} can be rewritten as follows: 
$$
\frac{a_1}{p^{k_1}} + \frac{a_2}{p^{k_1+k_2}}  + \cdots + \frac{a_{s-1}}{p^{k_1+k_2+\cdots+k_{s-1}}} = \gamma+ \frac{b_1}{p^{\ell_1}} + \frac{b_2}{p^{\ell_1+\ell_2}}  + \cdots + \frac{b_{t-1}}{p^{\ell_1+\ell_2+\cdots+\ell_{t-1}}}.
$$
Since $k_1,\ldots,k_{s-1} \geq C_{\widecheck{\a},\widecheck{\b},d}$, we have $\gamma=0$, $s-1=t-1$, $a_{i}=b_{i}$ and $k_{i}=\ell_{i}$ for all $i \in \{1, \ldots, t-1\}$. 
It follows from \eqref{eq:k_i ell_i} that $k_{s}=\ell_{s}$ as well. This concludes the induction step in this case. 

\vskip 5pt
{\it Case $a_{s} \neq b_{t}$.} Assume now that $a_s \neq b_t$. Let $L\in \Z_{\geq 0}$ be such that $p^{L}$ is the greatest power of $p$ dividing the numerator of $a_{s}-b_{t}$.  Let $D \in \Z_{\geq 1}$ be such that the numerators of $p^{D}\frac{1}{d}$, $p^{D}a_{1},\ldots,p^{D}a_{s-1}$ are divisible by $p^{L}$.  If $t=1$, we set $C_{\a,\b,d}=D+1$. If $t \geq 2$, we set $C_{\a,\b,d}=\max \{D+1,C_{\a,\b_{0},d},\ldots,C_{\a,\b_{L},d}\}$ where $\a=(a_{1},\ldots,a_{s})$,  
$\b_{\ell}=(b_{1},\ldots,b_{t-2},p^{\ell}b_{t-1}+b_{t})$ and $C_{\a,\b_{\ell},d}$ is the constant given by the induction hypothesis.

We claim that there is no $\gamma \in \frac{1}{d} \ZZ$ and $k_1,\ldots,k_s,\ell_1,\ldots,\ell_t\in \Z_{\geq 0}$ such that $k_1,\ldots,k_s\geq C_{\a,\b,d}$ satisfying \eqref{eq:egalite_puissances_Chi}. We argue by contradiction: we assume that such $\gamma,k_1,\ldots,k_s,\ell_1,\ldots,\ell_t$ exist. Arguing as in the case $a_{s}=b_{t}$ treated above, we see that 
$$
k_1+k_2+\cdots+k_s = \ell_1+\ell_2+\cdots+\ell_t.
$$
Multiplying both sides of \eqref{eq:egalite_puissances_Chi} by $\ellmahl^{k_1+\cdots+k_s}=\ellmahl^{\ell_1+\cdots+\ell_{t} }$ and substracting $b_t$, we obtain
\begin{equation}\label{eq:egality_num bis}
(a_s-b_{t})+\ellmahl^{k_2+\cdots+k_s}a_1 + \cdots + \ellmahl^{k_s}a_{s-1}-\ellmahl^{k_1+\cdots+k_s}\gamma=\ellmahl^{\ell_2+\cdots+\ell_t}b_1 +  \cdots +  \ellmahl^{\ell_t}b_{t-1}.
\end{equation}
Note that our choice of $C_{\a,\b,d}$ guarantees that $p^{L+1}$ divides the numerator of $\ellmahl^{k_2+\cdots+k_s}a_1 + \cdots + \ellmahl^{k_s}a_{s-1}-\ellmahl^{k_1+\cdots+k_s}\gamma $ but not the numerator of $a_{s}-b_{t}$. 
If $t=1$, this immediately leads to a contradiction with \eqref{eq:egality_num bis} which, in this case, reduces to 
\begin{equation*}\label{eq:egality_num bis bis}
0=(a_s-b_{1})+\ellmahl^{k_2+\cdots+k_s}a_1 + \cdots + \ellmahl^{k_s}a_{s-1}-\ellmahl^{k_1+\cdots+k_s}\gamma.
\end{equation*}
When $t \geq 2$, this implies that $\ell_{t} \in \{0,\ldots,L\}$. 
Rewriting \eqref{eq:egalite_puissances_Chi} as follows
$$
\frac{a_1}{p^{k_1}}
  + \cdots + \frac{a_s}{p^{k_1+k_2+\cdots+k_s}} = \gamma+ \frac{b_1}{p^{\ell_1}} 
   + \cdots +\frac{b_{t-2}}{p^{\ell_1+\ell_2+\cdots+\ell_{t-2}}}+ \frac{p^{\ell_{t}}b_{t-1}+b_t}{p^{\ell_1+\ell_2+\cdots+\ell_t}}
$$
and using the fact that $k_1,\ldots,k_s\geq C_{\a,\b_{\ell_{t}},d}$, we get that $s=t-1$, which is absurd because $t \leq s$. 
\end{proof}

\subsubsection{Existence}\label{sec:exist std decomp}  For any integer $s\geq 0$, we consider the $\Q$-vector space
\begin{align*}
\mathcal W_s=\operatorname{Span}_{\Qbar}\left\{z^{-\gamma}\xi_{\tuple} \ \vert \ \gamma \in \QQ_{\geq 0},\,
\tuple \in  \bigcup_{t=0}^{s}\alaNpt{t}
\right\}.
\end{align*}
To prove the existence of the standard decomposition stated in Proposition~\ref{prop:gen p-mahl} and the last assertion of this proposition, it is clearly sufficient to prove that, for any $\tuple=(\balpha,\lambd,\a)\in \ala$, we have $\xi_\tuple \in \mathcal W_s$ where $s \geq 0$ is the unique integer such that $\tuple\in \alat{s}$. Let us prove this by well-founded induction on $\balpha$ with respect to the partial well-order~$\prec$ on $\bigcup_{t \geq 0} \Z_{\geq 0}^{t}$ introduced in Section~\ref{sec: the xi are p mahler}. 
So, we consider an arbitrary $\tuple=(\balpha,\lambd,\a) \in \ala$ and we suppose that, for any $\tuple'=(\balpha',\lambd',\a') \in \ala$ such that $\balpha' \prec \balpha$, we have $\xi_{\tuple'} \in \mathcal W_{s'}$ where $s' \geq 0$ is the unique integer such that $\tuple' \in \alat{s'}$. If $\tuple=((),(),()) \in \alat{0}$, then $\xi_\tuple=1$ clearly belongs to $\mathcal W_{0}$, as desired. 
So, we can assume that $\tuple \neq ((),(),())$. 
Let $s \geq 1$ be the unique integer such that $\tuple \in \alat{s}$. Recall from \eqref{eq:def_xi_summable} that
\begin{equation}\label{eq:def_xi_summable_bis}
\xi_{\tuple} = \sum_{k\geq 1}k^{\alpha_1}(\lambda_1\cdots\lambda_t)^k \malop{\ellmahl}^{-k}(z^{-a_1}\xi_{\tuplespt})
\end{equation}
where $\tuplespt=(\balphaspt,\lambdspt,\aspt) \in \alat{s-1}$ denotes the tuple obtained from $\tuple=(\balpha, \lambd, \a)$ by removing the first coordinate of $\balpha$, $\lambd$ and $\a$ as in Section~\ref{sec: xi can be def induct}.
Since $\balphaspt \prec \balpha$, the induction hypothesis ensures that $\xi_{\tuplespt} \in \mathcal W_{s-1}$. So, by linearity, to prove that $\xi_{\tuple}$ belongs to $\mathcal W_{s}$, it is sufficient to prove that, for any $\alpha \in \Z_{\geq 0}$, any $\lambda \in \Qbar^\times$, any $\gamma \in \QQ_{>0}$ and any $\tuple' \in  \alaNpt{s'}$ with $s' \in \{0,\ldots,s-1\}$, 
\begin{equation}\label{eq:expression_sum_Mahler}
\sum_{k\geq 1}k^{\alpha}\lambda^k \malop{\ellmahl}^{-k}(z^{-\gamma}\xi_{\tuple'})
\end{equation}
belongs to $\mathcal W_s$. To prove this, we consider $\eta \in \N_{(\ellmahl)}$ and $u\in \Z$ such that $\gamma = \eta\ellmahl^u$ and we distinguish three cases.

\vskip 5pt
\textit{Case $u=0$.} In this case, we have $\gamma \in \N_{(\ellmahl)}$ and it follows from \eqref{eq:def_xi_summable_bis} that~\eqref{eq:expression_sum_Mahler} is equal to $\xi_{\tuple''}$ for some $\tuple'' \in \alaNpt{s'+1}$ and, hence, belongs to $\mathcal W_s$.

\vskip 5pt
\textit{Case $u>0$.}  Splitting the sum into the parts with $k\leq u$ and $k>u$, \eqref{eq:expression_sum_Mahler} may be rewritten as
\begin{equation}\label{eq:expression_sum_Mahler2}
\sum_{k=1}^uk^{\alpha}\lambda^k z^{-\eta \ellmahl^{u-k}} \xi_{\tuple'}(z^{\ellmahl^{-k}}) + \sum_{k\geq 1}(k+u)^{\alpha}\lambda^{k+u} \malop{\ellmahl}^{-k}(z^{-\eta}\xi_{\tuple'}(z^{\ellmahl^{-u}})).
\end{equation}
Using Lemma~\ref{lem:phichi-lamddchi}, we see that each $\xi_{\tuple'}(z^{\ellmahl^{-k}})$ is a $\Qbar$-linear combination of certain $\xi_{\tuple''}(z)$ with $\tuple'' \in \alaNpt{s'}$ and of certain $z^{-\theta}\xi_{\tuple''}(z)$ with $\tuple'' \in \alaNpt{s''}$ for some $s'' \in \{0,\ldots,s'-1\}$ and $\theta \in \QQ_{>0}$.
This immediately implies that the first sum of \eqref{eq:expression_sum_Mahler2} belongs to $\mathcal W_s$.  
Moreover, it also implies that the second sum of  \eqref{eq:expression_sum_Mahler2} is a $\Qbar$-linear combination of terms of one of the following forms
\begin{align}
\label{eq:sommation_part1} \sum_{k\geq 1}k^j\lambda^{k} \malop{\ellmahl}^{-k}(z^{-\eta}\xi_{\tuple''}(z))&\text{ with } \tuple'' \in \alaNpt{s'},\,j \in \Z_{\geq 0},\, \text{ or,}
\\
\label{eq:sommation_part2} \sum_{k\geq 1}k^j\lambda^{k} \malop{\ellmahl}^{-k}(z^{-\varphi}\xi_{\tuple''}(z))& \text{ with } \tuple'' \in \alaNpt{s''},\,s''<s',\,j \in \Z_{\geq 0}, \varphi\in \QQ_{<0}\,.
\end{align}
Using \eqref{eq:def_xi_summable_bis} and the fact that $\eta \in \N_{(\ellmahl)}$, we see that \eqref{eq:sommation_part1} is equal to $\xi_{\tuple'''}$ for some $\tuple''' \in \alaNpt{s'+1}$ and, hence, belongs to $\mathcal W_s$. Moreover, it follows from \eqref{eq:def_xi_summable_bis} that \eqref{eq:sommation_part2} is equal to $\xi_{\tuple'''}$ for some $\tuple'''\in \alat{s''+1}$. Since $s''+1\leq s'<s$, we have $\balpha'''\prec \balpha$ and our induction hypothesis guarantees that \eqref{eq:sommation_part2} belongs to~$\mathcal W_s$. Thus \eqref{eq:expression_sum_Mahler2} and, hence, \eqref{eq:expression_sum_Mahler} belong to $\mathcal W_s$.

\vskip 5pt
\textit{Case $u<0$.} In this case, setting $v=-u>0$, we may rewrite \eqref{eq:expression_sum_Mahler} as
$$
\sum_{k\geq 1}(k-v)^{\alpha}\lambda^{k-v} \malop{\ellmahl}^{-k}(z^{-\eta}\xi_{\tuple'}(z^{\ellmahl^{v}})) -\sum_{k=1}^v(k-v)^{\alpha}\lambda^{k-v} z^{-\eta \ellmahl^{-k}} \xi_{\tuple'}(z^{\ellmahl^{v-k}}). 
$$
A reasoning similar to that used in the case $u>0$ shows that it belongs to~$\mathcal W_s$. This concludes the inductive step and thus establishes the existence of a standard decomposition.

\subsection{Proof of Proposition~\ref{prop:def Cr gms}}\label{sec:proof_Cr_gms}

Only one of the implications requires a justification, namely that if a generalized $\ellmahl$-Mahler series $f$ satisfies $\condpuisprime{\indC}$, 
then the Puiseux series involved in its standard decomposition satisfy $\condprime{\indC}$ as well. So, we consider a decomposition of $f$ of the form \eqref{eq:generalized_mahler_series_bis} such that all the Puiseux series $f_{c,j,\tuple}$ satisfy $\condprime{\indC}$. Proposition~\ref{prop:gen p-mahl} ensures that the Puiseux series involved in the standard decomposition of $f$ are $\puisinfzero$-linear combinations of the $f_{c,j,\tuple}$. Since the set of $\ellmahl$-Mahler Puiseux series satisfying $\condprime{\indC}$ is a $\puisinfzero$-module, we get that the Puiseux series involved in the standard decomposition of $f$ satisfy $\condprime{\indC}$ as well. 

\section{Sketch of the proof of Theorem~\ref{thm:pureté1}}\label{sec:sketch proof purity theo}

In this section, we provide a sketch of the proof of Theorem~\ref{thm:pureté1}, which is presented in full detail in Section~\ref{sec:proof purity theo}. 

We do not elaborate on the case $r=1$ of Theorem~\ref{thm:pureté1}, as it follows directly from Theorem~\ref{thm:decomp_chi} and Theorem~\ref{thm: hgt}. 

To prove Theorem~\ref{thm:pureté1} when $r=2$ (resp.~$r=3$), we will first consider the case when
\begin{equation*}
f
=
\sum_{(c,j) \in \cj}\sum_{\tuple\in \ala}f_{c,j,\tuple}\xi_{\tuple}e_c \logm^{j}  
\end{equation*}
where the $f_{c,j,\tuple}$ belong to $\Qbar((z))$ and the $\tuple=(\balpha,\lambd,\a)$ appearing in the support of the above sum are such that $\a$ has entries in $\Z_{> 0}$. To address this particular case, we will proceed as follows:
\begin{enumerate}[label=(\roman*)]
 \item \label{step 1 proof purity cas 1} we will first prove that $f$ satisfies a $\ellmahl$-Mahler equation of the form~\eqref{eq: mahler1} with coefficients $a_{0},\ldots,a_{d} \in \Qbar[z]$ such that the nonzero roots of $a_{0}$ belong to the set $\mathcal U$  of complex roots of unity (resp.~to the set $\mathcal U_\ellmahl$ of complex roots of unity whose order is not coprime with $\ellmahl$); 
 \item \label{step 2 proof purity cas 1} we will then prove that such an equation has a full basis of generalized $\ellmahl$-Mahler series solutions satisfying the growth condition $\condpuisprime{2}$ (resp.~$\condpuisprime{3}$).
\end{enumerate}
To establish Theorem~\ref{thm:pureté1} for \( r \in \{2, 3\} \) in the general case, 
we will apply the special case above to \( f(z^{\mu}) \), for a suitably chosen \( \mu \in \mathbb{Z}_{\geq 1} \).

The justifications of properties~\ref{step 1 proof purity cas 1} and \ref{step 2 proof purity cas 1} above will rely on the results established in Section~\ref{sec:mahler denom} regarding $\ellmahl$-Mahler denominators.  
More specifically: 
\begin{itemize}
 \item we will derive \ref{step 1 proof purity cas 1} as a direct application of Proposition~\ref{prop growth denom}, which is the main result of Section~\ref{sec: ext only if part}; 
 \item we will derive \ref{step 2 proof purity cas 1} as a direct application of Proposition~\ref{theo avec a0}, which is the main result of  Section~\ref{sec: ext if part}. 
\end{itemize}

\section{$\ellmahl$-Mahler denominators}\label{sec:mahler denom}

\subsection{Reminders on the $\ellmahl$-Mahler denominator of $\ellmahl$-Mahler Laurent series}\label{sec: mahler den laurent series}

Theorem~\ref{thm: hgt} ensures that any $\ellmahl$-Mahler Laurent series $f \in \Qbar((z))$ satisfies the growth condition $\condprime{1}$. According to \cite{BorisBellSmertnigGap}, one can determine whether or not it satisfies one of the stronger conditions $\condprime{2}$ or $\condprime{3}$ by looking at its $\ellmahl$-Mahler denominator. Let us briefly recall this.

\begin{defi}
 The $\ellmahl$-Mahler denominator $\denommahl{f}$ of a $\ellmahl$-Mahler Laurent series $f \in \Q((z))$ is the monic generator of the ideal of $\Qbar[z]$ given by 
\begin{equation*}\label{ideal p-mahler}
 \left\{P \in \Qbar[z] \ \big\vert  \ Pf \in \sum_{i=1}^{d} \Q[z]\malop{\ellmahl}^i(f) \text{ for some } d \in \Z_{\geq 1} \right\}.
\end{equation*}
\end{defi}

\begin{rem}\label{rem: coeff Kinf vs ratio}
1) In \cite{BorisBellSmertnigGap}, the previous definition is formulated for $\ellmahl$-Mahler series $f \in \Qbar[[z]]$. Its extension to $\Qbar((z))$ is straightforward.  

2) The $\ellmahl$-Mahler equations considered in \cite{BorisBellSmertnigGap} have coefficients in $\Q(z)$, whereas the equations considered in the present paper have coefficients in the bigger field $\Kinf$. Note that $f \in \Q((z))$ satisfies a $\ellmahl$-Mahler equation with coefficients in $\Q(z)$ if and only if it satisfies a $\ellmahl$-Mahler equation with coefficients in $\Kinf$. 
Indeed, assume  that $f \in \Q((z))$ satisfies a  $\ellmahl$-Mahler equation of the form \eqref{eq: mahler1} with $a_0,\ldots,a_d \in \Qbar(z^{\frac{1}{m}})$ for some $m \in \Z_{\geq 1}$. Without loss of generality, we can assume that $a_{0}=1$. Let $G$ be the Galois group of the finite Galois extension $\Qbar((z^{\frac{1}{m}}))$ of $\Qbar((z))$. Then, $f$ satisfies 
 $$
 b_0 f + b_1 \malop{\ellmahl}(f)+ \cdots + b_d \malop{\ellmahl}^d(f)=0
 $$
 with $b_{i}=\frac{1}{\vert G\vert}\sum_{\sigma \in G} \sigma(a_{i}) \in \Qbar(z)$ and this equation is non trivial because $b_{0}=1$. 
\end{rem}

In what follows, we denote by $\mathcal{U}$ the set of complex roots of unity and by $\mathcal{U}_{\ellmahl}$ the set of complex roots of unity whose order is not coprime with $\ellmahl$. We recall the following result.

\begin{theo}[{\cite[Theorems~6.1 \& 7.1]{BorisBellSmertnigGap}}] \label{theo abs 7.1}
	Consider a $\ellmahl$-Mahler Laurent series $f\in \Qbar((z))$. We have:
	\begin{itemize}
		\item $f$ satisfies $\condprime{2}$ 
		if and only if every nonzero root of the $\ellmahl$-Mahler denominator of $f$ belongs to $\mathcal{U}$;
		\item $f$ satisfies $\condprime{3}$ 
		if and only if every nonzero root of the $\ellmahl$-Mahler denominator of $f$ belongs to $\mathcal{U}_{\ellmahl}$.
	\end{itemize}
\end{theo}

\begin{rem}
In \cite{BorisBellSmertnigGap}, the previous result is stated for any $\ellmahl$-Mahler series $f \in \Q[[z]]$. 
Theorem~\ref{theo abs 7.1} follows from \cite{BorisBellSmertnigGap} by using the following remark. 
Consider $g=z^{\nu} f$ with $\nu \in \Z_{\geq 0}$ large enough so that $g \in \Q[[z]]$. Then, $f$ satisfies $\condprime{2}$ (resp.~$\condprime{3}$) if and only if $g$ satisfies the same property. Moreover, the $\ellmahl$-Mahler denominators of $f$ and $g$ differ
by multiplication by a power of $z$, so every nonzero root of the $\ellmahl$-Mahler denominator of $f$ belongs to $\mathcal{U}_{\ellmahl}$ (resp.~$\mathcal{U}$) if and only if the $\ellmahl$-Mahler denominator of $g$ satisfies the same property. Now, Theorem~\ref{theo abs 7.1} follows from \cite[Theorems~6.1 \& 7.1]{BorisBellSmertnigGap}  applied to the $\ellmahl$-Mahler series $g$.
\end{rem}

\subsection{$\ellmahl$-Mahler denominator of generalized $\ellmahl$-Mahler series}

We extend the definition of $\ellmahl$-Mahler denominator to generalized $\ellmahl$-Mahler series in the following obvious way. 

\begin{defi}
The $\ellmahl$-Mahler denominator $\denommahl{f}$ of a generalized $\ellmahl$-Mahler series $f\in\Rspe$ is the monic generator of the ideal of $\Qbar[z]$ given by
\begin{equation}\label{ideal p-mahler gpms}
 \left\{P \in \Qbar[z] \ \big \vert \ Pf \in \sum_{i=1}^{d} \Q[z]\malop{\ellmahl}^i(f) \text{ for some } d \in \Z_{\geq 1} \right\}
\end{equation}
if the latter ideal is non trivial; otherwise, we set $\denommahl{f}=0$. 
\end{defi}

\begin{rem}
1) Contrary to the case when $f \in \Qbar((z))$ considered in Section~\ref{sec: mahler den laurent series}, the ideal \eqref{ideal p-mahler gpms} may be trivial and, hence, the $\ellmahl$-Mahler denominator $\denommahl{f}$ may be equal to $0$. For instance, this is the case for $f(z)=z^{\frac{1}{\ellmahl}}$.

2) For any generalized $\ellmahl$-Mahler series $f$ having a decomposition of the form \eqref{eq:generalized_mahler_series_bis} with the $f_{c,j,\tuple}$ in $\Qbar((z))$, the ideal \eqref{ideal p-mahler gpms} is non trivial. This can be proved using arguments involved in the proof of Proposition~\ref{prop growth denom}. As this will not be used in this paper, we do not include the details. 
\end{rem}

Since some generalized $\ellmahl$-Mahler series have a null $\ellmahl$-Mahler denominator, one cannot expect Theorem~\ref{theo abs 7.1} to extend to any generalized $\ellmahl$-Mahler series. However, certain partial extensions of Theorem~\ref{theo abs 7.1} remain valid and will be used in the proof of our purity theorem. The remainder of this section is devoted to establishing these extensions.

\subsubsection{Extension of the ``if'' part of Theorem~\ref{theo abs 7.1}}\label{sec: ext if part}

\begin{prop}\label{theo avec a0}
If the coefficients $a_0,\ldots,a_d$ of the $\ellmahl$-Mahler equation \eqref{eq: mahler1} belong to $\Qbar[z]$, 
then it has $d$  generalized $\ellmahl$-Mahler series solutions that are $\Qbar$-linearly independent and satisfy 
	\begin{itemize}
		\item $\condpuisprime{2}$ if the nonzero roots of $a_{0}$ belong to $\mathcal{U}$;  
		\item $\condpuisprime{3}$ if the nonzero roots of $a_{0}$ belong to $\mathcal{U}_{\ellmahl}$. 
	\end{itemize}
In particular, for any generalized $\ellmahl$-Mahler series $f$ with nonzero $\ellmahl$-Mahler denominator $\denommahl{f}$,
\begin{itemize}
	\item $f$ satisfies $\condpuisprime{2}$ if the nonzero roots of $\denommahl{f}$ belong to $\mathcal{U}$;  
	\item $f$ satisfies $\condpuisprime{3}$ if the nonzero roots of $\denommahl{f}$ belong to $\mathcal{U}_{\ellmahl}$. 
\end{itemize}
\end{prop}

The last part of the proposition follows immediately form the first statement whose proof is given below, after the following lemma.

\begin{lem}\label{lem:denomMahler}
	Let $f$ and $g$ be generalized $\ellmahl$-Mahler series, whose $\ellmahl$-Mahler denominators are denoted by $\denommahl{f}$ and $\denommahl{g}$ respectively, and suppose that 	
	\begin{equation}\label{delta g}
	\efrak  g = \sum_{i=1}^{n} a_i \malop{\ellmahl}^i(g)+ f 
	\end{equation}
	for some $\efrak \in \Qbar[z] \setminus \{0\}$ and some $a_{1},\ldots,a_{n} \in \Qbar[z]$. 
Then, we have:
\begin{enumerate}[label=\rm (\alph*)]
\item \label{denomMahler assertion 1} $\denommahl{g}$ divides $\efrak \denommahl{f}$; 
\item \label{denomMahler assertion 2} if $f,g \in \Qbar((z))$, then $g$ satisfies 
	\begin{itemize}
		\item $\condprime{2}$ if $f$ satisfies $\condprime{2}$ and the nonzero roots of $\efrak$ belong to $\mathcal{U}$; 
		\item $\condprime{3}$ if $f$ satisfies $\condprime{3}$ and the nonzero roots of $\efrak$ belong to $\mathcal{U}_{\ellmahl}$. 
	\end{itemize}
	\end{enumerate}
\end{lem}

\begin{proof} 
We first prove \ref{denomMahler assertion 1}.
If $\denommahl{f}=0$, there is nothing to prove since any element of $\Qbar[z]$ divides $\efrak \denommahl{f}=0$. 
We now assume that $\denommahl{f} \neq 0$. Equation~\eqref{delta g} can be rewritten as $M(g)=f$ with $M=\efrak   - \sum_{i=1}^{n} a_i \malopop^i \in \mathcal{D}_{\Q(z)}$. Moreover, by definition of $\denommahl{f}$, there exist $b_{1},\ldots,b_{m} \in \Qbar[z]$ such that 
	$
	\denommahl{f}f=\sum_{j=1}^{m}b_j\malop{\ellmahl}^j(f), 
	$
{\it i.e.}, $L(f)=0$ where $L=\denommahl{f}-\sum_{j=1}^{m}b_j\malopop^j$. Now, the equality $LM(g)=L(f)=0$ shows that  
	$
\efrak	\denommahl{f}g
	\in  \sum_{i=1}^{m+n} \Q[z]\malop{\ellmahl}^i(g)
$
and, hence, 
$\denommahl{g}$ divides $\efrak \denommahl{f}$. This concludes the proof of \ref{denomMahler assertion 1}.

The proof of \ref{denomMahler assertion 2} follows from an obvious  combination of \ref{denomMahler assertion 1} and Theorem~\ref{theo abs 7.1}.
\end{proof}

\begin{proof}[Proof of Proposition~\ref{theo avec a0}] 
We use the notation of Theorem~\ref{theo: fund syst sol princ} applied to the companion system associated with the Mahler equation \eqref{eq: mahler1}. Arguing as in the proof of Theorem~\ref{thm:decomp_chi} given in Section~\ref{sec:proof_thm_decomp_chi}, and reusing the same notations, 
we see that it is sufficient to show that every entry of 
$F_{1}$ satisfies the following property, denoted by $(\diamond)$:
	\begin{itemize}
		\item $\condprime{2}$ if the nonzero roots of $a_{0}$ belong to $\mathcal{U}$; 
		\item $\condprime{3}$ if the nonzero roots of $a_{0}$ belong to $\mathcal{U}_{\ellmahl}$.
	\end{itemize}
To prove this, the key ingredient is the first equation in  \eqref{eq:F1Theta=AF1}, which can be rewritten as follows:
\begin{equation}\label{eq:F1ThetaAbis}
	a_{0}F_1\Lambda=B\malop{\ellmahl}(F_1)
\end{equation}
	where
	$
	\Lambda=\Theta^{-1}
	$ 
	and  
	$$
	B=a_{0}A^{-1}=\begin{pmatrix}
	-a_1 & \cdots & \cdots & -a_d\\  a_{0} & 0 & \cdots & 0
	\\ & \ddots &  \ddots &\vdots \\ &&  a_{0} & 0 
	\end{pmatrix}.
	$$
By Remark~\ref{rem pour theo: fund syst sol princ}, we can and will  assume that $\Theta$ is upper-triangular with diagonal entries in $\Qbar^{\times}$ and above-diagonal entries in $\puisinfzero$, so that 
$$
	\Lambda=\Theta^{-1}
	=\left(
	\begin{array}{ccc}
	\lambda_{1}    &       &  *    \\ 
	\bord & \ddots       &      \\ 
	0 & \bord    & \lambda_{d}     
	\end{array}\right) 
	$$ 
has diagonal entries $\lambda_{1},\ldots,\lambda_{d} \in \Qbar^{\times}$ and above-diagonal entries in $\puisinfzero$.

Note  
that, for any $k \in \Z_{\geq 1}$, the entries of $F_{1}$ satisfy $\condprime{\indC}$ if and only if the entries of $F_{1}(z^{k})$ satisfy $\condprime{\indC}$. Thus, up to replacing $z$ by $z^{k}$ 
for a suitable $k \in \Z_{\geq 1}$, we can and will assume that $F_1$ has entries in $\Qbar((z))$ and that $\Lambda$ has above-diagonal entries in $\Qbar[z^{-1}]$ (the fact that the nonzero roots of $a_{0}$ belong to $\mathcal{U}$ (resp.~$\mathcal{U}_{\ellmahl}$) implies that the nonzero roots of $a_{0}(z^{k})$ belong to $\mathcal{U}$ (resp.~$\mathcal{U}_{\ellmahl}$) as well). 
		
	Setting $F_{1}=(f_{i,j})_{1\leq i,j \leq d}$, we deduce from \eqref{eq:F1ThetaAbis} that, for all $i,j \in \{1,\ldots,d\}$, 
	\begin{equation}\label{eq:condition_coeff_matrix}
	a_{0}\left(\sum_{l=1}^{j-1} * f_{i,l} + \lambda_{j} f_{i,j} \right) = 
	\left\{
	\begin{array}{ll}
	\sum_{k=1}^{d} -a_{k} \malop{\ellmahl}(f_{k,j}) & \mbox{if } i=1, \\
	a_{0} \malop{\ellmahl}(f_{i-1,j}) & \mbox{if } i \in \{2,\ldots,d\}
	\end{array}
	\right.
	\end{equation}
	where the symbol $*$ stands for elements of $\Qbar[z^{-1}]$. Let us prove by induction on $j\in\{1,\ldots,d\}$ that $f_{i,j}$ satisfies $(\diamond)$ for all $i \in\{1,\ldots,d\}$.
	
	\vskip 5 pt
	\noindent{\bf Base case: $j=1$.} When $j=1$, \eqref{eq:condition_coeff_matrix} gives  
	$$
	a_{0} \lambda_{1} f_{i,1}  = 
	\left\{
	\begin{array}{ll}
	\sum_{k=1}^{d} -a_{k} \malop{\ellmahl}(f_{k,1}) & \mbox{if } i=1, \\
	a_{0} \malop{\ellmahl}(f_{i-1,1}) & \mbox{if } i \in \{2,\ldots,d\}.
	\end{array} 
	\right.
	$$
	This implies that, for $i \in \{2,\ldots,d\}$, $f_{i,1}=\lambda_{1}^{-(i-1)}\malop{\ellmahl}^{i-1}(f_{1,1})$ and, hence,  
	$$
	a_{0} \lambda_{1} f_{1,1}  = \sum_{k=1}^{d} -a_{k}\lambda_{1}^{-(k-1)}\malop{\ellmahl}^k(f_{1,1}).
	$$ 
	Thus, the $p$-Mahler denominator of $f_{1,1}$ divides $a_{0}$ and it follows from Theorem~\ref{theo abs 7.1} that $f_{1,1}$ satisfies $(\diamond)$. Therefore, for any $i \in \{2,\ldots,d\}$, $f_{i,1}=\lambda_{1}^{-(i-1)}\malop{\ellmahl}^{i-1}(f_{1,1})$ satisfies $(\diamond)$ as well.

\vskip 5 pt
\noindent{\bf Inductive step.} Let $j\geq 2$ and assume that $f_{i,k}$ satisfies $(\diamond)$ for any $i \in\{1,\ldots,d\}$ and any $k\in\{1,\ldots,j-1\}$. In the following, we let $\star$ stand for $\ellmahl$-Mahler elements of $\Qbar((z))$ satisfying $(\diamond)$.  Isolating the term $f_{i,j}$ in \eqref{eq:condition_coeff_matrix}, one has
	$$
a_{0}\lambda_{j} f_{i,j}  = 
\left\{
\begin{array}{ll}
\sum_{k=1}^{d} -a_{k} \malop{\ellmahl}(f_{k,j}) + \star & \mbox{if } i=1, \\
a_{0} \malop{\ellmahl}(f_{i-1,2}) + \star & \mbox{if } i \in \{2,\ldots,d\}.
\end{array}
\right.
$$
Then, an immediate induction on $i$ gives $f_{i,j}=\lambda_{j}^{-(i-1)}\malop{\ellmahl}^{i-1}(f_{1,j})+\star$ for any $i \in\{1,\ldots,d\}$. Thus, when $i=1$, one obtains
$$
a_{0} \lambda_{j} f_{1,j}  = \sum_{k=1}^{d} -a_{k}\lambda_{j}^{-(k-1)}\malop{\ellmahl}^k(f_{1,j}) +\star
$$ 
It follows from \ref{denomMahler assertion 2} of Lemma~\ref{lem:denomMahler} that $f_{1,j}$ satisfies $(\diamond)$. Therefore, for any $i \in \{2,\ldots,d\}$, $f_{i,j}=\lambda_{j}^{-(i-1)}\malop{\ellmahl}^{i-1}(f_{1,j})+\star$  satisfies $(\diamond)$ as well. This ends the proof of the inductive step.
\end{proof}

\subsubsection{Partial extension of the ``only if'' part of Theorem~\ref{theo abs 7.1}}\label{sec: ext only if part}
 To state the ``only if'' extension of Theorem~\ref{theo abs 7.1}, we introduce the following sets:
 $$
 \alaZ = \bigcup_{t \in \Z_{\geq 0}} \alaZt{t}
 \text{ \ \ \ where \ \ \ }
 \alaZt{t} =\Z_{\geq 0}^{t}\times (\Qbar^{\times})^{t} \times \Z_{> 0}^{t}. 
 $$

\begin{prop}\label{prop growth denom}
	Consider a generalized $\ellmahl$-Mahler series of the form
\begin{equation}\label{eq: pour prop growth denom}
 	f = \sum_{(c,j) \in \cj }  \sum_{\tuple \in \alaZ}  f_{c,j,\tuple}\xi_{\tuple}e_c\logm^j	
\end{equation}
where the $f_{c,j,\tuple} \in \Q((z))$ are $\ellmahl$-Mahler Laurent series. Then, the following hold:
	\begin{itemize}
		\item the $f_{c,j,\tuple}$ satisfy $\condprime{2}$ only if the $\ellmahl$-Mahler denominator of $f$ has all its nonzero roots in $\mathcal U$;
		\item the $f_{c,j,\tuple}$ satisfy $\condprime{3}$ only if the $\ellmahl$-Mahler denominator of $f$ has all its nonzero roots in $\mathcal U_\ellmahl$.
			\end{itemize}
			In particular, in both cases, the $\ellmahl$-Mahler denominator of $f$ is nonzero.
\end{prop}

\begin{proof}
In the whole proof, we fix $\indC \in \{2,3\}$.
\vskip 5pt
\noindent {\bf First step: when \eqref{eq: pour prop growth denom} is a Hahn series}.  To carry out this first step, we need to introduce a total well-order $\prec$ on $\alaZ$. 
It is defined as follows.
\begin{itemize}
	\item For any $\alpha \in \Z_{\geq 0}$ and $t \in \Z_{\geq 1}$, we fix an arbitrary well order $\prec_{\alpha,t}$ on 
	$(\{\alpha\} \times (\Z_{\geq 0})^{t-1}) \times (\Qbar^\times)^t \times \Z_{> 0}^t$.
	Such a well-order exists because the latter set is countable.
	\item For $\tuple=(\balpha,\lambd,\a) \in\alaZt{t}$ and $\tuple'=(\balpha',\lambd',\a') \in \alaZt{t'}$, we write $\tuple' \prec \tuple$ if one of the following holds:
	\begin{itemize}
		\item $t'<t$,
		\item $t'=t\geq 1$ and $\alpha'_1 < \alpha_1$,
		\item $t'=t\geq 1$, $\alpha'_1=\alpha_1$ and $\tuple' \prec_{\alpha_1,t} \omeg$.
	\end{itemize}
\end{itemize}

\begin{rem}
In Section~\ref{sec: the xi are p mahler} and Proposition~\ref{prop:gen p-mahl}, our arguments relied on well-founded induction on $\balpha$ with respect to a partial well-order on $\bigcup_{t \ge 0} \mathbb{Z}_{\ge 0}^t$. This approach will no longer suffice here. 
\end{rem}

Consider, for all $\tuple_{0} \in \alaZ$, the set $\mathcal W_{\tuple_{0}}$ of $\ellmahl$-Mahler Hahn series of the form
$$
f=f_{\tuple_0}\xi_{\tuple_0}+\sum_{\substack{\tuple \in \alaZ \\ \tuple \prec \tuple_{0}}} f_\tuple\xi_\tuple
$$
where each $f_\tuple \in \Q((z))$ is a $\ellmahl$-Mahler Laurent series satisfying $\condprime{\indC}$. In what follows, we will freely use, without further comment, that $\mathcal{W}_{\tuple_0}$ is a $\Q[z]$-module closed under multiplication by any $\ellmahl$-Mahler Laurent series satisfying $\condprime{\indC}$ and under the action of $\malop{\ellmahl}$ (the latter property follows easily from Lemma~\ref{lem:phichi-lamddchi}). 

To prove that the proposition holds true in the special case when \eqref{eq: pour prop growth denom} is a Hahn series, it is clearly equivalent to prove that it holds true for any element of any $\mathcal W_{\tuple_{0}}$. 
We proceed by well-founded induction on $\tuple_0$ with respect to the total well-order $\prec$ on $\alaZ$ introduced above. So, we consider $\tuple_{0} \in \alaZ$ and we assume that, for all $\tuple_{0}' \in \alaZ$ such that $\tuple'_{0} \prec \tuple_{0}$, the proposition holds true for any element of $\mathcal W_{\tuple_{0}'}$. Consider $f \in \mathcal W_{\tuple_{0}}$ and let us prove that the conclusion of the proposition holds true for this $f$.
By definition of the $\ellmahl$-Mahler denominator $\denommahl{}$ of $f_{\tuple_{0}}$, there exist $a_1,\ldots,a_d \in \Q[z]$ such that
$$
\denommahl{} f_{\tuple_{0}}=\sum_{j=1}^d a_j\malop{\ellmahl}^j(f_{\tuple_{0}})\,.
$$
We claim that 
\begin{equation}\label{den f eq avec g}
\denommahl{}f = \sum_{i=1}^d a_jc^{j}\malop{\ellmahl}^j(f) + g 
\end{equation}
with $g \in \mathcal W_{\tuple_{0}'}$ for some $\tuple_0' \in\alaZ$ such that $\tuple_{0}' \prec \tuple_{0}$.
Indeed, since $f\equiv f_{\tuple_{0}}\xi_{\tuple_{0}} \mod \mathcal W_{\tuple_{0}'}$ for some $\tuple_{0}' \in\alaZ$ such that $\tuple_{0}' \prec \tuple_{0}$, we have 
\begin{equation}\label{d f cong}
\denommahl{}f\equiv \denommahl{}f_{\tuple_{0}}\xi_{\tuple_{0}}
\equiv  \sum_{j=1}^d a_j \malop{\ellmahl}^j(f_{\tuple_{0}})\xi_{\tuple_{0}} \mod \mathcal W_{\tuple_{0}'}. 
\end{equation}
But, it follows from Lemma~\ref{lem:phichi-lamddchi} that 
\begin{equation}\label{xi omega 0 cong}
\xi_{\tuple_{0}} \equiv c^{j} \malop{\ellmahl}^j(\xi_{\tuple_{0}} ) \mod \mathcal W_{\tuple_{0}''} 
\end{equation}
for some $c \in \Qbar^{\times}$ and some $\tuple_{0}'' \in \alaZ$. Up to replacing $\tuple_{0}'$ and $\tuple_{0}''$ by their maximum, we can and will assume that $\tuple_{0}'=\tuple_{0}''$. 
Combining \eqref{d f cong} and \eqref{xi omega 0 cong}, we get 
$$
\denommahl{}f
\equiv\sum_{i=1}^d a_jc^{j}\malop{\ellmahl}^j(f) 
\mod \mathcal W_{\tuple_{0}'}
$$
and this justifies \eqref{den f eq avec g}.

Now, since $f_{\tuple_{0}}$ is a $\ellmahl$-Mahler Laurent series satisfying $\condprime{\indC}$, it follows from Theorem~\ref{theo abs 7.1} that the nonzero roots of the $\ellmahl$-Mahler denominator $\denommahl{}$ of $f_{\tuple_{0}}$ belong to $S_{\indC}=\mathcal U$ if $\indC=2$, and $S_{\indC}=\mathcal U_\ellmahl$ if $\indC=3$. Moreover, by the induction hypothesis, the $\ellmahl$-Mahler denominator $\denommahl{g}$ of $g$ satisfies the same property. Using \eqref{den f eq avec g} and Lemma~\ref{lem:denomMahler}, we get that the $\ellmahl$-Mahler denominator $\denommahl{f}$ of $f$ divides $\denommahl{}\cdot \denommahl{g}$ and, hence, has its nonzero roots in $S_{\indC}$. This concludes the proof of the induction and the proof of Proposition~\ref{prop growth denom} in the case when when \eqref{eq: pour prop growth denom} is a Hahn series.

\vskip 5pt
\noindent {\bf Second step: the general case}. To any generalized $\ellmahl$-Mahler series $f$ of the form \eqref{prop growth denom}, 	
we attach a positive integer $K(f)$ defined as follows. For any $(c,j) \in \cj$, we set 
\begin{equation}\label{eq:hcj}
 h_{c,j} = \sum_{\tuple \in \alaZ} f_{c,j,\tuple}\xi_{\tuple} \in \Hahn, 
\end{equation}
so that 
\begin{equation}\label{eq:fhcj}
	f=\sum_{(c,j) \in \cj} h_{c,j}e_c\logm^j.
\end{equation}
	We let $\mathcal C(f)$ be the set of $c\in \Qbar^{\times}$ such that $h_{c,j}\neq 0$ for some $j \in \Z_{\geq 0}$. For each $c \in \mathcal C(f)$, we let $J(f,c)$ denote the maximal $j \in \Z_{\geq 0}$ such that $h_{c,j}\neq 0$. We set $K(f)=\sum_{c \in\mathcal C(f)}(1+J(f,c))$.
For any $k \in \Z_{\geq 1}$, we let $\mathcal K_{k}$ be the set of generalized $\ellmahl$-Mahler series $f$ of the form \eqref{eq: pour prop growth denom} such that the $f_{c,j,\tuple}$ satisfy $\condprime{\indC}$ and such that $K(f) \leq k$.
 To prove that the proposition holds true, it is clearly equivalent to prove that it holds true for any element of any  $\mathcal K_{k}$. 
Let us prove this by induction on $k \in \Z_{\geq 1}$.

\vskip 5 pt 
\noindent {\bf Base case $k=1$.} With the above notations, for any $f \in \mathcal K_{1}$, we have  $f = h_{c,0}e_c$ for some $c \in \Qbar^{\times}$. As $h_{c,0}e_c$ and $h_{c,0}$ have the same $\ellmahl$-Mahler denominator, the result follows immediately from the first step of the proof. 
\vskip 5 pt 
\noindent 	{\bf Inductive step.}  
Consider $k \in \Z_{\geq 2}$ and assume that the proposition holds true for any element of $\mathcal K_{k-1}$. Consider a nonzero generalized $\ellmahl$-Mahler series $f \in \mathcal K_{k}$.  
Choose an arbitrary $c_0 \in \mathcal C(f)$ and set $j_0=J(f,c_0)$. Let $\K \subset \K_{k-1}$ denote the set generalized $\ellmahl$-Mahler series $g$ such that $\mathcal C(g) \subset \mathcal C(f)$, $J(g,c)\leq J(f,c)$ for any $c \in \mathcal C(g)$ and $J(g,c_0)<j_0=J(f,c_0)$. In what follows, we will freely use the fact that $\mathcal K$ is a $\Q[z]$-module invariant under $\malop{\ellmahl}$. By definition of the $\ellmahl$-Mahler denominator $\dfrak$ of $h_{c_0,j_0}$, we have
	$$\dfrak h_{c_0,j_0}=\sum_{i=1}^d a_i\malop{\ellmahl}^i(h_{c_0,j_0})
	$$ 
	for some $d \in \Z_{\geq 1}$ and some $a_1,\ldots,a_{d} \in \Qbar[z]$. 
We claim that 
\begin{equation}\label{claim d f}
\dfrak f = \sum_{i=1}^d a_ic_0^{-i}\malop{\ellmahl}^i(f) + g 
\end{equation}
with $g \in \mathcal K$.
Indeed, since $f\equiv h_{c_0,j_0}e_{c_0}\logm^{j_0} \mod \mathcal K$, we have:
	$$
	\dfrak f \equiv  \dfrak h_{c_{0},j_{0}}e_{c_0}\logm^{j_{0}} \equiv \sum_{i=1}^d a_i\malop{\ellmahl}^i(h_{c_0,j_0})e_{c_0}\logm^{j_0} \mod \mathcal{K}.
$$
Furthermore, using the identity $\malop{\ellmahl}^i(e_{c_0}\logm^{j_0})=c_0^ie_{c_0}(\logm+i)^{j_0}$ for $i \in \{0,\ldots,j_{0}\}$, it is easily seen that $e_{c_0}\logm^{j_0}$ is equal to the sum of $c_0^{-i}\malop{\ellmahl}^i(e_{c_0}\logm^{j_0})$ with a $\Qbar$-linear combination of the $e_{c_0}\logm^j\in \K$, $j \in \{0,\ldots,j_{0}-1\}$. This implies that 
$$
\dfrak f \equiv
\sum_{i=1}^d a_ic_0^{-i}\malop{\ellmahl}^i(h_{c_0,j_0}e_{c_0}\logm^{j_0}) \mod \mathcal K.
$$
Since $f\equiv h_{c_0,j_0}e_{c_0}\logm^{j_0} \mod \mathcal K$, we get 
$$
\dfrak f \equiv 
\sum_{i=1}^d a_ic_0^{-i}\malop{\ellmahl}^i(f) \mod \mathcal K
$$
and this proves our claim. 
Now, it follows from the first part of the proof that the $\ellmahl$-Mahler denominator $\dfrak$ of $h_{c_0,j_0}$ has its nonzero roots in $S_{\indC}=\mathcal U$ if $\indC=2$, and $S_{\indC}=\mathcal U_\ellmahl$ if $\indC=3$. Moreover, since $\K \subset \K_{k-1}$, by the induction hypothesis, the $\ellmahl$-Mahler denominator $\denommahl{g}$ of $g$ satisfies the same property. Using \eqref{claim d f} and Lemma~\ref{lem:denomMahler}, we get that the $\ellmahl$-Mahler denominator $\denommahl{f}$ of $f$ divides $\denommahl{}\cdot \denommahl{g}$ and, hence, has its nonzero roots in $S_{\indC}$. This concludes the induction and the proof of Proposition~\ref{prop growth denom}. 
\end{proof}
\black

\section{Purity Theorem: proof of Theorem~\ref{thm:pureté1}}\label{sec:proof purity theo}

The proof is given below, after a few lemmas that will enable us to reduce the problem to the case of generalized $\ellmahl$-Mahler series of the form
\eqref{eq:generalized_mahler_series_bis}
in which the $f_{c,j,\tuple}$ are Laurent series, rather than merely Puiseux series.

The family $(e_c\logm^j)_{(c,j)\in\cj}$ is a basis of the $\Hahn$-vector space $\mathcal R$ introduced in Section~\ref{sec: fund mat sol rappels the ring R}.  For any $\nu \in \Z_{\geq 1}$ relatively prime with $\ellmahl$, for any $k \in \Z_{\geq 0}$, we consider the map $[\nu\ellmahl^k]_{*}: \mathcal R \rightarrow \mathcal R$ defined, for any
$$
f=\sum_{(c,j)\in\cj} f_{c,j}e_c\logm^j \in \mathcal R
$$
with $f_{c,j} \in \Hahn$, by  
$$
 [\nu\ellmahl^k]_{*}f  =  \sum_{(c,j)\in\cj} f_{c,j}(z^{\nu\ellmahl^k}) \malop{\ellmahl}^{k}(e_c)\malop{\ellmahl}^{k}(\logm)^j \\
 =  \sum_{(c,j)\in\cj} f_{c,j}(z^{\nu\ellmahl^k}) c^{k} e_c  (\logm+k)^j. 
$$
This map is an automorphism of the $\Qbar$-vector space $\mathcal R$ and satisfies, for any $h \in \Hahn$ and any $f \in \mathcal R$,  
\begin{equation}\label{eq : diff aut nu ell k}
  [\nu\ellmahl^k]_{*}(h f)=h(z^{\nu\ellmahl^k})[\nu\ellmahl^k]_{*}(f).
\end{equation}
Its inverse $[\nu\ellmahl^k]_{*}^{-1} : \mathcal R \rightarrow \mathcal R$ is given, for all $f \in \mathcal R$, by 
$$[\nu\ellmahl^k]_{*}^{-1}f=\sum_{(c,j)\in\cj} f_{c,j}(z^{\nu^{-1}\ellmahl^{-k}}) c^{-k} e_c  (\logm-k)^j.
$$

\begin{lem}\label{lem:eq pour * nu pk}
An element $f$ of $\mathcal R$ is solution of the equation 
	\begin{equation*}\label{eq: action nu ellmahl k}
	a_0(z)y+a_1(z)\malop{\ellmahl}(y)+ \cdots+ a_d(z) \malop{\ellmahl}^d(y)=0,
\end{equation*}
$a_0,\ldots,a_d \in \Kinf$, if and only if $ [\nu\ellmahl^k]_{*}f$ is solution of the equation
	$$
	a_0(z^{\nu\ellmahl^k})y  +a_1(z^{\nu\ellmahl^k})\malop{\ellmahl}(y)+ \cdots+ a_d(z^{\nu\ellmahl^k}) \malop{\ellmahl}^d(y)=0.
	$$
\end{lem}

\begin{proof} Since $ [\nu\ellmahl^k]_{*}$ is an isomorphism, the equation 
$$
a_0(z)f+a_1(z)\malop{\ellmahl}(f)+ \cdots+ a_d(z) \malop{\ellmahl}^d(f)=0
$$
is equivalent to the equation 
\begin{equation}\label{eq:nu ellmahl appl eq f}
 [\nu\ellmahl^k]_{*}(a_0(z)f+a_1(z)\malop{\ellmahl}(f)+ \cdots+ a_d(z) \malop{\ellmahl}^d(f))=0.  
\end{equation}
Using the fact that $[\nu\ellmahl^k]_{*}$ is $\Q$-linear and satisfies \eqref{eq : diff aut nu ell k}, we see that \eqref{eq:nu ellmahl appl eq f} is equivalent to 
 $$
	a_0(z^{\nu\ellmahl^k}) [\nu\ellmahl^k]_{*}(f)  +a_1(z^{\nu\ellmahl^k}) [\nu\ellmahl^k]_{*}(\malop{\ellmahl}(f))+ \cdots+ a_d(z^{\nu\ellmahl^k}) [\nu\ellmahl^k]_{*}(\malop{\ellmahl}^d(f))=0.
	$$
Last, since $[\nu\ellmahl^k]_{*}$ and $\malop{\ellmahl}$ commute, the latter equation is equivalent to 
$$
	a_0(z^{\nu\ellmahl^k}) [\nu\ellmahl^k]_{*}(f)  +a_1(z^{\nu\ellmahl^k}) \malop{\ellmahl}([\nu\ellmahl^k]_{*}(f))+ \cdots+ a_d(z^{\nu\ellmahl^k}) \malop{\ellmahl}^d([\nu\ellmahl^k]_{*}(f))=0.
	$$
\end{proof}

\begin{lem}\label{lem:change_variable}
Consider a generalized $\ellmahl$-Mahler series $f$, a positive integer $\nu$ relatively prime with $\ellmahl$, a non-negative integer $k$ and $\indC \in \{1,2,3,4,5\}$. Then, $f$ satisfies $\condpuisprime{\indC}$ if and only if $ [\nu\ellmahl^k]_{*}f$ satisfies $\condpuisprime{\indC}$.
\end{lem}

\begin{proof}
Let us first consider the case when $f$ is a Puiseux series. In order to prove the lemma in this case, it is clearly sufficient to prove that if $f \in \puis$ satisfies $\condprime{\indC}$ and if $c \in \QQ_{>0}$, then $g(z)=f(z^{c})$ satisfies $\condprime{\indC}$. Let us prove this. Setting $f=\sum_{\gamma \in \QQ} f_{\gamma}z^{\gamma}$, we have $g=\sum_{\gamma \in \QQ} g_{\gamma}z^{\gamma}$ with $g_{\gamma}= f_{c^{-1}\gamma}$. If $r=1$, we have $\hgt(f_\gamma) \in \mathcal O(H(\gamma))$, so $\hgt(g_\gamma) \in \mathcal O(H(c^{-1}\gamma))$. But, $H(c^{-1} \gamma) \leq H(c^{-1})H(\gamma)$. So $H(c^{-1} \gamma) \in \mathcal O(H(\gamma))$ and, hence, $\hgt(g_\gamma) \in \mathcal O(H(\gamma))$ so that $g$ satisfies $\condprime{1}$ as wanted. The cases $r \in \{2,3,4,5\}$ are similar.

	We now come to the general case. 
	Let 
	$$
	f=\sum_{(c,j) \in \cj}\sum_{\tuple \in \alaNp}f_{c,j,\tuple}\xi_{\tuple}e_c \logm^j
	$$
be the standard decomposition of $f$. 
\vskip 5pt	
\noindent 
{\bf Case $k=0$.} 
The standard decomposition of $[\nu]_{*}f$ is given by  
\begin{equation*}\label{nu * f}
 [\nu]_{*}f 
 =\sum_{(c,j) \in \cj}\sum_{\tuple\in \alaNp}f_{c,j,\tuple}(z^\nu)\xi_{\tuple}(z^{\nu})e_c \logm^j
 =\sum_{(c,j) \in \cj}\sum_{\tuple\in \alaNp}f_{c,j,\tuple}(z^\nu)\xi_{\tuple_{\nu}}(z)e_c \logm^j
\end{equation*}
where, for any $\tuple=(\balpha,\lambd,\a) \in \alaNp$, we set $\tuple_{\nu}=(\balpha,\lambd,\nu\a) \in \alaNp$.  
Then, the following properties are equivalent:
\begin{enumerate}
\item $f$ satisfies $\condpuisprime{\indC}$; 
\item $\forall (c,j) \in \cj, \forall \tuple\in \alaNp$, $f_{c,j,\tuple}(z)$ satisfies $\condprime{\indC}$; 
\item $\forall (c,j) \in \cj, \forall \tuple\in \alaNp$, $f_{c,j,\tuple}(z^\nu)$ satisfies $\condprime{\indC}$;  
\item	$[\nu]_{*}f$ satisfies $\condpuisprime{\indC}$.
\end{enumerate}
Indeed, the equivalences between (1) and (2) and between (3) and (4) follow directly from Proposition~\ref{prop:def Cr gms} and the equivalence between (2) and (3) follows from the Puiseux case considered at the beginning of the proof. This concludes the proof in the case $k=0$.

\vskip 5pt	
\noindent 
{\bf Case $\nu=1$ and $k=1$.}
It follows from Lemma~\ref{lem:phichi-lamddchi} that, for any $\tuple \in \alaNp$, $\xi_{\omeg}(z^{\ellmahl})$ is a $\puisinfzero$-linear combination of some $\xi_{\tuple'}(z)$ with $\tuple' \in \alaNp$. 
We infer form this that the standard decomposition of $[\ellmahl]_{*}f$ is given by 
\begin{align*}
[\ellmahl]_{*}f &=  \sum_{(c,j) \in \cj}\sum_{\tuple\in \alaNp}f_{c,j,\tuple}(z^\ellmahl)\xi_{\tuple}(z^{\ellmahl})ce_c (\logm+1)^j\\ \nonumber&
=\sum_{(c,j') \in \cj}\sum_{\tuple'\in \alaNp}g_{c,j',\tuple'}(z)\xi_{\tuple'}(z)e_c \logm^{j'}
\end{align*}
where each Puiseux series $g_{c,j',\tuple'}(z)$ is a $\puisinfzero$-linear combination of the $f_{c,j,\tuple}(z^{\ellmahl})$ for some $j \in \Z_{\geq 0}$ and $\tuple \in \alaNp$. Then, the following properties are equivalent:
\begin{enumerate}
	\item $f$ satisfies $\condpuisprime{\indC}$; 
	\item $\forall (c,j) \in \cj, \forall \tuple\in \alaNp$, $f_{c,j,\tuple}(z)$ satisfies $\condprime{\indC}$; 
	\item $\forall (c,j') \in \cj, \forall \tuple' \in \alaNp$, $g_{c,j',\tuple'}(z)$ satisfies $\condprime{\indC}$;  
	\item	$[\ellmahl]_{*}f$ satisfies $\condpuisprime{\indC}$.
\end{enumerate}
Indeed, the equivalences between (1) and (2) and between (3) and (4) follow directly from Proposition~\ref{prop:def Cr gms}. The equivalence between (2) and (3) follows from the Puiseux case considered at the beginning of the proof and the fact that the set $\condpuisprime{\indC}$ is a $\puisinfzero$-module.

\vskip 5pt	
\noindent
{\bf General case.}  The case $k \in \Z_{\geq 1}$ follows immediately from the previous particular cases by using the fact that $[\nu \ellmahl^{k}]_{*}=[\ellmahl]_{*}^{k} [\nu]_{*}$.
\end{proof}

We are now ready to prove Theorem~\ref{thm:pureté1}. 
We set 
\begin{equation}\label{eq:f for proof theo purete 1}
f
=
\sum_{(c,j) \in \cj}\sum_{\tuple \in \ala}f_{c,j,\tuple}\xi_{\tuple}e_c \logm^{j}  
\end{equation}
where the $f_{c,j,\tuple}$ are $\ellmahl$-Mahler Puiseux series satisfying $\condprime{\indC}$.

	Let first consider the case $\indC=1$. Theorem~\ref{thm:decomp_chi} ensures that the minimal $\ellmahl$-Mahler equation of $f$ over $\Kinf$ has a full basis of generalized $\ellmahl$-Mahler series solutions. Theorem~\ref{thm: hgt} guarantees that any generalized $\ellmahl$-Mahler series satisfies $\condpuisprime{1}$. This concludes the proof in the case $\indC=1$.
	
	We now consider the case $\indC \in \{2,3\}$.
	
Let us first consider the special case when the $f_{c,j,\tuple}$ belong to $\Qbar((z))$ and where the $\tuple$ involved in the support of the sum  \eqref{eq:f for proof theo purete 1} belong to the set $\alaZ$ introduced in Section~\ref{sec: ext only if part}. By definition of the $\ellmahl$-Mahler denominator $\denommahl{f}$ of $f$, there exists a $\ellmahl$-Mahler operator $L$ with coefficients in $\Qbar[z]$ and with constant coefficient $\denommahl{f}$ such that $L(f)=0$.   
But, Proposition~\ref{prop growth denom} ensures that  $\denommahl{f}$ has its nonzero roots in $\mathcal U$ if $\indC=2$ and in $\mathcal U_\ellmahl$ if $\indC=3$. So, 
according to Proposition~\ref{theo avec a0}, the operator $L$ has a full basis of generalized $\ellmahl$-Mahler series solutions satisfying $\condpuisprime{\indC}$. The minimal $\ellmahl$-Mahler operator of $f$ satisfies the same property because it is a right factor of $L$.

We now come to the general case. For any $\nu \in \Z_{\geq 1}$ relatively prime with $\ellmahl$ and any $k\in \Z_{\geq 0}$, we have 
\begin{equation}\label{eq:rescalling}
[\nu \ellmahl^{k}]_{*}f
=
\sum_{(c,j) \in \cj}\sum_{\tuple \in \ala}f_{c,j,\tuple}(z^{\nu \ellmahl^{k}})\xi_{\tuple_{\nu \ellmahl^{k}}}(z)c^{k}e_c (\logm +k)^j. 
\end{equation}
where, for any $\tuple=(\balpha,\lambd,\a)\in\ala$, we set $\tuple_{\nu \ellmahl^{k}}=(\balpha,\lambd,\nu\ellmahl^k\a)$. Fix such $\nu$ and $k$ such that the $f_{c,j,\tuple}(z^{\nu \ellmahl^{k}})$ belong to $\Qbar((z))$ and the tuples $\nu\ellmahl^k\a$ have entries in $\Z_{\geq 0}$, so that the $\tuple_{\nu \ellmahl^{k}}$ involved in \eqref{eq:rescalling} belong to $\alaZ$. Thus, we have 
$$
[\nu \ellmahl^{k}]_{*}f
=
\sum_{(c,j') \in \cj}\sum_{\tuple' \in \alaZ}g_{c,j',\tuple'}\xi_{\tuple'}e_c \logm^{j'} 
$$
where the $g_{c,j',\tuple'}\in\Qbar((z))$ are $\Q$-linear combinations of the $f_{c,j,\tuple}(z^{\nu \ellmahl^{k}})$ with $j \in \Z_{\geq 0}$ and $\tuple \in \alaZ$. In particular, these $g_{c,j',\tuple'}$ satisfy $\condprime{\indC}$. It follows from the first part of the proof that the minimal $\ellmahl$-Mahler equation of $[\nu \ellmahl^{k}]_{*}f$ over $\Kinf$, say
$$
a_0(z)y + \cdots+ a_{d-1}(z) \malop{\ellmahl}^{d-1}(y)+ \malop{\ellmahl}^d(y)=0, 
$$
has a full basis $g_1,\ldots,g_d$ of generalized $\ellmahl$-Mahler series satisfying $\condpuisprime{\indC}$. It follows from Lemma~\ref{lem:eq pour * nu pk} that  $f_{1}=[\nu \ellmahl^{k}]_{*}^{-1}(g_1),\ldots,f_{d}=[\nu \ellmahl^{k}]_{*}^{-1}(g_d)$ is a full basis of solutions of the equation
$$
a_0(z^{\nu^{-1}p^{-k}})y + \cdots+ a_{d-1}(z^{\nu^{-1}p^{-k}}) \malop{\ellmahl}^{d-1}(y)+ \malop{\ellmahl}^d(y)=0\,.
$$
which is the minimal $\ellmahl$-Mahler equation of $f$ over $\Kinf$. Lemma~\ref{lem:change_variable} ensures that $f_{1},\ldots,f_{d}$ satisfy $\condpuisprime{\indC}$, whence the desired result.

\section{Final remark about the Purity Theorem}\label{sec:final_rem}

Theorem~\ref{thm:pureté1} states that the property  $\condpuisprime{1}$, $\condpuisprime{2}$ or $\condpuisprime{3}$ of a generalized $\ellmahl$-Mahler series is inherited by the other solutions of its minimal equation. The following result shows that this is not true for $\condpuisprime{4}$ nor for $\condpuisprime{5}$.

\begin{prop}\label{prop:contre_exemple}
	There exists a $\ellmahl$-Mahler power series satisfying $\condpuisprime{4}$ and $\condpuisprime{5}$ having the following property: its minimal $\ellmahl$-Mahler equation has a generalized $\ellmahl$-Mahler series solution which does neither satisfy $\condpuisprime{4}$ nor $\condpuisprime{5}$.
\end{prop}
\begin{proof}
Let $\ellmahl=2$.	Consider the following equation 
		\begin{equation}\label{eq:RS}
	y(z)+(z-1)y(z^2)-2zy(z^4)=0. 
	\end{equation}
	It is well-known that one of its solutions is the generating series $f_{\rm RS}\in \Q[[z]]$ of the Rudin-Shapiro sequence (see for example \cite{AS03}):
	$$
	(a_n)_{n \in \Z_{\geq 0}} = 1, 1, 1, -1, 1, 1, -1, 1, 1, 1, 1, -1, -1, -1, 1,\ldots \,.
	$$
Its coefficients belong to $\{-1,1\}$, so that it satisfies $\condprime{5}$ and, thus, $\condpuisprime{4}$ and $\condpuisprime{5}$. A study of the Newton polygon of this equation, as in \cite{RoquesFrobForMahler}, shows that the exponents attached to this equation are $1$ and $-\frac{1}{2}$. Thus, it follows from \cite{RoquesFrobForMahler} 
	and Theorem~\ref{theo: fund syst sol princ} that the system associated with \eqref{eq:RS} has a fundamental matrix of solutions of the form $F_1F_2e_C$, where 
	$$
	F_1 \in \GL_2(\puis), \quad F_2= \begin{pmatrix}
	1 & \xi \\ 0 & 1
	\end{pmatrix}, \quad e_C=\begin{pmatrix} 1 & 0 \\ 0 & e_{-\frac{1}{2}}\end{pmatrix},
	$$
	with $\xi \in \V_1$.
 The upper-left entry of $F_1$ is solution of \eqref{eq:RS}. Hence, up to multiplication by a scalar, we may take it to be $f_{\rm RS}$. Let $g=(F_1)_{1,2} \in \puis$ be the upper-right entry of $F_1$. Then, a second solution of \eqref{eq:RS} is the generalized $2$-Mahler series $fe_{-\frac{1}{2}}$ where $f=f_{\rm RS}\xi + g \in \Hahn$.

	Using the fact that $\malop{2}(e_{-\frac{1}{2}})=-\frac{1}{2}e_{-\frac{1}{2}}$, we obtain that $f$ is solution of the equation
	\begin{equation}
	\label{eq:RS-Hanh}
	y(z)-\frac{1}{2}(z-1)y(z^2)-\frac{1}{2}zy(z^4)=0.
	\end{equation}
	Let $\tuple=(0,-2,1)$. We claim that we can take
		$$
	\xi(z)=	\frac{1}{2}\xi_{\tuple}(z) = -\sum_{k=1}^\infty (-2)^{k-1}z^{-1/2^k} \in \Hahn
		$$
		To prove this claim it is sufficient to prove that there exists a Puiseux series $\widetilde{g}$ such that $\frac{1}{2}f_{\rm RS}\xi_{\tuple}(z) + \widetilde{g}$ is solution of \eqref{eq:RS-Hanh}. Since $\xi_\omeg(z^2)=-2\xi_\omeg(z)-\frac{2}{z}$, it is equivalent to prove that the equation
			\begin{equation}\label{eq:g}
		y(z)-\frac{1}{2}(z-1)y(z^2)-\frac{1}{2}zy(z^4)
		= \frac{1}{2z}f_{\rm RS}(z)-\frac{1}{2z}f_{\rm RS}(z^4)
		\end{equation}
		has a Puiseux series solution.
		 A study of the Newton polygon associated to \eqref{eq:g}, as in \cite{CompSolMahlEq}, shows that it is the case. Thus, we can take $\xi=\frac{1}{2}\xi_{\tuple}$ and $g$ to be this power series solution of \eqref{eq:g}. Since the decomposition 
		 $$f=\left(\frac{1}{2}f_{\rm RS}\right)\xi_\omeg + g\xi_{((),(),())}
		 $$ is the standard decomposition of $f$, to conclude it is sufficient to prove that $g$ does not satisfy $\condprime{4}$.

		Using \eqref{eq:g}, it is easily checked that $\val g=0$.   Looking at the coefficient of $z^0$ in \eqref{eq:g} we obtain that $g_0=\frac{1}{3}$. Let $g=\sum_{n\geq 0}g_nz^n$.  Looking at the coefficient of $z^1$ in \eqref{eq:g} and using the fact that $g \in \Q[[z]]$ we obtain
	$$
	g_1 - \frac{1}{2}g_0 - \frac{1}{2} g_0 = \frac{1}{2}a_2,
	$$ 
	where we let $f_{\rm RS}=\sum_{n\geq 0} a_nz^n$. Thus, $g_1=\frac{5}{6}$. Now, since $g$ is a power series, looking at the coefficient of $z^{2^n}$ in \eqref{eq:g} we obtain,
	$$
	g_{2^{n}}=\frac{1}{2}a_{2^n+1} - \frac{1}{2}g_{2^{n-1}}
	$$
	Since $a_{2^n+1} = \pm 1$, it follows by induction on $n$ that the $2$-adic valuation of $g_{2^{n}}$ is equal to $n+1$. In particular, $\hgt(g_{2^n})\geq n$. Since $n = \log(H(2^n))/\log(2)$, we have $\hgt(g_\gamma) \in \Omega(\log(H(\gamma)))$ and $g(z)$ does not satisfy $\condprime{4}$.
	\end{proof}

\bibliographystyle{alpha}
\bibliography{biblio}

\def\cprime{$'$} \def\polhk#1{\setbox0=\hbox{#1}{\ooalign{\hidewidth
  \lower1.5ex\hbox{`}\hidewidth\crcr\unhbox0}}} \def\cprime{$'$}
\begin{thebibliography}{CDDM18}

\bibitem[AB17]{BorisAboutMahler}
B.~Adamczewski and J.~P. Bell.
\newblock A problem about {M}ahler functions.
\newblock {\em Ann. Sc. Norm. Super. Pisa Cl. Sci. (5)}, 17(4):1301--1355,
  2017.

\bibitem[ABS23]{BorisBellSmertnigGap}
B.~Adamczewski, J.~Bell, and D.~Smertnig.
\newblock A height gap theorem for coefficients of {M}ahler functions.
\newblock {\em J. Eur. Math. Soc. (JEMS)}, 25(7):2525--2571, 2023.

\bibitem[Ada19]{BorisMahlerSelecta}
B.~Adamczewski.
\newblock Mahler's method.
\newblock {\em Doc. Math.}, pages 95--122, 2019.

\bibitem[ADH21]{ADHHLDE}
B.~Adamczewski, T.~Dreyfus, and C.~Hardouin.
\newblock Hypertranscendence and linear difference equations.
\newblock {\em J. Amer. Math. Soc.}, 34(2):475--503, 2021.

\bibitem[AF17]{BorisFaverjonMethodeMahler17}
B.~Adamczewski and C.~Faverjon.
\newblock M\'{e}thode de {M}ahler: relations lin\'{e}aires, transcendance et
  applications aux nombres automatiques.
\newblock {\em Proc. Lond. Math. Soc. (3)}, 115(1):55--90, 2017.

\bibitem[AF18]{BorisFaverjonMethodeMahler}
B.~Adamczewski and C.~Faverjon.
\newblock M\'{e}thode de {M}ahler, transcendance et relations lin\'{e}aires:
  aspects effectifs.
\newblock {\em J. Th\'{e}or. Nombres Bordeaux}, 30(2):557--573, 2018.

\bibitem[AF23]{AF23}
B.~Adamczewski and C.~Faverjon.
\newblock A new proof of {N}ishioka's theorem in {M}ahler's method.
\newblock {\em C. R. Math. Acad. Sci. Paris}, 361:1011--1028, 2023.

\bibitem[AF24a]{AF24_SevVar}
B.~Adamczewski and C.~Faverjon.
\newblock {M}ahler's method in several variables and finite automata.
\newblock {\em to appear in Annals of Math.}, 2024.
\newblock arXiv:2012.08283 [math.NT], 60p.

\bibitem[AF24b]{AF24_EM}
Boris Adamczewski and Colin Faverjon.
\newblock Relations alg\'ebriques entre valeurs de {$E$}-fonctions ou de
  {$M$}-fonctions.
\newblock {\em C. R. Math. Acad. Sci. Paris}, 362:1215--1241, 2024.

\bibitem[And89]{AndreGFAG}
Y.~Andr\'{e}.
\newblock {\em {$G$}-functions and geometry}, volume E13 of {\em Aspects of
  Mathematics}.
\newblock Friedr. Vieweg \& Sohn, Braunschweig, 1989.

\bibitem[And00]{AndreSGA1}
Y.~Andr\'{e}.
\newblock S\'{e}ries {G}evrey de type arithm\'{e}tique. {I}. {T}h\'{e}or\`emes
  de puret\'{e} et de dualit\'{e}.
\newblock {\em Ann. of Math. (2)}, 151(2):705--740, 2000.

\bibitem[And03]{AndreArithGevreySeriesSurvey}
Y.~Andr\'e.
\newblock Arithmetic {G}evrey series and transcendence. {A} survey.
\newblock {\em Journal de th\'eorie des nombres de Bordeaux}, 15(1):1--10,
  2003.

\bibitem[And14]{Andre14}
Y.~Andr\'{e}.
\newblock Solution algebras of differential equations and quasi-homogeneous
  varieties: a new differential {G}alois correspondence.
\newblock {\em Ann. Sci. \'{E}c. Norm. Sup\'{e}r. (4)}, 47(2):449--467, 2014.

\bibitem[AS03]{AS03}
J.-P. Allouche and J.~Shallit.
\newblock {\em Automatic sequences}.
\newblock Cambridge University Press, Cambridge, 2003.
\newblock Theory, applications, generalizations.

\bibitem[AZ22]{ArrecheZhangMahlerDiscreteResidues}
C.~E. Arreche and Y.~Zhang.
\newblock Mahler discrete residues and summability for rational functions.
\newblock In {\em I{SSAC} '22---{P}roceedings of the 2022 {I}nternational
  {S}ymposium on {S}ymbolic and {A}lgebraic {C}omputation}, pages 525--533.
  ACM, New York, 2022.

\bibitem[BBC15]{BBC15}
J.~P. Bell, Y.~Bugeaud, and M.~Coons.
\newblock Diophantine approximation of {M}ahler numbers.
\newblock {\em Proc. Lond. Math. Soc. (3)}, 110(5):1157--1206, 2015.

\bibitem[BCCD19]{BeckerConj}
J.~Bell, F.~Chyzak, M.~Coons, and P.~Dumas.
\newblock Becker's conjecture on {M}ahler functions.
\newblock {\em Trans. Amer. Math. Soc.}, 372(5):3405--3423, 2019.

\bibitem[BCR13]{BCR13}
J.~P. Bell, M.~Coons, and E.~Rowland.
\newblock The rational-transcendental dichotomy of {M}ahler functions.
\newblock {\em J. Integer Seq.}, 16(2):Article 13.2.10, 11, 2013.

\bibitem[BCZ16]{BCZ16}
R.~P. Brent, M.~Coons, and W.~Zudilin.
\newblock Algebraic independence of {M}ahler functions via radial asymptotics.
\newblock {\em Int. Math. Res. Not. IMRN}, (2):571--603, 2016.

\bibitem[Beu06]{Beukers06}
F.~Beukers.
\newblock A refined version of the {S}iegel-{S}hidlovskii theorem.
\newblock {\em Ann. of Math. (2)}, 163(1):369--379, 2006.

\bibitem[CC85]{ChudnovskyAPADIVGF}
D.~V. Chudnovsky and G.~V. Chudnovsky.
\newblock Applications of {P}ad\'{e} approximations to {D}iophantine
  inequalities in values of {$G$}-functions.
\newblock In {\em Number theory ({N}ew {Y}ork, 1983--84)}, volume 1135 of {\em
  Lecture Notes in Math.}, pages 9--51. Springer, Berlin, 1985.

\bibitem[CDDM18]{CompSolMahlEq}
F.~Chyzak, T.~Dreyfus, P.~Dumas, and M.~Mezzarobba.
\newblock Computing solutions of linear {M}ahler equations.
\newblock {\em Math. Comp.}, 87(314):2977--3021, 2018.

\bibitem[DHR18]{DHRMahler}
T.~Dreyfus, C.~Hardouin, and J.~Roques.
\newblock Hypertranscendence of solutions of {M}ahler equations.
\newblock {\em J. Eur. Math. Soc. (JEMS)}, 20(9):2209--2238, 2018.

\bibitem[Dum93]{DumasThese}
P.~Dumas.
\newblock {R}{\'e}currences mahl{\'e}riennes, suites automatiques : {\'e}tudes
  asymptotiques. {T}h{\`e}se de doctorat.
\newblock 1993.

\bibitem[DVS11]{dVS11}
L.~Di~Vizio and J.~Sauloy.
\newblock Outils pour la classification locale des \'equations aux
  {$q$}-diff\'erences lin\'eaires complexes.
\newblock In {\em Arithmetic and {G}alois theories of differential equations},
  volume~23 of {\em S\'emin. Congr.}, pages 169--222. Soc. Math. France, Paris,
  2011.

\bibitem[FP25]{FaverjonPoulet2}
C.~Faverjon and M.~Poulet.
\newblock Regular singular {M}ahler equations and {N}ewton polygons.
\newblock {\em To appear in J. Math. Soc. Japan}, 2025.
\newblock 28 pp.

\bibitem[FR24]{FaverjonRoquesHahnMahlAlgoAspects}
C.~Faverjon and J.~Roques.
\newblock Hahn series and {M}ahler equations: {A}lgorithmic aspects.
\newblock {\em J. Lond. Math. Soc. (2)}, 110(1):Paper No. e12945, 2024.

\bibitem[FR25]{FaverjonRoquesHGT}
C.~Faverjon and J.~Roques.
\newblock A height gap theorem for the coefficients of {M}ahler {H}ahn series.
\newblock {\em Work in progress, preprint available soon}, 2025.

\bibitem[HS99]{SolDiffEqFinTerms}
P.~A. Hendriks and M.~F. Singer.
\newblock Solving difference equations in finite terms.
\newblock {\em J. Symbolic Comput.}, 27(3):239--259, 1999.

\bibitem[Mah29]{MahlerArith1929}
K.~Mahler.
\newblock Arithmetische {E}igenschaften der {L}\"osungen einer {K}lasse von
  {F}unktionalgleichungen.
\newblock {\em Math. Ann.}, 103(1):532, 1929.

\bibitem[Mah30a]{MahlerArith1930}
K.~Mahler.
\newblock {A}rithmetische {E}igenschaften einer {K}lasse
  transzendental-transzendente funktionen.
\newblock {\em Math. Z.}, 32:545--585, 1930.

\bibitem[Mah30b]{MahlerUber1930}
K.~Mahler.
\newblock Uber das {V}erschwinden von {P}otenzreihen mehrerer
  {V}er\"anderlichen in speziellen {P}unktfolgen.
\newblock {\em Math. Ann.}, 103(1):573--587, 1930.

\bibitem[MNS22]{medvedev2022skewinvariantcurvesalgebraicindependence}
A.~Medvedev, K.~D. Nguyen, and T.~Scanlon.
\newblock Skew-invariant curves and the algebraic independence of {M}ahler
  functions, 2022.

\bibitem[Ngu11]{NG}
P.~Nguyen.
\newblock Hypertranscendance de fonctions de {M}ahler du premier ordre.
\newblock {\em C. R. Math. Acad. Sci. Paris}, 349(17-18):943--946, 2011.

\bibitem[Ngu12]{NGT}
P.~Nguyen.
\newblock \'{E}quations de {M}ahler et hypertranscendance.
\newblock {\em Th\`ese de l'Institut de Math\'ematiques de Jussieu}, 2012.

\bibitem[NN12]{NishiokaNishiokaATDEMF}
K.~Nishioka and S.~Nishioka.
\newblock Algebraic theory of difference equations and {M}ahler functions.
\newblock {\em Aequationes Math.}, 84(3):245--259, 2012.

\bibitem[NS20]{NagySzamuelyAGTAndreSolAlg}
L.~Nagy and T.~Szamuely.
\newblock A general theory of {A}ndr\'e's solution algebras.
\newblock {\em Ann. Inst. Fourier (Grenoble)}, 70(5):2103--2129, 2020.

\bibitem[Pel09]{pellarinAIMM}
F.~Pellarin.
\newblock An introduction to {M}ahler's method for transcendence and algebraic
  independence.
\newblock In G.~Boeckle, D.~Goss, U.~Hartl, and M.~Papanikolas, editors, {\em
  EMS proceedings of the conference "Hodge structures, transcendence and other
  motivic aspects"}, pages 297--349, 2009.

\bibitem[Phi15]{PhilipponGaloisNA}
P.~Philippon.
\newblock Groupes de {G}alois et nombres automatiques.
\newblock {\em J. Lond. Math. Soc. (2)}, 92(3):596--614, 2015.

\bibitem[Pou23]{PouletDensity}
M.~Poulet.
\newblock A density theorem for the difference {G}alois groups of regular
  singular {M}ahler equations.
\newblock {\em Int. Math. Res. Not. IMRN}, (1):536--587, 2023.

\bibitem[Ran92]{RandeThese}
B.~Rand\'e.
\newblock \'{E}quations fonctionnelles de {M}ahler et applications aux suites
  $p$-r\'eguli\`eres.
\newblock {\em Th\`ese de l'{U}niversit\'e Bordeaux I available at
  https://tel.archives-ouvertes.fr/tel-01183330}, 1992.

\bibitem[Roq18]{Ro15}
J.~Roques.
\newblock On the algebraic relations between {M}ahler functions.
\newblock {\em Trans. Amer. Math. Soc.}, 370(1):321--355, 2018.

\bibitem[Roq21]{RoquesLSMS}
J.~Roques.
\newblock On the local structure of {M}ahler systems.
\newblock {\em Int. Math. Res. Not. IMRN}, (13):9937--9957, 2021.

\bibitem[Roq24]{RoquesFrobForMahler}
J.~Roques.
\newblock Frobenius method for {M}ahler equations.
\newblock {\em J. Math. Soc. Japan}, 76(1):229--268, 2024.

\bibitem[Sau04]{S04}
J.~Sauloy.
\newblock La filtration canonique par les pentes d'un module aux
  {$q$}-diff\'erences et le gradu\'e associ\'e.
\newblock {\em Ann. Inst. Fourier (Grenoble)}, 54(1):181--210, 2004.

\bibitem[Sch03]{Schmidt03}
W.~M. Schmidt.
\newblock Linear recurrence sequences.
\newblock In {\em Diophantine approximation ({C}etraro, 2000)}, volume 1819 of
  {\em Lecture Notes in Math.}, pages 171--247. Springer, Berlin, 2003.

\bibitem[Sie29]{Siegel29}
C.~L. Siegel.
\newblock \"{U}ber einige {A}nwendungen diophantischer {A}pproximationen.
\newblock {\em Abh. Preuß, Akad. Wissen. Phys.-math. Klasse}, 1929.

\bibitem[Sie14]{Siegel29-trad}
C.~L. Siegel.
\newblock \"{U}ber einige {A}nwendungen diophantischer {A}pproximationen
  [reprint of {A}bhandlungen der {P}reu\ss ischen {A}kademie der
  {W}issenschaften. {P}hysikalisch-mathematische {K}lasse 1929, {N}r. 1].
\newblock In {\em On some applications of {D}iophantine approximations},
  volume~2 of {\em Quad./Monogr.}, pages 81--138. Ed. Norm., Pisa, 2014.

\bibitem[SS19]{SchafkeSinger}
R.~Sch\"{a}fke and M.~Singer.
\newblock Consistent systems of linear differential and difference equations.
\newblock {\em J. Eur. Math. Soc. (JEMS)}, 21(9):2751--2792, 2019.

\bibitem[vdPS97]{VdPS97}
M.~van~der Put and M.~F. Singer.
\newblock {\em Galois theory of difference equations}, volume 1666 of {\em
  Lecture Notes in Mathematics}.
\newblock Springer-Verlag, Berlin, 1997.

\bibitem[vdPS03]{VdPS}
M.~van~der Put and M.~F. Singer.
\newblock {\em Galois theory of linear differential equations}, volume 328 of
  {\em Grundlehren der Mathematischen Wissenschaften [Fundamental Principles of
  Mathematical Sciences]}.
\newblock Springer-Verlag, Berlin, 2003.

\bibitem[Wal00]{Wald2000}
M.~Waldschmidt.
\newblock {\em Diophantine approximation on linear algebraic groups}, volume
  326 of {\em Grundlehren der mathematischen Wissenschaften [Fundamental
  Principles of Mathematical Sciences]}.
\newblock Springer-Verlag, Berlin, 2000.
\newblock Transcendence properties of the exponential function in several
  variables.

\end{thebibliography}
\end{document}